\documentclass[12pt]{amsart}

\usepackage{hyperref}
\usepackage{amssymb}
\usepackage[bbgreekl]{mathbbol}
\usepackage{graphicx}
\usepackage{ifthen}
\usepackage{pict2e}
\usepackage{xargs}
\usepackage{xspace}
\usepackage{xcolor}
\usepackage{pgf,tikz}
\usepackage{pgfplots}
\usepackage{environ}
\usepackage{caption}
\usetikzlibrary{arrows}
\usepackage{enumitem}
\usepackage{xifthen}
\usepackage[a4paper,left=2cm,right=2cm,top=1.5cm,bottom=1.5cm, includefoot]{geometry}
\usepackage{multicol}
\usepackage{multirow}
\usepackage{todonotes}


\newcounter{dummy}
\makeatletter
\newcommand\myitem[1][]{\item[#1]\refstepcounter{dummy}\def\@currentlabel{#1}}
\makeatother

\makeatletter
\newsavebox{\measure@tikzpicture}
\NewEnviron{scaletikzpicturetowidth}[1]{%
	\def\tikz@width{#1}%
	\begin{lrbox}{\measure@tikzpicture}%
		\BODY
	\end{lrbox}%
	\pgfmathparse{#1/\wd\measure@tikzpicture}%
	\BODY
}
\makeatother

\DeclareSymbolFontAlphabet{\mathbb}{AMSb}

\newcommand{\thistheoremname}{}
\newtheorem*{genericthm*}{\thistheoremname}
\newenvironment{namedthm*}[1]
{\renewcommand{\thistheoremname}{#1}%
	\begin{genericthm*}}
	{\end{genericthm*}}

\newcommand{\Bairespace}[1][]{
	\ifthenelse{\equal{#1}{}}{\functions{\N}{\N}}{\functions{#1}{\N}}
}
\newcommand{\bbL}{\mathbb{L}}
\newcommand{\bbX}{\mathbb{X}}

\newcommand{\Cantorspace}[1][]{
	\ifthenelse{\equal{#1}{}}{\functions{\N}{2}}{\functions{#1}{2}}
}

\newcommandx{\concatenation}[2][1 = undefined, 2 = undefined]{
	\ifthenelse{\equal{#1}{undefined}}{{}\smallfrown}{
		\ifthenelse{\equal{#2}{undefined}}{\bigoplus #1}{\bigoplus_{#1} #2}
	}
}

\newcommandx{\functions}[3][3 =]{
	\ifthenelse{\equal{#3}{}}{#2^{#1}}{#2_{#3}^{#1}}
}

\newcommand{\Gzero}[1][]{
	\ifthenelse{\equal{#1}{}}
	{\mathbb{G}_0}
	{\mathbb{G}_{0,n}}
}
\newcommandx{\Hzero}[2][2 = undefined]{
	\ifthenelse{\equal{#2}{undefined}}
	{\mathbb{H}_{#1}}
	{\mathbb{H}_{#1, #2}}
}
\newcommandx{\intersection}[2][1 =, 2 =]{
	\ifthenelse{\equal{#1}{}}{\cap}{
		\ifthenelse{\equal{#2}{}}{\bigcap #1}{{\bigcap_{#1} #2}}
	}
}
\newcommand{\Lzero}[1][]{\ifthenelse{\equal{#1}{}}{\bbL_0}{L_{0, #1}}}
\newcommand{\Lzerospace}[1][]{\ifthenelse{\equal{#1}{}}{\bbX_0}{X_{0, #1}}}

\newcommand{\modulo}[1]{\ (\text{mod } 2)}
\newcommand{\N}{\mathbb{N}}

\newcommandx{\product}[2][1 =, 2 =]{
	\ifthenelse{\equal{#1}{}}{\times}{
		\ifthenelse{\equal{#2}{}}{\prod #1}{{\prod_{#1} #2}}
	}
}

\newcommandx{\sequence}[2][2 = undefined]{
	\ifthenelse{\equal{#2}{undefined}}{(#1)}{
		(#1)_{#2}
	}
}

\newcommandx{\set}[2][2 = undefined]{
	\ifthenelse{\equal{#2}{undefined}}{\{ #1 \}}{
		\{ #1 \suchthat #2 \}
	}
}
\newcommandx{\sets}[3][3 =]{
	\ifthenelse{\equal{#3}{}}{[#2]^{#1}}{[#2]^{#1}_{#3}}
}

\newcommand{\suchthat}{\mid}
\newcommand{\R}{\mathbb{R}}

\newcommandx{\union}[2][1 =, 2 =]{
	\ifthenelse{\equal{#1}{}}{\cup}{
		\ifthenelse{\equal{#2}{}}{\bigcup #1}{{\bigcup_{#1} #2}}
	}
}

\newtheorem{theorem}{Theorem}[section]
\newtheorem{lemma}[theorem]{Lemma}

\newtheorem{claim}[theorem]{Claim}
\newtheorem{corollary}[theorem]{Corollary}
\newtheorem{proposition}[theorem]{Proposition}

\newtheorem{question}[theorem]{Question}

\newtheorem{definition}[theorem]{Definition}

\newtheorem{remark}[theorem]{Remark}

\numberwithin{equation}{section}

\newcommand{\eps}{\varepsilon}

\newcommand{\bd}{\begin{definition}}
	\newcommand{\ed}{\end{definition}}

\DeclareMathOperator{\dist}{dist}
\DeclareMathOperator{\didistance}{didist}

\newcommand{\distance}[3]{\ifthenelse{\isempty{#3}}{\dist(#1,#2)}{\dist^{#3}(#1,#2)}}
\newcommand{\didist}[3]{\ifthenelse{\isempty{#3}}{\didistance(#1,#2)}{\didistance^{#3}(#1,#2)}}
\newcommand{\digraph}[3]{\ifthenelse{\equal{#1}{b}}{\mathbb{#2}_{#3}}
	{{#2}_{#3}}}
\newcommand{\linegraph}[3]{\ifthenelse{\equal{#1}{b}}{\mathbb{#2}_{#3}}
	{#2_{#3}}}

\newcommand{\underlyingspace}[3]{\ifthenelse{\equal{#1}{b}}{\mathbb{#2}_{#3}}
	{#2_{#3}}}
\newcommand{\distanceset}[2]{\ifthenelse{\isempty{#2}}{D(#1)}{D^{#2}(#1)}}

\pgfplotsset{soldot/.style={color=blue,only marks,mark=*}}
\usetikzlibrary{math}
\usetikzlibrary{arrows.meta}
\usetikzlibrary{shapes,decorations}
\usetikzlibrary{decorations.pathmorphing}
\usetikzlibrary{math}
\usetikzlibrary{arrows.meta}
\usetikzlibrary{shapes,decorations}
\usetikzlibrary{decorations.pathmorphing}
\usetikzlibrary{calc}
\definecolor{pastelred}{rgb}{1.0, 0.41, 0.38}
\definecolor{pastelblue}{rgb}{0.52, 0.63, 0.94}
\definecolor{pastelyellow}{rgb}{0.99, 0.99, 0.59}
\definecolor{pastelgreen}{rgb}{0.47, 0.87, 0.47}
\definecolor{pastelorange}{rgb}{1.0, 0.7, 0.28}

\usepackage{framed}
\definecolor{shadecolor}{named}{lightgray}

\title{Poisson-Voronoi percolation in higher rank}

\author{Jan Greb\'{i}k}
\address{JG: Universität Leipzig, Mathematisches Institut, D-04009, Leipzig, Germany, E-mail: \href{mailto:grebikj@gmail.com}{grebikj@gmail.com}}

\author{Konstantin Recke}
\address{KR: Universit{\"a}t M{\"u}nster, Orleans-Ring 10, 48149 M{\"u}nster, Germany, E-mail: \href{mailto:konstantin.recke@uni-muenster.de}{konstantin.recke@uni-muenster.de}}

\begin{document}

\maketitle

\newcommand{\group}{\operatorname{SL}_n(\mathbb{R})}
\newcommand{\vol}{\operatorname{vol}}

\begin{quote}{\small {\bf Abstract:} We show that the uniqueness thresholds for Poisson-Voronoi percolation in symmetric spaces of connected higher rank semisimple Lie groups with property~(T) converge to zero in the low-intensity limit. This phenomenon is fundamentally different from situations in which Poisson-Voronoi percolation has previously been studied.

Our approach builds on a recent breakthrough of Fraczyk, Mellick and Wilkens (\url{https://arxiv.org/abs/2307.01194}) and provides an alternative proof strategy for Gaboriau's fixed price problem. As a further application of our result, we give a new class of examples of non-amenable Cayley graphs that admit factor of iid bond percolations with a unique infinite cluster and arbitrarily small expected degree, answering a question inspired by Hutchcroft--Pete ({\em Invent. math.} {\bf 221} (2020)).}
\end{quote}

\section{Introduction}

\emph{Poisson-Voronoi percolation} is a continuum percolation model that can be defined on any metric space $(M,d)$ with an infinite Radon measure $\mu$ as follows.
For $\lambda>0$, consider a Poisson point process of intensity $\lambda\mu$ and associate to each point of the process its Voronoi cell, that is the set of all points in $M$ closer to this point than to any other point of the process. For $p\in(0,1)$, color each cell black with probability $p$ and white with probability $1-p$, independently of the colors of all other cells, and let $\omega_p^{(\lambda)}$ denote the union of black cells.

Poisson-Voronoi percolation has been extensively studied in the probabilistic literature, see e.g.~\cite{Z96,BS98,BS00,BR06,BR06b,DRT19,HM24} and the references therein. In addition, the underlying Poisson-Voronoi tessellation  is a central object studied in stochastic geometry, see e.g.~\cite{SW08,BBK20}. Recently, low-intensity limits of such tessellations on hyperbolic spaces and, more generally, on Riemannian symmetric spaces have emerged as fascinating probabilistic objects with powerful applications \cite{B19, BCP22,DCELU23,FMW23,D24} (see Remark \ref{rm:convergence} for details). In this paper, we build on these works to prove new statements at low, but non-zero, intensity $\lambda>0$.

The quantity of interest will be the {\em uniqueness threshold}
\vspace{1mm}
\begin{equation*}
p_u(\lambda) := \inf \big\{ p\in(0,1) : \mathbb P\big( \mathcal \omega_p^{(\lambda)} \, \, \text{has a unique unbounded cluster} \big)>0 \big\},
\vspace{1mm}
\end{equation*}
where {\em cluster} refers to a path connected component of $\omega_p^{(\lambda)}$. The following is the main result of this paper (see Theorem \ref {thm:vanishing}).

\begin{theorem}[Vanishing uniqueness thresholds] \label{maintheorem}
Let $G$ be a connected higher rank semisimple real Lie group with property (T) and let $(X,d_X)$ be its symmetric space. Then 
    \begin{equation*}
        \lim_{\lambda\to0} \, p_u(\lambda)=0.
    \end{equation*}
\end{theorem}

The behavior in Theorem \ref{maintheorem} is in striking contrast to situations in which Poisson-Voronoi percolation has previously been studied. More precisely, it was shown in a seminal paper by Benjamini and Schramm \cite{BS00} that
\begin{equation*}
\lim_{\lambda\to0} \, p_u(\lambda) = 1
\end{equation*}  
for the hyperbolic plane $\mathbb H^2$ equipped with its volume measure. On the other hand, in the Euclidean plane (and $\R^d$, $d\ge2$), $p_u(\lambda)$ is equal to a constant $p_u\in(0,1)$ irrespective of the intensity, see e.g.~\cite{BR06b}.

The main ideas of the proof of Theorem \ref{maintheorem} are outlined in Section \ref{sec:IdeasProof}. There we describe in particular how our approach builds on the spectacular recent result of Fraczyk, Mellick and Wilkens~\cite{FMW23}. Here, let us highlight only the following important remark.

\begin{remark}[Continuity of $p_u$]\label{rm:Continuity} In the setting of Theorem \ref{maintheorem}, there is a natural candidate tessellation for a low-intensity limit of Poisson-Voronoi tessellations on $X$, namely the {\em ideal Poisson-Voronoi tessellation (IPVT)} in the sense of \cite{FMW23}. In \cite{FMW23}, this object was constructed and shown to have the remarkable property that every pair of cells shares an unbounded boundary. Let us in particular mention the inspiring earlier works~\cite{B19,BCP22,DCELU23} which provide a different treatment of ideal Poisson-Voronoi tessellations on hyperbolic spaces and certain Cayley graphs. 

Given the ideal Poisson-Voronoi tessellation on $X$, let us denote by $p_u(0)$ the uniqueness threshold for percolation. Then $p_u(0)=0$ because percolation with any $p>0$ actually yields a {\em single} cluster. With this notation, Theorem~\ref{maintheorem} entails that $p_u(\lambda)$ is continuous at $\lambda=0$. It may thus be tempting to think of our result as a consequence of continuity of the parameter $p_u$. Let us emphasize that this is not the way we prove Theorem~\ref{maintheorem} (cf.~Remark \ref{rm:convergence}). Setting aside the technical obstacle that convergence of the Voronoi tessellations to the ideal Poisson-Voronoi tessellation is not known, we are, more importantly, not aware of such a continuity result. In particular, $p_u(\lambda)$ is {\em not} continuous at $\lambda=0$ in certain special cases, for instance $\mathbb R^2$, where $p_u(\lambda)=1/2$ for all $\lambda>0$ by \cite{Z96,BR06}, but $p_u(0)=0$. We also note that there is no obvious monotonicity,  of $p_u(\lambda)$ and Poisson-Voronoi percolation more generally, in $\lambda$.
\end{remark}

\subsection{Applications to the sparse FIID unique infinite cluster property} \label{sec:applications}
As an application of our main result, we construct a factor of iid (FIID) \emph{sparse unique infinite cluster} for a certain class of non-amenable countable groups, see Theorem~\ref{main:FIIDSUICP}.

\begin{definition}
    A countable group~$\Gamma$ has the {\em FIID sparse unique infinite cluster property} if there exists a Cayley graph ${\rm Cay}(\Gamma)$ of~$\Gamma$ such that
    \begin{equation} \label{def-FIIDsparseUIC}
    \inf \bigg\{ \int_{\omega\in \mathcal U(G)} {\rm deg}_\omega(1_{\Gamma}) \, d \mu(\omega) \colon \mu\in F_{{\rm IID}}(\Gamma,\mathcal U({\rm Cay}(\Gamma))) \bigg\} = 0,
    \end{equation}
    where $\mathcal U({\rm Cay}(\Gamma))$ is the set of subgraphs of ${\rm Cay}(\Gamma)$ with a unique infinite cluster, $F_{{\rm IID}}(\Gamma,\mathcal S(\Gamma))$ is the set of $\Gamma$-invariant probability measures on $\mathcal U({\rm Cay}(\Gamma))$ which are factors of iid~processes on $\Gamma$ and $1_\Gamma$ denotes the identity of $\Gamma$.
\end{definition}

Prior to the present work, the FIID sparse unique infinite cluster property was known for certain amenable Cayley graphs by combining~\cite{TT13,BRR23}, see also \cite{L13}. It was pointed to us by Hutchcroft that the property extends from any $wq$-normal subgroup to the ambient group, cf.~\cite{HP24} and see also \cite[Theorem 7.18]{FMW23} and \cite[VI.24.(3)]{G00}. In particular, there exist non-amenable, resp.~property~(T), examples such as $\mathbb F_2\times\mathbb Z^3$, resp.~${\rm SL}_3(\mathbb Z)\ltimes \mathbb Z^3$, with this property. Note that these constructions are based on the FIID sparse unique infinite cluster in the underlying amenable building blocks ($\mathbb Z^3$ in the above examples). In contrast, our construction in Theorem~\ref{main:FIIDSUICP} is different and builds on Theorem \ref{maintheorem}.

Beyond constituting a perplexing property of intrinsic interest, the relevance of the FIID sparse unique infinite cluster property is due to the fact that it implies that $\Gamma$ has \emph{fixed price~$1$} (see below).
As an invariant random process and not as FIID, the property was established for all Cayley graphs of groups with property~(T) in groundbreaking work of Hutchcroft and Pete~\cite{HP20}, and allowed them to show that every group with property~(T) has \emph{cost $1$}.
The question about the FIID version was left open, cf.~\cite[Remark 4.4]{HP20}. It was posed explicitly by Pete and Rokob \cite[Question 1.5]{PR25}.

\begin{question}[{\cite[Question 1.5]{PR25}}] \label{q:FIIDSparseUnique}
Give examples of non-amenable Cayley graphs with FIID sparse unique infinite clusters. 
\end{question}

As the promised main application of Theorem \ref{maintheorem}, we give an answer to Question~\ref{q:FIIDSparseUnique}.

\begin{theorem}[Cayley graphs with the FIID sparse unique infinite cluster property] \label{main:FIIDSUICP}
Let $\Gamma\subset G$ be a co-compact lattice in a connected higher rank semisimple real Lie group $G$ with property~(T). Let ${\rm Cay}(\Gamma,S)$ be the Cayley graph of $\Gamma$ with respect to a finite symmetric generating set $S$. Then, for every $\eps>0$, there is a $\Gamma$-equivariant FIID bond percolation $\omega$ on ${\rm Cay}(\Gamma,S)$ with a unique infinite cluster and $\mathbb E\big[ {\rm deg}_{\omega}(1_G) \big]\le\eps$.
\end{theorem}

Let us now discuss in more details the connections with fixed price and cost, and the novel strategy that our approach introduces.
The {\em cost} of a free p.m.p.~action of a countable group is an orbit-equivalence invariant introduced in~\cite{L95}. In particular, it was studied in seminal work of Gaboriau \cite{G00,G02}.  The {\em cost} of a countable group is defined to be the infimal cost of its free, ergodic p.m.p.~actions. The group has {\em fixed price} if all its free, ergodic p.m.p.~actions have the same cost. The following famous question is due to Gaboriau \cite[Question 6.3]{G10}.

\begin{question}[Fixed price problem]\label{q:FixedPrice} Does every countable group have fixed price?
\end{question}

We refer to \cite{KM04,G10,HP20,FMW23} and the references therein for background and instead focus on a recent probabilistic approach due to \cite{HP20}, where it was used to show that groups with property~(T) have cost $1$, answering another well-known question of Gaboriau.

\medskip

{\noindent\bf Factor of iid sparse unique infinite clusters.} The maximal cost over all free ergodic p.m.p. actions of $\Gamma$ is given by the following probabilistic formula:
\begin{equation} \label{def:maxcost}
{\rm cost}^*(\Gamma) := \frac{1}{2} \inf \bigg\{ \int_{\omega\in\mathcal S(\Gamma)} \deg_\omega(1_{\Gamma}) \, d \mu(\omega) \colon \mu\in F_{{\rm IID}}(\Gamma,\mathcal S(\Gamma)) \bigg\},
\end{equation}
where $\mathcal S(\Gamma)$ denotes the set of connected spanning graphs on $\Gamma$,  $F_{{\rm IID}}(\Gamma,\mathcal S(\Gamma))$ is the set of $\Gamma$-invariant probability measures on $\mathcal S(\Gamma)$ which are factors of iid~processes on $\Gamma$ and $1_\Gamma$ denotes the identity of $\Gamma$, see \cite{HP20}. This formula is based on a similar representation of the cost of $\Gamma$, see \cite[Proposition 29.5]{KM04}, and the fact that Bernoulli actions have maximal cost among the free ergodic p.m.p.~actions by a result of Ab{é}rt and Weiss \cite{AW13}. The following reduction step was observed in \cite{HP20}:
\begin{equation} \label{equ:HPReduction}
{\rm cost}^*(\Gamma) \leq 1 + \frac{1}{2} \inf \bigg\{ \int_{\omega\in\mathcal U(\Gamma)} {\rm deg}_\omega(1_\Gamma) \, d \mu(\omega) : \mu\in F_{{\rm IID}}(\Gamma,\mathcal U(\Gamma)) \bigg\},
\end{equation}
where $\mathcal U(\Gamma)\subset\{0,1\}^{\Gamma\times\Gamma}$ is the set of graphs on $\Gamma$ with a unique infinite cluster. In particular, the FIID sparse unique infinite cluster property defined in \eqref{def-FIIDsparseUIC} implies that $\Gamma$ has fixed price~$1$.

In \cite{HP20}, a similar reduction was applied to study the cost of $\Gamma$. Using an ingenious construction \cite[Section 2.2]{HP20}, it was shown there that Cayley graphs of countable groups with property (T) have the {\em sparse unique infinite cluster property} defined as in~\eqref{def-FIIDsparseUIC} using $\Gamma$-invariant instead of FIID processes, which established that these groups have cost~$1$. This construction is very far from being a factor of iid, see \cite[Remark 4.4]{HP20}, and is hence quite different from the construction used in our proof of Theorem~\ref{main:FIIDSUICP}.

\begin{remark} Let us discuss an instructive link between Poisson-Voronoi percolation and the FIID sparse unique infinite cluster property. Note that we may similarly consider Poisson-Voronoi percolation (or Bernoulli-Voronoi percolation) on Cayley graphs. Denote the corresponding uniqueness threshold by $p_u(\lambda)$. An important observation is that every Cayley graph such that $p_u(\lambda)\to0$ as $\lambda\to0$ has the FIID sparse unique infinite cluster property. Indeed, let $\eps>0$, then by assumption there exists $\lambda_0>0$ with $p_u(\lambda_0)<\eps/d$, where $d$ is the degree in the graph. By keeping all edges between black vertices, we obtain an FIID bond percolation with expected degree at most $\eps$ and a unique infinite cluster. We emphasize that our proof of Theorem \ref{main:FIIDSUICP} is not based on this observation, i.e.~vanishing uniqueness thresholds for the discrete model, but instead uses the continuum percolation model on the associated symmetric space to construct the desired FIID processes.
\end{remark}

We refer to Section \ref{sec:closing} for a discussion of the discrete case including several open questions.

\medskip

{\bf\noindent Sparse factor graphs of Poisson point processes with a unique infinite cluster.} The notion of cost, and the fixed price problem, were extended to unimodular locally compact second countable groups via connected (equivariant) factor graphs of free invariant point processes on the group $G$ by {Á}bert and Mellick \cite{AM22}. For our purposes, it suffices to recall that Poisson point processes then have maximal cost among all free invariant point processes \cite[Theorem 1.2]{AM22}. Thus
\begin{equation*}
{\rm cost}^*(G) = \frac{1}{2} \inf \Big\{ \mathbb E\big[ {\rm deg}_{\mathcal G(Y_0)}(1_G) \big] : \mathcal G \, \, \text{connected factor graph of} \, \, Y  \Big\},
\end{equation*}
where $Y$ is a Poisson point process of intensity $1$ on $G$ and $Y_0:=Y\cup\{1_G\}$ \cite[Definition 4.1]{AM22}. We may also assume that $Y$ is equipped with iid ${\rm Unif}[0,1]$ marks, see \cite[Theorem 1.7 \& 1.8]{AM22}.

In the setting of a non-compact connected semisimple real Lie group $G$ acting on its symmetric space $X$, it is possible to transfer the picture from $X$ to $G$. This perspective was used in \cite[Theorem~C]{FMW23} to prove that higher rank semisimple real Lie groups have fixed price~$1$. Assuming that $G$ has property~(T), fixed price $1$ alternatively follows from a reduction analogous to \eqref{equ:HPReduction} and the following application of Theorem \ref{maintheorem} (see Corollary \ref{cor:FIIDsparseUSectionEight}).

\begin{corollary}[FIID sparse unique infinite clusters] \label{cor:FIIDsparseU} Let $G$ be a connected higher rank semisimple real Lie group with property (T). Let~$\Pi$ be the Poisson point process on $G$ of intensity $1$ equipped with iid ${\rm Unif}[0,1]$ marks. Then, for every $\eps>0$, there is a $G$-equivariant factor graph $\mathcal H$ of $\Pi$ with a unique infinite cluster and $\mathbb E\big[ {\rm deg}_{\mathcal H(\Pi_0)}(1_G) \big]\le\eps$ for $\Pi_0=\Pi\cup\{1_G\}$.
\end{corollary}

Let us conclude this section with the following overview.
\begin{itemize}
\item[(i)] We present a new strategy for proving fixed price $1$ and, in fact, the FIID sparse unique infinite cluster property based on the phenomenon ''$p_u(\lambda)\to0$ as $\lambda\to0$'' for the well-known Poisson-Voronoi percolation model.
\item[(ii)] To establish this phenomenon, we need to use both the infinite touching result of \cite{FMW23} and property (T) (which, for instance, is well-known for connected higher rank simple real Lie groups  with finite center~\cite{BHV08}).
\item[(iii)] This approach has the following main advantages:
\begin{itemize}
    \item[1.] It uses the natural FIID process $\omega_p^{(\lambda)}$. In this way, we avoid passing to a limit. In particular, cells are compact and thus trivially hyperfinite. This avoids certain technicalities needed in the case of a limiting tessellation, cf.~\cite[Section 7.1]{FMW23}.
    \item[2.] It establishes the phenomenon $p_u(\lambda)\to0$ as $\lambda\to0$ as a viable strategy for proving the FIID sparse unique infinite cluster property and fixed price~$1$.
\end{itemize}
\end{itemize}

\subsection{Strategy of proof of the main result, Theorem~\ref{maintheorem}}\label{sec:IdeasProof}

The main difficulty in estimating $p_u(\lambda)$ is that the existence of a {\em unique} unbounded cluster is a {\em non-local} phenomenon. Note that we are in a non-amenable setting and thus the co-existence of infinitely many unbounded clusters is possible and, in fact, expected. In particular, estimating $p_u(\lambda)$ via the critical probability $p_c(\lambda)$ does not seem possible. Hence the approaches to estimating $p_u$ are essentially limited to establishing versions of long-range order, see e.g.~\cite[Theorem 7.50]{LP16}. This is in fact the starting point of our analysis.

\medskip

{\noindent\bf The long-range order approach.} Our starting point is the following characterization of $p_u$ (see Theorem \ref{thm:LRO}). Note that this result does not assume higher rank or property~(T).

\begin{theorem}[Long-range order implies uniqueness] \label{main:LRO}
    Let $G$ be a non-compact connected semisimple real Lie group and $(X,d_X)$ be its symmetric space. Then 
    \begin{equation*}
        p_u(\lambda) = \inf \Big\{ p\in(0,1) \, : \, \inf_{x,y\in X} \mathbb P\Big( x \overset{\omega_p^{(\lambda)}}{\longleftrightarrow} y \Big)>0 \Big\}.
    \end{equation*}
\end{theorem}

Theorem \ref{main:LRO} is expected as versions of it are well-known in percolation theory \cite{LS99,AL07} and go back to seminal work of Lyons and Schramm \cite{LS99}. Since we could not find a suitable reference, we include the proof. This proof is longer than one might expect, which is due to a technical subtlety in applying the method of \cite{LS99}. We refer to Section~\ref{sec:LRO} for a detailed discussion.

In order to use Theorem \ref{main:LRO}, we have to verify a uniform lower bound on the two-point function. We deduce such a lower bound by combining two ingredients.

\medskip

{\noindent\bf Higher rank and property (T).} The first ingredient is a {\em finitary} (meaning that $\lambda>0$ is fixed and sufficiently small, instead of passing to the low intensity limit) analogue of the {\em infinite touching} phenomenon which was obtained for the ``ideal Poisson-Voronoi tessellation'' in higher rank in \cite{FMW23} (see Theorem \ref{thm:A2}).

\begin{theorem}[Intersection of cells at small intensities] \label{main:FinitaryTouching}
    Let $G$ be a connected higher rank semisimple real Lie group and $(X,d_X)$ be its symmetric space.
    Then for every $\eps>0$ and $R>0$ there is $\lambda_0\in (0,1]$ such that 
    \begin{equation*}
        \inf_{0<\lambda\le \lambda_0} \mathbb{P} \Big( \text{all cells at intensity } \lambda, \text{ which intersect } \mathcal B_R(o), \text{pairwise share a boundary}  \Big)>1-\eps.
    \end{equation*}
\end{theorem}

\begin{remark}
     Note that Theorem \ref{main:FinitaryTouching} does not assume property (T). While property (T) is well-known in many examples, most notably for all connected higher rank simple real Lie groups with finite center~\cite{BHV08}, there are connected higher rank semisimple real Lie groups which fail to have it. One such example is the isometry group of $\mathbb H^2\times\mathbb H^2$ endowed with the $L^2$-metric and its natural Riemannian structure; see Section~\ref{sec:closing} for more on this particular example.
\end{remark}

Theorem \ref{main:FinitaryTouching} is derived from a (formally) weaker version of a result in \cite{FMW23} which we recall in Theorem~\ref{thm:InfiniteTouching} and Remark \ref{rem:InfiniteTouching}. The proof of the result suitable for our purposes is again longer than one might expect. This is due to the fact that convergence of the Voronoi tessellations to the object appearing in Theorem \ref{thm:InfiniteTouching} is proven neither in \cite{FMW23} nor here. Instead, we provide the details to derive Theorem \ref{main:FinitaryTouching} directly from Theorem \ref{thm:InfiniteTouching}, see Section \ref{sec:FinitaryConditionsProofs} for details.

The second ingredient is the following criterion for long-range order which holds for $G$-invariant normal (this technical notion is defined in Section~\ref{sec:threshold}) random closed subsets in the presence of a group action by a group with property (T) (see Theorem \ref{thm:threshold}). This theorem does not use the assumption of higher rank.

\begin{theorem}[Long-range order threshold] \label{main:threshold}
    Let $G$ be a lcsc group acting continuously and transitively by isometries on a geodesic lcsc metric space $(M,d)$ and fix some $o\in M$. Suppose that $G$ has property (T). Then for every $R>0$, there exists $p^*<1$ such that every $G$-invariant normal random closed subset $\mathcal Z$ of $M$ with $\mathbb P(o \in \partial \mathcal Z)=0$ satisfies
    \begin{equation*}
        \inf_{x,y\in \mathcal B_R(o)} \mathbb P \big( x\overset{\mathcal Z}{\longleftrightarrow} y \big)>p^* \quad \Rightarrow \quad \inf_{x,y\in M} \mathbb P \big( x\overset{\mathcal Z}{\longleftrightarrow} y \big)>0.
    \end{equation*}
\end{theorem}

Theorem \ref{main:threshold} is inspired by a similar result about group-invariant percolation on Cayley graphs \cite{LS99, MR23} and can be proven along similar lines. We refer to Section \ref{sec:threshold} for details. We also point out that this result and its application in this paper fit into a broader recent approach to the interplay between group-invariant percolation and geometric properties of groups developed by Mukherjee and the second author in \cite{MR22,MR23,MR24}.

\medskip

{\noindent\bf Proof of long-range order.} Theorem \ref{main:threshold} allows to establish the global phenomenon of long-range order, from the local information about large density. This is precisely how we use property (T) in our proof. However, it is clearly not possible to apply this result directly to Poisson-Voronoi percolation with parameters $(\lambda,p)$ for $p$ arbitrarily small because this model has arbitrarily small, instead of large, density.

Instead, we use a coupling argument, which is inspired by the proof of non-uniqueness at the uniqueness threshold of Bernoulli percolation on groups with property (T) in~\cite[Theorem 4.4]{MR23}. The idea of this coupling argument is to define auxiliary percolations which artificially increase the density of the percolation under consideration while keeping the number of unbounded clusters constant. For this step, we use the non-triviality of limit points of Poisson-Voronoi tessellations (see Theorem~\ref{thm:A1}) and use the higher rank assumption in the form of Theorem~\ref{main:FinitaryTouching}. The details are provided in Section~\ref{sec:vanishing}.

\subsection{Organization} The rest of this paper is organized as follows. In Section~\ref{sec:Preliminaries}, we recall the relevant background about symmetric spaces and Poisson-Voronoi percolation. The aforementioned properties of Poisson-Voronoi tessellations at low intensities (including Theorem~\ref{main:FinitaryTouching}) are stated in Section~\ref{sec:FinitaryConditions} and proved in Section~\ref{sec:FinitaryConditionsProofs}. Section~\ref{sec:threshold} contains the statement and proof of the  long-range order criterion (Theorem~\ref{main:threshold}). In Section~\ref{sec:LRO}, we prove the characterization of the uniqueness phase in terms of long-range order (Theorem~\ref{main:LRO}). The proof of our main result (Theorem~\ref{maintheorem}) is given in Section~\ref{sec:vanishing}.
The main two applications Corollary~\ref{cor:FIIDsparseU} and Theorem~\ref{main:FIIDSUICP} are proved in Section~\ref{sec:sparse} and Section \ref{sec:Lattices} respectively. In Section~\ref{sec:closing}, we pose open questions raised by our results.

\medskip

{\noindent\bf Acknowledgments.} We thank Mikołaj Fraczyk, Tom Hutchcroft, Chiranjib Mukherjee, Gabor Pete and Amanda Wilkens for valuable discussions and comments. We also thank Itai Benjamini, Matteo d'Achille, Tim de Laat, Russ Lyons, Sam Mellick and Elliot Paquette for comments on a preliminary version of this paper.
JG was supported by MSCA Postdoctoral Fellowships 2022 HORIZON-MSCA-2022-PF-01-01 project BORCA grant agreement number 101105722.
The research of KR is funded by the Deutsche Forschungsgemeinschaft (DFG) under Germany's Excellence Strategy EXC 2044-390685587, Mathematics M\"unster: Dynamics-Geometry-Structure.

\section{Poisson-Voronoi percolation on symmetric spaces} \label{sec:Preliminaries}

In this section, we first introduce the necessary background regarding symmetric spaces and recall some fundamental properties important for our analysis, see Section \ref{section symmetric spaces}. We then define the Poisson-Voronoi percolation model, see Section \ref{section PVP}.

\subsection{Lie groups and symmetric spaces} \label{section symmetric spaces}

We introduce the background following \cite{FMW23}, to which we refer for more details.

Let $G$ be a \emph{non-compact connected semisimple real Lie group} and $X$ be its \emph{symmetric space}, our main reference here is \cite[Chapter VI]{helgason2024differential}, in particular, both $G$ and $X$ are locally compact second countable (lcsc) and $X=G/K$, where $K$ is a compact subgroup.
We write $p:G\to G/K$ for the canonical projection that sends $g\mapsto p(g)=gK$.
We denote by $m_G$ the left-invariant Haar measure on $G$ and by $\vol$ the push-forward of $m_G$ via the projection $p$.
The canonical left action of $G\curvearrowright G/K$ defined as $(g,hK)\mapsto ghK$ then preserves $\vol$.
We endow the quotient space $X = G/K$ with the canonical left $G$-invariant Riemannian metric induced by the Killing form on the Lie algebra of $G$. We denote the metric on $X$ induced by this Riemannian metric by $d_X$. This Riemannian metric induces the $G$-invariant volume measure ${\rm vol}$ defined above. In particular, ${\rm vol}$ is additionally invariant under all Riemannian isometries of $X$, equivalently isometries of $(X,d_X)$.
We denote the group of isometries of $(X,d_X)$ as ${\rm Isom}(X)$.
We also fix the base point $o=K\in X$.

\begin{remark}
For studying questions about the uniqueness threshold of Poisson-Voronoi percolation, the assumptions that $G$ is connected and non-compact are clearly appropriate.
For our main result, we will need to additionally assume that $G$ has \emph{higher rank} and \emph{property (T)}.
Well-known examples are provided by $G=\group$ for $n\ge 3$ and $X=\group/\operatorname{SO}(n)$. Recall that the {\em rank} of a real Lie group can be defined as the dimension of a maximal \emph{flat} in its symmetric space~$X$, i.e. the maximal dimension of an isometrically embedded Euclidean space in~$X$.
It follows that a semisimple real Lie group has rank $0$ if and only if it is compact -- we will thus omit the non-compactness assumption whenever we require $G$ to have {\em higher} rank, meaning that its rank is greater than or equal to $2$. We also remark that if $G$ is additionally simply connected, then it is a direct product of simple real Lie groups.
Hence it has higher rank and property (T) if and only if the sum of ranks of the factors is greater than or equal to $2$ and each factor has property (T).
\end{remark}

\medskip

Let $(M,d)$ be a metric space.
For $x\in M$, $r\ge 0$ and $A\subseteq M$, we define $\mathcal{B}_r(x)$ to be the open ball of radius $r$ around $x$, $\overline{\mathcal{B}}_r(x)$ to be the closed ball of radius $r$ around $x$ and $d(x,A)=\inf_{a\in A}d(x,a)$. 
Recall that a \emph{geodesic} from $x\in M$ to $y\in M$ is a map $\gamma:[0,\ell]\to M$ such that $\gamma(x)=0$, $\gamma(\ell)=y$ and $d(\gamma(t),\gamma(t'))=|t-t'|$ for every $t,t'\in [0,\ell]$.
In particular, $\ell=d(x,y)$.
We say that $(M,d)$ is a \emph{geodesic metric space} if every two points in $X$ are joined by a geodesic. We say that $(M,d)$ is {\em proper} if closed and bounded subsets of $M$ are compact.

We collect all the properties we need regarding symmetric spaces in the following theorem.
Recall that a continuous action of a lcsc group $G$ on a lcsc metric space $X$ is \emph{proper} if the subset $\{g\in G:g\cdot K\cap K\not=\emptyset\}$ of $G$ is compact for every compact $K\subset X$.

\begin{theorem}\label{thm:BasicSymmetricSpace}
    Let $G$ be a non-compact connected semisimple real Lie group and $X$ be its symmetric space endowed with the canonical $G$-invariant metric $d_X$, $G$-invariant measure $\vol$ and the base point $o\in X$.
    Then 
    \begin{enumerate}
        \item[{\rm(1)}] $G$ is non-amenable and unimodular. The canonical action of $G$ on $X$ is continuous, proper, transitive and by measure-preserving isometries. Moreover, the measure ${\rm vol}$ is invariant under all isometries of $X$, 
        \item[{\rm(2)}] $(X,d_X)$ is a proper geodesic metric space,
        \item[{\rm(3)}] $\vol(C)<\infty$ for every compact or $d_X$-bounded set $C\subseteq X$ and $\vol(X)=\infty$,
        \item[{\rm(4)}] $\vol(\{x\in X:d_X(o,x)=r\})=0$ for every $r\in [0,\infty)$,
        \item[{\rm(5)}] the map $t\mapsto \vol(\mathcal{B}_t(o))$ is a continuous bijection from $[0,\infty)$ to $[0,\infty)$,
        \item[{\rm(6)}] there exist $a,b,c>0$ such that 
        \begin{equation*}
            \vol(\mathcal{B}_t(o))=ce^{at}t^{b}(1+o(1)),
        \end{equation*} 
        where the quantity $o(1)$ vanishes as $t\to \infty$.
    \end{enumerate}
\end{theorem}
\begin{proof} {\noindent\em Proof of (1):} The first part follows from the classical fact that a connected semisimple Lie group is non-compact if and only if it is non-amenable, see e.g.~\cite[Theorem 3.8]{P88}. Unimodularity follows from the more general fact that a connected semisimple Lie group does not admit any non-trivial continuous homomorphism into an abelian group, see e.g.~\cite[Remark A3.8]{M15}.
The properties of the action $G\curvearrowright X$ are well-known.
The fact that $\vol$ is invariant under all isometries of $X$ follows from the construction, as $\vol$ can be defined from the Riemannian metric $d_X$.

    \medskip
    
    {\noindent\em Proof of (2) and (3)}: It is a well-known fact that $(X,d_X)$ is a complete, connected Riemannian manifold, see e.g.~\cite[Part IV]{M63}, hence (2) follows from the Hopf-Rinow theorem and (3) follows from standard properties of Haar measure on $G$.
    
    \medskip
    
    {\noindent\em Proof of (4):} See \cite[Proposition 2.4.6]{gangolli2012harmonic}.

    \medskip
    
    {\noindent\em Proof of (5):} Since $m_G(K)=0$ and $m_G(G)=\infty$  by (3) and (4), it suffices to show that $\alpha$ is continuous.
    This follows from writing $\mathcal{B}_t(o)=\bigcup_{s< t}\mathcal{B}_s(o)=\bigcap_{s>t}\mathcal{B}_s(o)\setminus \{x\in X:d_X(o,x)=t\}$ for $t\in[0,\infty)$ and using (4).

    \medskip
    
    {\noindent\em Proof of (6):} See \cite[Lemma 4.4]{FMW23}. 
\end{proof}

\subsection{The Poisson-Voronoi percolation model} \label{section PVP}

Let $(M,d,o,\mu)$ be a proper geodesic metric space with some fixed origin $o\in M$ and infinite Radon measure $\mu$ such that the spheres centered at the origin, i.e., the sets of the form $\{x\in X:d(o,x)=r\}$ for $r\ge 0$, have $\mu$-measure zero. The continuum percolation model we now define has two parameters, an {\em intensity} $\lambda>0$ and {\em survival probability} $p\in(0,1]$. Let
\begin{equation}
\mathbf Y^{(\lambda)}=\big\{(Y_1^{(\lambda)},Z_1^{(\lambda)}),(Y_2^{(\lambda)},Z_2^{(\lambda)}),\ldots\big\}
\end{equation}
be such that 
\begin{itemize}
\item[--] the sequence $Y^{(\lambda)}:=\big\{Y_1^{(\lambda)},Y_2^{(\lambda)},\ldots\big\}$ is a Poisson point process on $M$ with intensity $\lambda \cdot \mu$ ordered according to increasing distance from the origin.
\item[--] the sequence $Z^{(\lambda)}:=\{Z_1^{(\lambda)},Z_2^{(\lambda)},\ldots\}$ consists of iid uniform $[0,1]$-labels and is independent of $Y^{(\lambda)}$.
\end{itemize}
The associated {\em Voronoi diagram} is defined to be 
\begin{equation} \label{equ-Voronoi}
\mathrm{Vor}(\mathbf{Y}^{(\lambda)})=\big\{C_1^{(\lambda)},C_2^{(\lambda)},\ldots\big\},
\end{equation}
where $C_i^{(\lambda)}\subset V$ consists of all points $x\in M$ for which $d(x,Y^{(\lambda)})=d(x,Y^{(\lambda)}_i)$. We refer to the elements of this collection as {\em (Voronoi) cells}. Let
\begin{equation}
\mathrm{Vor}(\mathbf{Y}^{(\lambda)})_p = \big\{ B_1^{(\lambda)},B_2^{(\lambda)},\ldots \big\}
\end{equation}
be obtained from $\mathrm{Vor}(\mathbf{Y}^{(\lambda)})$ by independently keeping or deleting each cell with retention probability $p$ according to whether $0\le Z_i^{(\lambda)}\le p$ or not. We interpret this procedure as an independent black-and-white coloring followed by retaining the black cells $B_i^{(\lambda)}$. Finally, let 
\begin{equation}
\omega^{(\lambda)}_{p} = \bigcup_{i=1,2,\ldots} B_i^{(\lambda)}
\end{equation}
denote the random closed (see Lemma \ref{lm:closed} below) subset of $M$ that consists of all points which belong to black cells. We refer to this continuum percolation model as $(\lambda,p)$-{\em Poisson-Voronoi percolation} or simply as {\em Poisson-Voronoi percolation} on $M$. We will refer to the path connected components of $\omega^{(\lambda)}_p$ as {\em clusters}.

\begin{lemma} \label{lm:closed}
    Poisson-Voronoi percolation $\omega_p^{(\lambda)}$ with parameters $\lambda>0$ and $p\in(0,1)$ defines a random closed set.
    Moreover, the following hold: 
    \begin{itemize}
        \item[{\rm(i)}] every cell is path connected and closed a.s. 
        \item[{\rm(ii)}] every cluster is closed a.s.
        \item[{\rm(iii)}] every bounded subset of $M$ is covered by finitely many cells a.s.
        \item[{\rm(iv)}] every bounded subset of $M$ is intersected by finitely many clusters a.s.
    \end{itemize}
\end{lemma}
\begin{proof}
    The assumption that $\mu$ is a Radon measure, hence locally finite, implies that every $x\in M$ belongs to some cell and as a consequence, every cell is closed a.s.
    Every cell is also path connected a.s.~because it contains the geodesic from its nucleus to each of its members.
    This shows (i). Local finiteness of $\mu$ implies that $\mu$ is finite on compact subsets of $M$. Using the fact that $M$ is proper, it is not difficult to see that every ball is split into finitely many cells a.s. In particular, every bounded subset of $M$ is covered by finitely many cells a.s., which shows (iii). Since every cell is path-connected, it also follows that every bounded subset of $M$ intersects finitely many clusters a.s., i.e.~item (iv). 
    Now, if $x_n\in \omega^{(\lambda)}_p$ and $x\in M$ with $x_n\to x$, then, by (iii), there is a subsequence $(x_{n_k})_{k=1}^\infty$ which lies inside some fixed black cell $B^{(\lambda)}_i$. Since every cell is closed, it follows that $x\in B_i^{(\lambda)}\subset\omega^{(\lambda)}_p$, i.e.~$\omega_p^{(\lambda)}$ is a closed set a.s. The same argument together with the fact that every cell is path-connected shows that every cluster is a closed set a.s., i.e.~item (ii). 
    
    Let us also briefly comment on measurability: It is a standard fact that $\big\{Y_1^{(\lambda)},Y_2^{(\lambda)},\ldots\big\}$ together with the iid labels $Z^{(\lambda)}$ is a well-defined random closed set with marks.
    It is then routine to check that, for every $i\in \mathbb{N}$, the assignment $Y^{(\lambda)}_i\mapsto C^{(\lambda)}_i$ is measurable, see for instance \cite[Example 9.2.5]{BBK20} for the case of $\mathbb R^d$.
    Consequently, $\omega^{(\lambda)}_{p}$, as the countable union of the cells, is a well-defined random closed set.
    Similarly, the cells and clusters are well-defined random closed sets.
\end{proof}

Note that Lemma \ref{lm:closed} entails that we may consider the clusters as random closed sets. To lighten notation, we will denote by $\mathbb P_p^{(\lambda)}$, resp.~$\mathbb P^{(\lambda)}$, the law of $\omega_p^{(\lambda)}$, resp.~$\mathbf Y^{(\lambda)}$.

\medskip

{\noindent\bf Phase transition and the uniqueness threshold.} Poisson-Voronoi percolation typically undergoes a {\em phase transition} where the geometry of clusters changes drastically if $\lambda$ is fixed and $p$ varies through a critical value.

Fix $\lambda>0$. The {\em critical probability} is 
\begin{equation} \label{equ-pc}
p_c := p_c(\lambda) := \inf \big\{p\in(0,1) : \mathbb P\big( \omega_p^{(\lambda)} \, \, \text{has an unbounded cluster} \big)>0 \big\}.
\end{equation}
The {\em uniqueness threshold} is
\begin{equation} \label{equ-pu}
p_u := p_u(\lambda) := \inf \big\{p\in(0,1) : \mathbb P\big( \omega_p^{(\lambda)} \, \, \text{has a unique unbounded cluster} \big)>0 \big\}.
\end{equation}
In the setting of our main results, and more generally when there is a non-compact group acting continuously and properly on the metric measure space $(M,d,\mu)$, ergodicity (see Lemma \ref{lm:Ergodicity} below) implies that one may equivalently require in \eqref{equ-pc}, resp.~\eqref{equ-pu}, that the probabilities are equal to $1$.

\medskip

{\noindent\bf Basic properties.} We recall two useful properties of Poisson-Voronoi percolation, ergodicity and the Harris-FKG-Inequality.

\begin{lemma}[Ergodicity] \label{lm:Ergodicity}
    In the above settting, assume that $G$ is a non-compact lcsc group acting continuously and properly by measure-preserving isometries on $(M,d,\mu)$. Then $\omega_p^{(\lambda)}$ is $G$-invariant and ergodic.
\end{lemma}

\begin{proof}[Proof sketch]
    The assumptions guarantee that the marked Poisson point process $\mathbf Y^{(\lambda)}$ is $G$-invariant and mixing, hence ergodic. Therefore $\omega_p^{(\lambda)}$,  as a $G$-equivariant factor of $\mathbf Y^{(\lambda)}$, is $G$-invariant and ergodic.
\end{proof}

Let $\Omega$ denote the set of configurations of  $\omega_p^{(\lambda)}$. An event $A\subset\Omega$ is called {\em increasing}, if it is preserved under adding black points and under erasing white points, see \cite[Chapter 8]{BR06}.

\begin{lemma}[Harris--FKG--Inequality] \label{lm:FKG}
    In the above setting, $\omega_p^{(\lambda)}$ satisfies the Harris-FKG-Inequality, i.e.~
    \begin{equation*}
        \mathbb P_p^{(\lambda)} (A_1 \cap A_2) \geq \mathbb P_p^{(\lambda)} (A_1) \mathbb P_p^{(\lambda)} (A_2)
    \end{equation*}
    for every increasing events $A_1$ and $A_2$.
\end{lemma}

\begin{proof}[Proof]
This is well-known, see e.g.~the general result \cite[Theorem 1.4]{LP11}. See also \cite{MR96,BR06}.
\end{proof}

{\noindent\bf Poisson-Voronoi percolation on symmetric spaces.} We will be interested in the setting where $(M,d,o,\mu)=(X,d_X,o,\vol)$ is a symmetric space of a non-compact connected semisimple real Lie group $G$ equipped with the canonical metric $d_X$, volume measure and origin.

For this setting, we collect the basic properties needed later in the following lemma. Let $\partial Z := Z\setminus {\rm interior}(Z)$ denote the boundary of a closed set $Z$.

\begin{lemma}[Properties of Poisson-Voronoi percolation]\label{lm:BasicVoronoi}
    Let $G$ be a non-compact connected semisimple real Lie group, $(X,d_X)$ be its symmetric space, $\lambda>0$ and $p\in (0,1)$.
    Then $\omega_p^{(\lambda)}$ defines a random closed subset of $X$ which
    \begin{itemize}
        \item[{\rm (i)}] is $G$-invariant and ergodic, and, moreover, is invariant under all isometries of $(X,d_X)$.
        \item[{\rm (ii)}] satisfies the Harris-FKG-inequality.
        \item[{\rm (iii)}] satisfies $\mathbb P\big(o\in \partial \omega_p^{(\lambda)}\big)=0$.
    \end{itemize}
    Moreover, the Voronoi diagram $\mathrm{Vor}(\mathbf{Y}^{(\lambda)})$ consists of compact subsets of $X$.
\end{lemma}

\begin{proof}
    The model is well-defined by Lemma~\ref{lm:closed} together with Theorem~\ref{thm:BasicSymmetricSpace} (2), (3) and (4). Item (i) follows from Lemma \ref{lm:Ergodicity} and Theorem \ref{thm:BasicSymmetricSpace} (1). Item (ii) was proved in Lemma~\ref{lm:FKG}.

\medskip

    {\noindent\em Proof of {\rm (iii)}.} This follows from a standard argument, which we include for the convenience of the reader. By Theorem~\ref{thm:BasicSymmetricSpace} (4),  
    \begin{equation} \label{equ:vol sphere}
        \vol(\{x\in X:d_X(o,x)=r\})=0
    \end{equation} 
    for every $r\in [0,\infty)$. Now if $o\in \partial \omega_p^{(\lambda)}$, then there exist $i\neq j$ with $d(o,Y_i^{(\lambda)})=d(o,Y_j^{(\lambda)})$. The claim follows because the latter event has probability zero by \eqref{equ:vol sphere} and the multivariate Mecke equation, see e.g.~\cite[Theorem 4.4]{LP17}. 
    
\medskip

    {\noindent\em Proof of compactness.}
    By Lemma \ref{lm:closed}, it suffices to show that every cell is bounded. By ergodicity, $G$-invariance and a routine application of the Mecke equation \cite[Theorem 4.1]{LP17}, it is enough to show that if we insert the origin to $\mathbf Y^{(\lambda)}$, that is, if we consider
    $$\{Y^{(\lambda)}_1,Y^{(\lambda)}_2,\dots\}\cup \{o\},$$
    then the cell $C^{(\lambda)}_o$ of the origin is bounded.

    Set $f(t)=\vol(\mathcal{B}_t(o))$ for every $t\in [0,\infty)$.
    Let $r>0$ and $y\in X$ be such that $d_X(o,y)=r$.
    Then $y\not \in C^{(\lambda)}_o$ whenever there is $i\in \mathbb{N}$ and $z\in X$ such that $d_X(o,z)=r$ and $\{y,Y^{(\lambda)}_i\}\subseteq \mathcal{B}_{r/2}(z)$.
    Since $Y^{(\lambda)}$ is a Poisson point process with intensity $\lambda{\rm vol}$, we have
    $$\mathbb{P} \left(Y^{(\lambda)}\cap \mathcal{B}_{r/2}(z)=\emptyset\right)= e^{-\lambda f(r/2)}.$$
    A standard packing argument guarantees that there is a set $Z$ of size at most $\frac{f(3r/2)}{f(r/4)}$ such that 
    $$Z\subseteq \{y\in X:d_X(x,o)=r\}\subseteq \bigcup _{z\in Z} \mathcal{B}_{r/2}(z).$$
    Consequently, the union bound gives that 
    $$\mathbb{P}\Big(C^{(\lambda)}_o\not \subseteq \mathcal B_R(o) \Big)\le \frac{f(3r/2)}{f(r/4)}e^{-\lambda f(r/2)},$$
    which goes to $0$ as $r\to\infty$ by Theorem~\ref{thm:BasicSymmetricSpace}~(6).
\end{proof}

\begin{remark}[Boundary volume]\label{rem:BoundaryVolume}
    Recall that hyperplanes have zero volume, i.e.
    \begin{equation*} 
    {\rm vol} \big( \big\{ x\in X \colon d(x,v)=d(x,w) \big\}\big) = 0,
    \end{equation*}
    for every pair $v\neq w \in X$, see e.g.~\cite[Section 3.3]{AM22}. Since the boundary of Poisson-Voronoi percolation is contained in a countable union of such sets, we have that $\vol(\partial \omega_p^{(\lambda)})=0$ a.s.
\end{remark}

Let us also observe the following more general lemma (which, when combined with Remark~\ref{rem:BoundaryVolume}, gives an alternative proof of Lemma~\ref{lm:BasicVoronoi}~(iii)).

\begin{lemma}\label{lm:boundary}
Let $\mathcal Z$ be a $G$-invariant random closed subset of $X$. Then 
$$
\mathbb P(o\in\partial \mathcal Z)=0 \qquad \Longleftrightarrow \qquad {\rm vol}(\partial \mathcal Z) = 0 \, \, \mbox{a.s.}
$$
\end{lemma}

\begin{proof}
     By transitivity and $G$-invariance, $\mathbb P(x\in\partial \mathcal Z)=\mathbb P(o\in\partial \mathcal Z)$ for every $x\in X$. By the Fubini-Tonelli theorem,
 \begin{equation*}
     \mathbb E\big[ {\rm vol}(\partial \mathcal Z) \big] = \int \int \mathbf 1_{\{x\in\partial\mathcal Z\}} \, {\rm vol}(dx) \, d \mathbb P^{\mathcal Z} = \int \mathbb P(o\in\partial\mathcal Z) \, {\rm vol}(dx). 
 \end{equation*}
 Hence $\mathbb P(o\in\partial \mathcal Z)=0$ if and only if $\mathbb E\big[ {\rm vol}(\partial \mathcal Z) \big]=0$, which proves the lemma.
\end{proof}

We point out that, in the setting of Lemma \ref{lm:boundary}, the same argument shows that $\mathbb P(o\in\mathcal Z)$ equals the {\em volume fraction} $\mathbb E[{\rm vol}(B\cap \mathcal Z)]/{\rm vol}(B)$, where $B$ is measurable with $0<{\rm vol}(B)<\infty$.

\section{Finitary conditions: statements} \label{sec:FinitaryConditions}

In this section, we consider the setting where $G$ is higher rank and formulate a finitary analogue of the {\em infinite touching phenomenon} which was obtained for the IPVT in \cite{FMW23}. 

\medskip

We start with the intuitive fact that for every $N\in \mathbb{N}$, there exists some sufficiently large $R>0$ such that the $R$-ball around the root $o$ is split into at least $N$ Voronoi cells with high probability uniformly in small $\lambda>0$. This is closely related to the fact that subsequential weak limits of Poisson Voronoi tessellations are non-trivial. Thus the result below presumably holds more generally -- we focus on the case relevant for our purposes.

\begin{theorem}\label{thm:A1}
    Let $G$ be a connected higher rank semisimple real Lie group and $(X,d_X)$ be its symmetric space.
    Then for every $\eps>0$ and $N\in \mathbb{N}$ there is $\lambda_0\in (0,1]$ and $R>0$ such that 
    $$\mathbb{P}^{(\lambda)}\left(\forall 1\le i\le N \ \mathcal{B}_R(o)\cap C^{(\lambda)}_i\not=\emptyset\right)>1-\eps$$
    for every $0<\lambda\le \lambda_0$.    
\end{theorem}

Next, we formulate a finitary analogue of the infinite touching phenomenon proved in \cite{FMW23}. More precisely, the following result shows that, as $\lambda\to 0$, with high probability all pairs of Voronoi cells that touch a ball of fixed radius share a boundary. 

\begin{theorem}[Intersection of cells at small intensitites]\label{thm:A2}
    Let $G$ be a connected higher rank semisimple real Lie group and $(X,d_X)$ be its symmetric space.
    Then for every $\eps>0$ and $R>0$ there is $\lambda_0\in (0,1]$ such that 
    \begin{equation}\label{eq:FinitaryTouching}
        \mathbb{P}^{(\lambda)}\left(\forall i,j \in \mathbb{N} \  \left[C^{(\lambda)}_i\cap \mathcal{B}_R(o)\not=\emptyset\not= C^{(\lambda)}_j\cap \mathcal{B}_R(o) \  \Rightarrow \  C^{(\lambda)}_i\cap C^{(\lambda)}_j\not=\emptyset \right]\right)>1-\eps
    \end{equation}
    for every $0<\lambda\le \lambda_0$.
\end{theorem}

The proofs of Theorem \ref{thm:A1} and Theorem \ref{thm:A2} will be given in Section \ref{sec:FinitaryConditionsProofs}.

\section{Finitary conditions: proofs of Theorems~\ref{thm:A1} and Theorem~\ref{thm:A2}} \label{sec:FinitaryConditionsProofs}

This section is devoted to the proof of the finitary conditions formulated in Theorem~\ref{thm:A1} and Theorem~\ref{thm:A2}.

We start by recalling a general result about coupling of Poisson point processes in Section~\ref{sec:Skohorod}. In Section~\ref{sec:Corona}, we recall the definition and basic properties of the corona space, denoted ${\bf D}$, of the action $G\curvearrowright X$ from~\cite{FMW23}.
We use these preliminaries in Section~\ref{sec:CouplingOnCorona} to couple a sequence of Poisson point processes on $X$ with vanishing intensity measures with an ``ideal'' Poisson point process on ${\bf D}$. 
This result, together with the infinite touching phenomenon from \cite{FMW23}, which is recalled in Section~\ref{sec:VoronoiDiagram}, and an elementary fact about geodesic spaces, which is proven in Section~\ref{sec:Geodesic}, is then used in Sections~\ref{sec:ProofA2} and ~\ref{sec:ProofA1} to derive Theorems~\ref{thm:A2} and~\ref{thm:A1} respectively.

\subsection{A lemma about convergence of Poisson processes}\label{sec:Skohorod}

Let $S$ be a complete, separable metric space and let $\mathbf M(S)$ denote the set of Borel measures $\mu$ on $S$ which are locally finite in the sense that $\mu(A)<\infty$ for every bounded subset $A\subset S$. Let $\mathcal M(S)$ denote the $\sigma$-field generated by the evaluation mappings $\mu\mapsto\mu(A)$, where $A$ ranges over Borel subsets of $S$. We may identify a point process $\xi$ with  a random variable taking values in $\mathbf M(S)$ by identifying it with the induced counting measure. The distribution of a random measure $\xi$ is determined by its {\bf Laplace functional} 
\begin{equation}
    L_\xi(u) := \mathbb E \bigg[ \exp\bigg( - \int u(x) \xi(dx) \bigg) \bigg],
\end{equation}
for $u\in\R_+(S)$, the set of non-negative measurable functions, see e.g. \cite[Proposition~2.10]{LP17}. We equip $\mathbf M(S)$ with the topology of {\bf vague convergence}, denoted $\mu_n\overset{v}{\to}\mu$ and meaning that
\begin{equation} \label{def:vague}
    \lim_{n\to\infty} \int f(x) \mu_n(dx) \to \int f(x) \mu(dx)
\end{equation}
for every continuous function $f: S \to \R$  
with bounded support. We endow $\mathbf M(S)$ with a compatible metric making it a complete, separable metric space, see \cite[Proposition 9.1.IV]{DV08}.

A sequence $\xi_n$ of random measures {\bf converges weakly} to a random measure $\xi$ if the distributions $\mathbb P^{\xi_n}$ tend to $\mathbb P^{\xi}$ in the corresponding weak$^*$-topology, i.e.~
\begin{equation}
    \lim_{n\to\infty} \int h(\mu) \, \mathbb P^{\xi_n}(d\mu) = \int h(\mu) \, \mathbb P^{\xi}(d\mu) 
\end{equation}
for every continuous (in the vague toplogy) and bounded function $h: \mathbf M(S)\to\R $. This is equivalent to pointwise convergence of the Laplace functionals in the sense that
\begin{equation}
    \lim_{n\to\infty} L_{\xi_n}(u) = L_\xi(u)
\end{equation}
for every continuous function $u: S \to [0,\infty)$ with bounded support, see \cite[Prop.~11.1.VIII]{DV08}.

For Poisson processes, it is not difficult to see that vague convergence of the intensity measures suffices to guarantee a.s.~vague convergence of the realizations in a suitable coupling. We start with the following elementary observation.

\begin{lemma} \label{lm-WeakConvPPP}
    Let $S$ be a complete, separable metric space. Let $\eta,\eta_1,\eta_2,\ldots$ be Poisson processes with intensity measures $\mu,\mu_1,\mu_2,\ldots \in \mathbf M(S)$ such that $\mu_n\overset{v}{\to}\mu$. Then $\eta_n$ converges weakly to~$\eta$.  
\end{lemma}

\begin{proof}
    The Laplace functional of a Poisson process $\xi$ with locally finite intensity measure $\lambda$ is given by
    \begin{equation*}
        L_\xi(u) = \exp \bigg( - \int (1-e^{-u(x)}) \lambda(dx) \bigg), \quad u \in \R_+(S),
    \end{equation*}
    see \cite[Theorem 3.9]{LP17}. For every $u: S \to [0,\infty)$ continuous with bounded support, $f=1-e^{-u}$ is continuous with bounded support and hence $\mu_n\overset{v}{\to}\mu$ implies $L_{\eta_n}(\mu)\to L_\eta(u)$.
\end{proof}

We now describe the aforementioned suitable coupling.

\begin{theorem} \label{theorem:ConvergencePPP}
    Let $S$ be a complete, separable metric space. Let $\mu,\mu_1,\mu_2,\ldots \in \mathbf M(S)$ be such that $\mu_n\overset{v}{\to}\mu$. Then there exist Poisson processes $\eta,\eta_1,\eta_2,\ldots$ coupled on the same probability space such that for all $n\geq0$, the intensity measure of $\eta_n$ is $\mu_n$ and such that $\eta_n\overset{v}{\to}\eta$ a.s.
\end{theorem}

\begin{proof}
     This follows from Lemma \ref{lm-WeakConvPPP} and Skohorod's representation theorem \cite[Theorem 6.7]{B99}. 
\end{proof}

\begin{remark}\label{rem:LCSCvague}
    In our application of Theorem~\ref{theorem:ConvergencePPP}, we work exclusively with lcsc spaces.
    It is easy to see that in this context, vague convergence of locally finite measures with respect to any compatible complete proper metric can be equivalently defined as $\mu_n\overset{v}{\to}\mu$ if
    \begin{equation}\label{eq:TopVague}
        \lim_{n\to\infty} \int f(x) \mu_n(dx) \to \int f(x) \mu(dx)
    \end{equation}
    for every continuous function $f: S \to \R$ with compact support, cf.~\cite[Appendix A2.6]{DV08}.

    The way how we define the convergence in \eqref{def:vague} is the so-called $w^{\#}$-convergence of \cite{DV08} and refers to a metric.
    As it is a well-known fact that every lcsc space admits a compatible complete proper metric, we use the terminology of topological vague convergence, that is, convergence satisfying \eqref{eq:TopVague}, interchangeably with the $w^{\#}$-convergence with respect to any compatible complete proper metric.
    Importantly, we also remark that this topology is referred to as a \emph{weak-$^*$ convergence of measures} in \cite{FMW23}.
\end{remark}

\subsection{The corona space}\label{sec:Corona}
We start by recalling the setting from \cite[Section~3]{FMW23}.
Note that the results in \cite[Section~3]{FMW23} hold for any non-amenable lcsc group $G$ acting continuously, properly and transitively by isometries on a lcsc metric space $X$.
In particular, they apply to the setting of a non-compact connected semisimple real Lie group $G$ acting on its symmetric space $(X,d_X)$ by Theorem~\ref{thm:BasicSymmetricSpace}~(1).
Fix an origin $o\in X$.

\medskip

The \emph{corona space ${\bf D}$ of $X$} is defined as the minimal closed subspace of the space of continuous functions $C(X)$, endowed with the topology of uniform convergence on compact sets, that contains the set
\begin{equation}
    \{d_X(x,{-})+t:x\in X, t\in \mathbb{R}\}.
\end{equation}
In particular, any $f\in {\bf D}$ is a $1$-Lipschitz function and it is easy to see that ${\bf D}$ is a lcsc space see \cite[Section~3.1]{FMW23}.
The group $G$ acts continuously on ${\bf D}$ by left translations via $gf(x)=f(g^{-1}x)$ for every $g\in G$ and $x\in X$.

For each $t\in \mathbb{R}$, define a $G$-equivariant embedding $\iota_t:X\to {\bf D}$ as
\begin{equation}
\iota_t(x)(y)=d_X(x,y)-t,
\end{equation}
and for $t\ge 0$, define $\mu_t$ to be the push-forward of $\vol$ under $\iota_t$ normalized such that 
\begin{equation}
\mu_t(\{f\in {\bf D}:f(o)\le 0\})=1.
\end{equation}
Note that $\mu_t$ coincides with the push-forward of $c_t\vol$, where $\vol(\mathcal{B}_t(o))=c^{-1}_t$, under $\iota_t$. For the purposes of this paper, we will work with the vague limit points of $(\mu_t)_{t}$ as $t\to\infty$. The fundamental properties of such limit points are collected in the next result. Recall that $\mathbf M(\mathbf D)$ denotes the space of locally finite Borel measures on $\mathbf D$ endowed with the topology of vague convergence.

\begin{proposition}\label{pr:BasicIdealMeasure}
    Let $G$ be a non-compact connected semisimple real Lie group, $X$ be its symmetric space and ${\bf D}$ be the corona space of $X$.
    Then the sequence $\{\mu_t\}_{t\ge 1}$ is relatively compact in $\mathbf M(\mathbf D)\setminus \{0\}$.
    
    Moreover, any subsequential limit $\mu_{t_n}\overset{v}{\to}\mu\in \mathbf M({\bf D})$, $t_n\to \infty$, satisfies the following:
    \begin{enumerate}
        \item[{\rm(1)}] $\mu$ is $G$-invariant,
        \item[{\rm(2)}] $\mu(\{f\in {\bf D}: f(o)\le r\})<\infty$ for every $r\in \mathbb{R}$,
        \item[{\rm(3)}] $\mu(\{f\in {\bf D}: f(o)\le 0\})=1$,
        \item[{\rm(4)}] $\mu({\bf D})=\infty$,
        \item[{\rm(5)}] $\mu(C)<\infty$ on every compact set $C\subseteq {\bf D}$,
        \item[{\rm(6)}] $\mu(\{f\in {\bf D}: f(o)= r\})=0$ for every $r\in \mathbb{R}$.
    \end{enumerate}
\end{proposition}
\begin{proof}
    Recall from Remark~\ref{rem:LCSCvague} that vague convergence coincides with the weak$^*$-convergence considered in \cite{FMW23}. Hence, the fact that $\{\mu_t\}_{t\ge 1}$ is relatively compact in $\mathbf M(\mathbf D)\setminus \{0\}$ and items (2) and (3) follow  from \cite[Corollary 3.4]{FMW23} as $G$ is non-amenable by Theorem~\ref{thm:BasicSymmetricSpace}~(1).
    By  definition, we have that $\mu_t$ is $G$-invariant for every $t\in \mathbb{R}$.
    Consequently, we get item (1).
    If $C$ is compact, then $\{f(o):f\in C\}$ is relatively compact.
    Consequently, we have that $C\subseteq \{f\in {\bf D}: f(o)\le r\}$ for some $r>0$ and (5) follows from (2).
    It remains to show (4) and (6).
    
    \medskip

    {\noindent\em Proof of (4):}
    Let $r>0$.
    By Urysohn's lemma, there is a continuous function $H_r:{\bf D}\to [0,1]$ such that $H_r(f)=1$ for every $f\in {\bf D}$ such that $f(o)\in [-r,r]$ and $H_r(f)=0$ for every $f\in {\bf D}$ such that $f(o)\not\in (-r-1,r+1)$.
    By the definition, we have that
    $$
    \lim_{n\to\infty} \int_{{\bf D}} H_r \ d\mu_{t_n} = \int_{{\bf D}} H_r \ d\mu$$
    as $H_r$ has compact support.
    Observe that
    \begin{equation*}
        \begin{split}
            \int_{{\bf D}} H_r \ d\mu_{t_n}\ge & \  \mu_{t_n}(\{f\in {\bf D}:f(o)\le r\})-\mu_{t_n}(\{f\in {\bf D}:f(o)\le -r\})\\
            \ge & \ \frac{\vol(\mathcal{B}_{t_n+r}(o))}{\vol(\mathcal{B}_{t_n}(o))}-1
        \end{split}
    \end{equation*}
    for every $n\in \mathbb{N}$ by the definition.
    By Theorem~\ref{thm:BasicSymmetricSpace}~(6), we have that
    $$\frac{\vol(\mathcal{B}_{t_n+r}(o))}{\vol(\mathcal{B}_{t_n}(o))}=\frac{e^{a(t_n+r)}(t_n+r)^{b}(1+o(1))}{e^{at_n}(t_n)^{b}(1+o(1))}\overset{n\to\infty}{\longrightarrow} e^{ar}.$$
    It follows that 
    $$\int_{{\bf D}} H_r \ d\mu\ge e^{ar}-1\overset{r\to\infty}{\longrightarrow} \infty.$$
    Finally, note that 
    $$\int_{{\bf D}} H_{r} \ d\mu\le \mu(\{f\in {\bf D}: f(o)\le r+1\})\le \mu(\mathbf D)$$
    and (4) follows.

    \medskip
    
    {\noindent\em Proof of (6):}
    Let $r\in \mathbb{R}$ and $\eps>0$.
    By Urysohn's lemma, there is a continuous function $H_\eps:{\bf D}\to [0,1]$ such that $H_\eps(f)=1$ for every $f\in {\bf D}$ such that $f(o)\in [r-\eps,r+\eps]$ and $H_r(f)=0$ for every $f\in {\bf D}$ such that $f(o)\not\in (r-2\eps,r+2\eps)$.
    By the definition, we have that
    $$\lim_{n\to\infty} \int_{{\bf D}} H_\eps \ d\mu_{t_n} = \int_{{\bf D}} H_\eps \ d\mu$$
    as $H_\epsilon$ has compact support.
    Observe that
    \begin{equation*}
        \begin{split}
            \int_{{\bf D}} H_\eps \ d\mu_{t_n}\le & \  \mu_{t_n}(\{f\in {\bf D}: r-2\eps< f(o)\le r+2\eps\})\\
            = & \ \frac{\vol(\mathcal{B}_{t_n+r+2\eps}(o))}{\vol(\mathcal{B}_{t_n}(o))}-\frac{\vol(\mathcal{B}_{t_n+r-2\eps}(o))}{\vol(\mathcal{B}_{t_n}(o))}\\
            = & \ \frac{e^{a(t_n+r+2\eps)}(t_n+r+2\eps)^{b}(1+o(1))-e^{a(t_n+r-2\eps)}(t_n+r-2\eps)^{b}(1+o(1))}{e^{at_n}(t_n)^{b}(1+o(1))}\\
            \to & \ e^{a(r+2\eps)}-e^{a(r-2\eps)}
        \end{split}
    \end{equation*}
    as $n\to\infty$ by Theorem~\ref{thm:BasicSymmetricSpace} (6).
    Consequently, $\int_{{\bf D}} H_\eps \ d\mu\to 0$ as $\eps\to 0$, which implies (6) as
    $$\mu(\{f\in {\bf D}:f(o)=r\})\le \int_{{\bf D}} H_\eps \ d\mu$$
    for every $\eps>0$.
    This finishes the proof.
\end{proof}

\subsection{Coupling of Poisson point processes on ${\bf D}$}\label{sec:CouplingOnCorona}
It will be useful to identify the Poisson process $Y^{(\lambda)}$ with a corresponding Poisson process on the corona space, and to couple the Poisson processes of a given sequence $(\lambda_n)$ in a specific way. This is done in this section.

More precisely, let $t\ge 1$ and write $\Upsilon_t$ for the Poisson point process on ${\bf D}$ with intensity $\mu_t$. It follows directly from the definition of $\iota_t:X\to {\bf D}$ and $\mu_t$, that $\Upsilon_t$ is the push-forward under $\iota_t$ of the Poisson process $Y^{(\lambda)}$ on $X$ with intensity $\lambda\vol$, where $\vol(\mathcal{B}_t(o))=\lambda^{-1}$. In particular, the process has the form
\begin{equation} \label{equ rep Upsilon t}
\Upsilon_t = \sum_{i=1}^\infty \delta_{X_i^{(t)}}    
\end{equation}
with $X^{(t)}:=\big\{X^{(t)}_i\big\}_{i\in \mathbb{N}}$ such that $\big\{X^{(t)}_i(o)\big\}_{i\in \mathbb{N}}$ is strictly increasing.

\begin{remark}
    Note that we use the notation $X_i^{(t)}$ for the points of the Poisson process $\Upsilon_t$. We emphasize that these points are elements of $\mathbf D$, i.e.~functions on $X$, for which the notation $X_i^{(t)}(x)$ to denote the value at a point $x$ of $X$ will be used throughout. 
\end{remark}

Following \cite[Definition~1.2]{FMW23}, we say that a countable subset $F\subseteq {\bf D}$, resp.~locally finite Borel measure $\nu\in \mathbf M(\mathbf D)$ of the form $\nu=\sum_{f\in F} \delta_f$, is \emph{admissible} if $\{f(x):f\in F\}$, as a multiset, is discrete and bounded from below for every $x\in X$.

\begin{proposition}\label{pr:BasicPPP}
    Let $t_n\to\infty$, $\mu\in\mathbf M(\mathbf D)$ be such that $\mu_{t_n}\overset{v}\to \mu$ and let $\Upsilon$ be a Poisson point process on ${\bf D}$ with intensity $\mu$.
    Then $\Upsilon$ and $\Upsilon_{t_n}$ are a.s.~admissible and we may write 
    $$\Upsilon=\sum_{i=1}^\infty \delta_{X_i},$$
    with $\{X_i(o)\}_{i\in \mathbb{N}}$ strictly increasing. 
\end{proposition}
\begin{proof}
We first prove admissibility: For every $r>0$, $\{f\in \mathbf D: f(o)\le r\}$ has finite measure under $\mu_{t_n}$ as well as under $\mu$ by Proposition~\ref{pr:BasicIdealMeasure} (2). It follows that 
\begin{equation*}
\{f\in \Upsilon_{t_n} : f(o)\le r\} \quad \mbox{and} \quad \{f\in \Upsilon : f(o)\le r\}    
\end{equation*} 
are finite for every $r>0$ a.s. In particular, the multisets
\begin{equation*}
\{f(o) : f\in \Upsilon_{t_n}\} \quad \mbox{and} \quad \{f(o) : f\in \Upsilon \}    
\end{equation*} 
are discrete and bounded from below a.s.
We claim that this implies that $\Upsilon_{t_n}$ and $\Upsilon$ are a.s. admissible.
Indeed, let $x\in X$.
Then, using the fact that every $g\in {\bf D}$ is $1$-Lipschitz, we have that
$$f(o)-d_X(o,x)\le f(x)$$
for every $f\in \Upsilon$, which implies that the multiset $\{f(x):f\in \Upsilon\}$ is bounded from below.
Similarly, if
$$\{f(x):f\in \Upsilon\}\cap [a,b]$$
is infinite for some $a<b\in \mathbb{R}$, then the multiset
$$\{f(o):f\in \Upsilon\}\cap [a-d_X(o,x),b+d_X(o,x)]$$
is infinite as well.
Consequently, $\Upsilon$ (and similar argument applies to $\Upsilon_{t_n}$ for every $n\in \mathbb{N}$) is admissible a.s.

Moreover, writing $\Upsilon$ in its proper point process representation, cf.~\cite[Corollary 3.7]{LP17}, and rearranging the random elements according to the value at the root yields the desired representation 
\begin{equation} \label{equ rep Upsilon}
\Upsilon=\sum_{i=1}^\infty \delta_{X_i},
\end{equation}
with $\{X_i(o)\}_{i\in \mathbb{N}}$ non-decreasing.
The fact that $\{X_i(o)\}_{i\in \mathbb{N}}$ is strictly increasing follows from Proposition~\ref{pr:BasicIdealMeasure} (6) by a straightforward application of the Mecke equation as in the proof of Lemma \ref{lm:BasicVoronoi}.
\end{proof}

\begin{theorem}[Coupling] \label{thm:Coupling}
    Let $t_n\to\infty$, $\mu\in\mathbf M(\mathbf D)$ be such that $\mu_{t_n}\overset{v}\to \mu$ and let $\Upsilon$ be a Poisson point process on ${\bf D}$ with intensity $\mu$.
    There exists a coupling of $\Upsilon$ and $\Upsilon_{t_n}$, $n\in \mathbb{N}$, on the same probability space such that 
    $$
    \Upsilon=\sum_{i=1}^\infty \delta_{X_i} \quad \mbox{and} \quad \Upsilon_{t_n}=\sum_{i=1}^\infty \delta_{X_i^{(t_n)}},
    $$
    where $\{X_i(o)\}_{i\in \mathbb{N}}$ and $\{X_i^{(t_n)}(o)\}_{i\in \mathbb{N}}$ are strictly increasing, such that
    $$
    \lim_{n\to\infty} X^{(t_n)}_i= X_i
    $$
    for every $i\in \mathbb{N}$ a.s.
\end{theorem}
\begin{proof}
    Clearly, our aim is to apply Theorem~\ref{theorem:ConvergencePPP}.
    However, a direct application of this theorem does not guarantee in a straightforward way that $X^{(t_n)}_1\not\to-\infty$ a.s.
    We circumvent this issue as follows.

    \medskip

    Let ${\bf D}^-={\bf D}\cup\{-\infty\}$ with the sets $\{f\in {\bf D}:f(o)<r\}\cup \{\infty\}$ for $r\in \mathbb{R}$ forming an open neighborhood base at $-\infty$.
    That is, we compactify ${\bf D}$ at $-\infty$.
    It can be easily checked that the topology generated by the original topology on ${\bf D}$ together with this base at $\{-\infty\}$ turns ${\bf D}^-$ into a lcsc space.
    Observe also that the restriction of the $\sigma$-algebra of Borel sets from ${\bf D}^-$ to ${\bf D}$ coincides with the original  $\sigma$-algebra of Borel sets on ${\bf D}$.
    In particular, we may view $\mu_{t_n},\mu\in {\bf M}({\bf D}^-)$ for every $n\in \mathbb{N}$, and we have that $\mu_{t_n}(\{-\infty\})=\mu(\{-\infty\})=0$ for every $n\in \mathbb{N}$.
    Similarly, we abuse the notation and write $\Upsilon$ and $\Upsilon_{t_n}$, $n\in \mathbb{N}$, for the Poisson point processes with intensity measure $\mu$ and $\mu_{t_n}$ on ${\bf D}^-$.
    Notice that as almost surely the point $-\infty$ does not appear in $\Upsilon$ and $\Upsilon_{t_n}$ for every $n\in \mathbb{N}$ the restrictions of these Poisson point processes to ${\bf D}$ coincide with the original definition of $\Upsilon$ and $\Upsilon_{t_n}$ for every $n\in \mathbb{N}$.

    \begin{claim}
        We have that $\mu_{t_n}$ converges vaguely to $\mu$ on ${\bf D}^-$. 
    \end{claim}
    \begin{proof}
        We need to show that if $H:{\bf D}^-\to \mathbb{R}$ is a continuous function with compact support, c.f. Remark~\ref{rem:LCSCvague}, then 
        $$\int_{{\bf D}^-} H \ d\mu_{t_n}\to \int_{{\bf D}^-} H \ d\mu.$$
        Let $M>0$ be such that $|H(f)|\le M$ for every $f\in {\bf D}^-$. 

        \medskip
        
        Assume first that $H\ge 0$ and $H(f)=0$ whenever $f(o)\ge 0$ for $f\in {\bf D}$, and fix $\eps>0$.
        By Proposition~\ref{pr:BasicIdealMeasure}~(3) and~(6) combined with the fact that $\mu(\{-\infty\})=0$, there are $s,\delta>0$ such that
        $$\mu(\{f\in {\bf D}:-s\le f(o)\le -\delta\})>1-\eps/M.$$
        By Urysohn's lemma, there is a continuous function $F:{\bf D}\to [0,1]$ such that $F(f)=1$ whenever $-s\le f(o)\le -\delta$ and $F(f)=0$ whenever $f(o)\le -s-1$, or $f(o)\ge 0$.
        Then we have
        $$\int_{{\bf D}} F \ d\mu_{t_n}\to \int_{{\bf D}} F \ d\mu,$$
        by $\mu_{t_n}\overset{v}\to \mu$ on ${\bf D}$, which implies that there is $n_0\in \mathbb{N}$ such that
        $$\mu_{t_n}(\{f\in {\bf D}:-s-1\le f(o)\le 0\})\ge \int_{{\bf D}} F \ d\mu_{t_n}>1-\eps/M$$
        holds for every $n\ge n_0$. 

        Let $G:{\bf D}^-\to [0,1]$ be a continuous function such that $G(-\infty)=0$, $G(f)=1$ whenever $-s-1\le f(o)\le 0$ and $G(f)=0$ whenever $f(o)\le -s-2$, or $f(o)\ge 1$ for $f\in {\bf D}$.
        Then we have
        $$\left|\int_{{\bf D}^-} H \ d\mu_{t_n}-\int_{{\bf D}^-} G\cdot H \ d\mu_{t_n}\right|<\eps$$
        for every $n\ge n_0$ as well as
        $$\left|\int_{{\bf D}^-} H \ d\mu-\int_{{\bf D}^-} G\cdot H \ d\mu\right|<\eps.$$
        The desired claim then follows by sending $\eps\to 0$ as
        $$\int_{{\bf D}^-} G\cdot H \ d\mu_{t_n}=\int_{{\bf D}} G\cdot H \ d\mu_{t_n}\to \int_{{\bf D}} G\cdot H \ d\mu=\int_{{\bf D}^-} G\cdot H \ d\mu$$
        by the assumption that $\mu_{t_n}\overset{v}\to \mu$ on ${\bf D}$ as $G\cdot H\upharpoonright {\bf D}$ has compact support.

        \medskip

        To finish the proof for general $H$, note that we can write $H=H_0+H_1$, where $H_1$ has compact support when restricted to ${\bf D}$ and $H_0(f)=0$ for all $f\in {\bf D}$ such that $f(o)\ge 0$, and further $H_i=H_i^+-H_i^-$ with $H_i^+,H_i^-\ge 0$ for $i=0,1$.
    \end{proof}

    By Theorem~\ref{theorem:ConvergencePPP}, we may couple $\Upsilon$ and $\Upsilon_{t_n}$, $n\in \mathbb{N}$, on the same probability space such that
    \begin{equation}\label{eq:ConvergenceOnCompactification}
        \Upsilon_{t_n} \overset{v}\to \Upsilon \quad \mbox{a.s. on } {\bf D}^-.    
    \end{equation}
    Representing $\Upsilon_{t_n}$ in the form \eqref{equ rep Upsilon t} and $\Upsilon$ in the form \eqref{equ rep Upsilon}, it remains to show that 
    $$\lim_{n\to\infty} X^{(t_n)}_i=X_i$$
    a.s. in ${\bf D}$.
    
    \medskip

    Let $i\in \mathbb{N}$.
    Using Proposition~\ref{pr:BasicPPP}, we have that $X_i(o)<X_{i+1}(o)$ a.s.
    By Urysohn's lemma, there is a continuous function $H:{\bf D}^-\to [0,1]$ such that $H(-\infty)=1$, $H(f)=1$ for every $f\in {\bf D}$ such that $f(o)\le 5/8X_i(o)+3/8X_{i+1}(o)$ and $H(f)=0$ for every $f\in {\bf D}$ such that $f(o)\ge 3/8X_i(o)+5/8X_{i+1}(o)$.
    It follows from \eqref{eq:ConvergenceOnCompactification} that 
    $$\int_{{\bf D}^-} H \ d\Upsilon_{t_n}\to \int_{{\bf D}^-} H \ d\Upsilon=i.$$
    We claim that there is $n_0\in \mathbb{N}$ such that $X^{(t_n)}_{i+1}(o)\ge 3/4X_i(o)+1/4X_{i+1}(o)$.
    Indeed, suppose for a contradiction that $X^{(t_n)}_{i+1}<3/4X_i(o)+1/4X_{i+1}(o)$ for infinitely many $n\in \mathbb{N}$.
    Then we have
    $$\int_{{\bf D}^-} H \ d\Upsilon_{t_n}=i+1$$
    for every such $n\in \mathbb{N}$, and that is a contradiction.
    Consequently, using \eqref{eq:ConvergenceOnCompactification} again, it is straightfroward to show by induction that
    $$\lim_{n\to\infty} X^{(t_n)}_j=X_j$$
    for every $1\le j\le i$ in ${\bf D}^-$.
    As $i\in \mathbb{N}$ was arbitrary, and $X_1\not=-\infty$ a.s., we get that the convergence holds when restricted to ${\bf D}$ and the claim follows.
\end{proof}

\begin{remark}[Convergence of Voronoi diagrams] \label{rm:convergence}
    Let us reiterate that we do not prove that the Poisson-Voronoi diagrams converge (in the Fell topology) to a unique limiting tessellation. This has been shown (along with a systematic study of the limiting object) for (discrete) trees~\cite{B19}, hyperbolic spaces~\cite{DCELU23} and the $L^1$-product of hyperbolic planes~\cite{D24}. Instead, we work with subsequential limit points and provide a coupling to obtain convergence of the point processes on~$\mathbf D$. Since this suffices for our purposes, we do not pursue the question about convergence of the Voronoi diagrams. 
    
    Notably, a very general convergence criterion for Voronoi diagrams was proven in \cite[Theorem 2.3]{DCELU23}. This result provides sufficient conditions for a deterministic list of nuclei in any locally compact proper metric space to converge to an ideal diagram. We point out that the nuclei need not form a Poisson point process and the space need not be a symmetric space (in fact, the space need not be a Riemannian manifold \cite{D24} and many discrete spaces, also beyond Cayley graphs, are allowed). To apply their result in the present context, one would additionally have to check convergence of the Poisson points in the Gromov compactification of $X$, cf.~\cite{DCELU23}.
\end{remark}

\subsection{Voronoi diagrams, revisited}\label{sec:VoronoiDiagram}
For an admissible set $F$, or an admissable locally finite Borel measure $\nu\in \mathbf M(\mathbf D)$ of the form $\nu=\sum_{f\in F} \delta_f$, define the \emph{Voronoi diagram}, see \cite[Definition~1.2]{FMW23}, as 
\begin{equation}C^\nu:=C^F:=\{C^F_f\}_{f\in F},
\end{equation}
where
\begin{equation} 
C^\nu_f:=C^F_f=\{x\in X:\forall g\in F \ f(x)\le g(x)\}
\end{equation}
for every $f\in F$.
Also, for $D>0$, $f_1,f_2\in F$, we follow \cite[Section~6]{FMW23} and define the {\em $D$-wall} of~$F$ with respect to $(f_1,f_2)$ as 
\begin{align}
    W^\nu_D(f_1,f_2) & :=W^F_D(f_1,f_2) \nonumber\\
    & :=\{x\in X: f_1(x)=f_2(x) \text{ and } \forall g\in F\setminus \{f_1,f_2\} \ f_1(x)+D< g(x)\}.
\end{align}
Note that if $t\ge 1$ and $\lambda>0$ is such that $\vol(\mathcal{B}_t(o))=\lambda^{-1}$, then 
\begin{equation}
    \mathrm{Vor}(\mathbf{Y}^{(\lambda)}) \overset{(d)}{=} C^{\Upsilon_t},
\end{equation} 
i.e.~we have simply expressed the previous Voronoi tessellation with a different formalism.

Note that the above formalism make sense for realizations of the Poisson point processes in Theorem \ref{thm:Coupling}. With a slight abuse of notation, we will treat these realizations as both measures and sets depending on the context. We now recall a (formally) weaker version of a key result proved in~\cite{FMW23}, which suffices for our purposes.

\begin{theorem}[Touching in the limit, {cf.~\cite[Theorem 6.1]{FMW23}}]\label{thm:InfiniteTouching}
    Let $G$ be a connected higher rank semisimple real Lie group, $X$ its symmetric space, ${\bf D}$ the corona space and $t_n\to \infty$ such that $\mu_{t_n}\overset{v}\to \mu$. Then the Poisson point process $\Upsilon$ with intensity $\mu$ has the following property almost surely: for every $D>0$ and $f_1,f_2\in \Upsilon$, the set $W^\Upsilon_D(f_1,f_2)$ is non-empty. 
\end{theorem}

The relationship with \cite[Theorem 6.1]{FMW23} is explained in the following remark.

\begin{remark}[Touching vs.~infinite touching] \label{rem:InfiniteTouching}
    In the setting of Theorem \ref{thm:InfiniteTouching}, a stronger property was proven in \cite[Theorem 6.1]{FMW23}. Namely, it was shown there that ''non-empty'' may  be replaced by ''unbounded''. We shall only need the (formally) weaker version stated above. 
\end{remark}

\subsection{A lemma about geodesics}\label{sec:Geodesic}

Let $(X,d)$ be a geodesic metric space. We fix, for every $x,y\in X$, some geodesic $\gamma_{x,y}$ connecting $x$ to $y$. Abusing notation, we will also write $\gamma_{x,y}$ for the image $\gamma_{x,y}([0,d(x,y)])\subseteq X$.
Given $\delta>0$, we define the {\em $\delta$-thickening} of $\gamma_{x,y}$ as 
\begin{equation}
   \Gamma_{x,y}(\delta)=\{z\in X:d(z,\gamma_{x,y})<\delta\} 
\end{equation}

\begin{lemma}\label{lm:BasicGeodesic}
    Let $(X,d)$ be a geodesic metric space and $x,y,w\in X$ be such that $d(x,w)<d(y,w)$.
    Then
    $$d(x,z)<d(y,z)$$
    holds for every $z\in \Gamma_{x,w}(\delta)$ and $0<\delta<(d(y,w)-d(x,w))/2$.
\end{lemma}
\begin{proof}
    Suppose for a contradiction that there is $z\in \Gamma_{x,w}(\delta)$ such that $d_G(y,z)\le d_G(x,z)$ and let $z_0\in \gamma_{x,w}$ be such that $d_G(z_0,z)<\delta$.
    Then we have by the triangle inequality that
    \begin{equation*}
        \begin{split}
            d_G(y,w)\le & \ d_G(y,z)+d_G(z,w)\le d_G(x,z)+d_G(z,w)\\
            \le & \ d_G(x,z_0)+ d_G(z_0,z)+d_G(z_0,z)+d_G(z_0,w)\\
            \le & \ d_G(x,w)+2\delta< d_G(y,w),     
        \end{split}
    \end{equation*}
    which is a contradiction.
\end{proof}

\subsection{Proof of Theorem~\ref{thm:A2}}\label{sec:ProofA2}
We shall prove the following stronger statement that might be useful in other applications as well. 

\begin{theorem}\label{thm:StrongerA2}
    Let $G$ be a connected higher rank semisimple real Lie group and let $(X,d_X)$ be its symmetric space.
    Then for every $\eps>0$, $D>0$ and $R>0$ there is $\lambda_0\in (0,1]$ such that for
    \begin{align}
        E^{(\lambda)}_{R,D} := \Big\{ \forall i,j \in \mathbb{N} \  C^{(\lambda)}_i\cap \mathcal{B}_R(o) & \not=\emptyset\not= C^{(\lambda)}_j\cap \mathcal{B}_R(o) \nonumber \\  
        & \Rightarrow\exists z\in C^{(\lambda)}_i\cap C^{(\lambda)}_j \text{ such that } \mathcal{B}_D(z)\subseteq C^{(\lambda)}_i\cup C^{(\lambda)}_j \label{eq:StrongerTouching} \Big\},
    \end{align}
    we have that
    \begin{equation}\label{eq:StrongerFinitaryTouching}
        \inf_{0<\lambda\le\lambda_0} \mathbb{P} \left(E^{(\lambda)}_{R,D}\right)>1-\eps.
    \end{equation}
\end{theorem}
\begin{proof}
    For a contradiction, suppose that there exist $\eps>0$, $R>0$ and $D>0$ that do not satisfy the conclusion of Theorem~\ref{thm:StrongerA2}.
    In particular, we may find a decreasing sequence $\lambda_n\to 0$ such that the events $E^{(\lambda_n)}_{R,D}$ in \eqref{eq:StrongerTouching} have probability at most $1-\eps$ for all $n\in\mathbb N$.

    By Theorem~\ref{thm:BasicSymmetricSpace} (5), there are $t_n\to \infty$ such that $\vol(\mathcal{B}_{t_n}(o))=\lambda_n^{-1}$.
    By Proposition~\ref{pr:BasicIdealMeasure}, we may without loss of generality assume that $\mu_{t_n}\overset{v}\to \mu$.
    As this sequence will be fixed from now on, we set $\mu_n=\mu_{t_n}$, $\Upsilon_{n}=\Upsilon_{t_n}$, etc.

    Let $\Upsilon$ and $\Upsilon_n$ for every $n\in \mathbb{N}$ be coupled as in Theorem~\ref{thm:Coupling} and define a random variable
    $$T=\Big\{i\in \mathbb{N}: C^{\Upsilon}_{X_i}\cap \overline{\mathcal{B}}_{R}(o)\not=\emptyset\Big\}.$$
    As $\Upsilon$ is a.s. admissible, we have that $|T|<\infty$ a.s.
    Indeed, as $X_i$ is $1$-Lipschitz for every $i\in \mathbb{N}$, we have $X_i(o)\le X_1(o)+2R$ for every $i\in T$. 
    
    \begin{claim}\label{cl:RandomTime}
        Define 
        $$T_n=\Big\{i\in \mathbb{N}: C^{\Upsilon_n}_{X^{(n)}_i}\cap \mathcal{B}_R(o)\not=\emptyset\Big\} \quad.$$
        Then there is a random variable $N_0\in \mathbb{N}$ such that $T_n\subseteq T$ for every $n\ge N_0$ a.s.

    \end{claim}
    \begin{proof}
        First, we show that there is a random variable $K\in \mathbb{N}$ such that $T_n\subseteq [K]=\{0,\dots,K\}$ for every $n\in \mathbb{N}$ a.s.
        Consider the event that for every $k\in \mathbb{N}$ there is $n_k\in \mathbb{N}$ and $k<\ell_k\in T_{n_k}$.
        By the definition, there is $x_k\in \mathcal{B}_R(o)$ such that $X^{(n_k)}_{\ell_k}(x_k)\le X^{(n_k)}_1(x_k)$.
        As both functions are $1$-Lipschitz, it follows that $X^{(n_k)}_{\ell_k}(o)-2R\le X^{(n_k)}_1(o)$.
        This implies, as $X^{(n)}_1(o)\to X_1(o)\in \mathbb{R}$ by Theorem~\ref{thm:Coupling}, that there is $M>0$ such that $X^{(n_k)}_{\ell_k}(o)<M$.
        Consequently, as $\ell_k\to \infty$, we have that $X_i(o)\le M$ for every $i\in \mathbb{N}$ which shows that on this event $\Upsilon$ is not admissible.
        Hence, this event has probability $0$.

        Let $i\in \mathbb{N}$ be such that $i\in T_n$ for infinitely many $n\in \mathbb{N}$.
        To finish the proof it is clearly enough, by the previous paragraph, to show that $i\in T$.
        It follows from the definition that for every such $n\in \mathbb{N}$ there is $x_n\in \mathcal{B}_{R}(o)$ such that $X^{(n)}_{i}(x_n)\le X^{(n)}_{j}(x_n)$ for every $j\in \mathbb{N}$.
        As $\overline{\mathcal{B}}_{R}(o)$ is compact, we find, after passing to a subsequence if necessary, $x\in \overline{\mathcal{B}}_{R}(o)$ such that $x_n\to x$.
        We claim that $X_i(x)\le X_j(x)$ for every $j\in \mathbb{N}$, which shows $i\in T$.
        Indeed, as the elements of ${\bf D}$ are $1$-Lipschitz, we have for every such $n\in \mathbb{N}$ that
        \begin{equation*}
            \begin{split}
                0\ge & \ X^{(n)}_i(x_n)-X^{(n)}_j(x_n)\\
                = & \  X^{(n)}_i(x)-X^{(n)}_j(x)+(X^{(n)}_i(x_n)-X^{(n)}_i(x))-(X^{(n)}_j(x_n)-X^{(n)}_j(x))\\
                \ge & \ X^{(n)}_i(x)-X^{(n)}_j(x)-2d_X(x_n-x).
            \end{split}
        \end{equation*}
        Note that this implies that $ X_i(x)\le X_j(x)$ as $X^{(n)}_i(x)\to X_i(x)$, $X^{(n)}_j(x)\to X_j(x)$ and $d_X(x_n-x)\to 0$.
        This finishes the proof.
    \end{proof}

    We claim that there is another random variable $N_1\ge N_0$ such that for every $n\ge N_1$ and every $i,j\in T_n$ there is $z^{n}_{i,j}\in X$ such that 
    \begin{equation}\label{eq:LastStep}
        |X^{(n)}_i(z^n_{i,j})-X^{(n)}_j(z^n_{i,j})|<2D \text{ and } X^{(n)}_\ell(z^n_{i,j})+4D< X^{(n)}_k(z^n_{i,j})
    \end{equation}
    for every $\ell\in \{i,j\}$ and $k\in \mathbb{N}\setminus \{i,j\}$.

    Appealing to Theorem~\ref{thm:InfiniteTouching}, there is a.s.~$z_{i,j}\in  W^{\Upsilon}_{8D}(X_i,X_j)$ for every $i,j\in T$.
    We show that $N_1$ is well-defined a.s. by showing that the inequalities in \eqref{eq:LastStep} are satisfied for large $n\in \mathbb{N}$ with the particular choice of $z^n_{i,j}=z_{i,j}$.
    In the rest of the argument we use tacitly Claim~\ref{cl:RandomTime}, that is, we assume that $n\in \mathbb{N}$ is large enough so that $T_n\subseteq T$.
    
    The inequality on the left-hand side of \eqref{eq:LastStep} for large $n\in \mathbb{N}$ follows directly from
    $$X_i(z_{i,j})=X_j(z_{i,j})$$
    combined with the fact that $X^{(n)}_m(z_{i,j})\to X_m(z_{i,j})$ fo every $m\in \mathbb{N}$ Theorem~\ref{thm:Coupling}.
    Concerning the inequalities on the right-hand side \eqref{eq:LastStep}, let $\ell=i$ and assume for a contradiction that there are infinitely many $n\in \mathbb{N}$ such that $X^{(n)}_{k_n}(z_{i,j})\le X^{(n)}_i(z_{i,j})+4D$ for some $k_n\in \mathbb{N}\setminus \{i,j\}$.
    As $\{X_i\}_{i\in \mathbb{N}}$ is admissible almost surely, it must be the case that there is $k'\in \mathbb{N}$ such that $k_n\le k'$ for every such $n\in \mathbb{N}$.
    Indeed, we have
    $$X^{(n)}_{k_n}(o)\le X^{(n)}_{k_n}(z_{i,j})+d_X(z_{i,j},o)\le X^{(n)}_i(z_{i,j})+4D+d_X(z_{i,j},o)\le M$$
    for some $M>0$ as $X^{(n)}_i(z_{i,j})\to X_i(z_{i,j})$ by Theorem~\ref{thm:Coupling}.
    In particular, there are infinitely many $n\in \mathbb{N}$ such that $k_n=k\in \mathbb{N}$, and consequently
    $$X^{(n)}_k(z_{i,j})\to X_k(z_{i,j})\le X_i(z_{i,j})+4D,$$
    which contradicts the definition of $z_{i,j}\in W^\Upsilon_{8D}(X_i,X_j)$.
    The argument for $\ell=j$ is completely analogous.
    This proves the existence of the random variable $N_1$.
    
    Choose $n\in \mathbb{N}$ such that the event $A_0:=\{N_1\le n\}$ has probability at least $1-\eps$. We claim that, conditional on $A_0$, \eqref{eq:StrongerTouching} holds almost surely for $\lambda_{n}$, which is the desired contradiction.

    Conditional on $A_0$, let $i\not =j\in \mathbb{N}$ be such that 
    $$C^{\Upsilon_n}_{X^{(n)}_i}\cap \mathcal{B}_R(o)\not=\emptyset\not= C^{\Upsilon_n}_{X^{(n)}_j}\cap \mathcal{B}_R(o),$$
    or equivalently, $i\not =j\in T_n$.
    As $n\ge N_1$, there is $z^n_{i,j}\in X$ such that 
    $$|X^{(n)}_i(z^n_{i,j})-X^{(n)}_j(z^n_{i,j})|<2D \quad \text{and} \quad X^{(n)}_\ell(z^n_{i,j})+4D\le X^{(n)}_k(z^n_{i,j})$$
    for every $\ell\in \{i,j\}$ and $k\in \mathbb{N}\setminus \{i,j\}$.

    Assume without loss of generality that $X^{(n)}_i(z^n_{i,j})\ge X^{(n)}_j(z^n_{i,j})$ and consider the geodesic $\gamma_{z^n_{i,j},y_i}$ in $X$, where $y_i=\iota^{-1}_{t_n}(X^{(n)}_i)$.
    As $d_X(y_j,\cdot)-d_X(y_i,\cdot)$ is continuous, where $y_j=\iota^{-1}_{t_n}(X^{(n)}_j)$, there is $z\in \gamma_{z^n_{i,j},y_i}$ such that
    $$d_X(y_j,z)-d_X(y_i,z)=0.$$
    Finally, by Lemma~\ref{lm:BasicGeodesic}, we have that
    $$d(y_i,z')<d(y_k,z')$$
    for every $z'\in \mathcal{B}_{D}(z)\subseteq \Gamma_{z^n_{i,j},y_i}(3/2D)$, where $y_k=\iota^{-1}_{t_n}(X^{(n)}_k)$ for every $k\in \mathbb{N}\setminus \{i,j\}$.
    It follows that $z\in X$ works as required in $\eqref{eq:StrongerTouching}$.
    This concludes the proof.    
\end{proof}

\subsection{Proof of Theorem~\ref{thm:A1}}\label{sec:ProofA1}

Theorem~\ref{thm:A1} can be proved in an analogous way by using Theorem~\ref{thm:InfiniteTouching} with $f_1=f_2$.
In fact, this way we obtain the following strengthening of Theorem~\ref{thm:A1}.

\begin{theorem}\label{thm:StrongerA1}
    Let $G$ be a connected higher rank semisimple real Lie group and $(X,d_X)$ be its symmetric space.
    Then for every $\eps>0$, $D\ge 0$ and $N\in \mathbb{N}$ there is $\lambda_0>0$ and $R>0$ such that 
    $$\mathbb{P}\left(\forall 1\le i\le N \ \exists z_i\in \mathcal{B}_R(o) \text{ such that }  \mathcal{B}_D(z_i)\subseteq C^{(\lambda)}_i\right)>1-\eps$$
    for every $0<\lambda\le \lambda_0$.    
\end{theorem}
\begin{proof}
    For a contradiction, suppose that there exist $\eps>0$, $D>0$ and $N>0$ that do not satisfy the conclusion of Theorem~\ref{thm:StrongerA1}.
    In particular, we may find a decreasing sequence $\lambda_n\to 0$ and an increasing sequence $R_n\to \infty$ such that
    $$\mathbb{P}\left(\forall 1\le i\le N \ \exists z_i\in \mathcal{B}_{R_n}(o) \text{ such that }  \mathcal{B}_D(z_i)\subseteq C^{(\lambda_n)}_i\right)\le 1-\eps$$
    for every $n\in \mathbb{N}$.

    By Theorem~\ref{thm:BasicSymmetricSpace} (5), there are $t_n\to \infty$ such that $\vol(\mathcal{B}_{t_n}(o))=\lambda_n^{-1}$.
    By Proposition~\ref{pr:BasicIdealMeasure}, we may without loss of generality assume that $\mu_{t_n}\overset{v}\to \mu$.
    As this sequence will be fixed from now on, we set $\mu_n=\mu_{t_n}$, $\Upsilon_{n}=\Upsilon_{t_n}$, etc.

    Let $\Upsilon$ and $\Upsilon_n$ for every $n\in \mathbb{N}$ be coupled as in Theorem~\ref{thm:Coupling}.
    By Theorem~\ref{thm:InfiniteTouching}, we have that a.s. for every $1\le i\le N$ there is $z_i\in X$ such that
    $$z_i\in W^\Upsilon_{4D}(X_i,X_i).$$
    In particular, there is $R>0$ such that
    \begin{equation}\label{eq:SingleWall}
        \mathbb{P}\big(\forall 1\le i\le N \ \exists z_i\in \mathcal{B}_{R}(o)\cap W^\Upsilon_{4D}(X_i,X_i)\big)>1-\eps/2.        
    \end{equation}
    By Theorem~\ref{thm:Coupling}, we have that a.s.
    $$\lim_{n\to\infty} X^{(n)}_i\to X_i$$
    for every $i\in \mathbb{N}$.

    Let $n_0\in \mathbb{N}$ be such that $R_n\ge R$ for every $n\ge n_0$.
    We claim that conditioned on the event from \eqref{eq:SingleWall} there is a random variable $K\ge n_0$ such that a.s.~$K\in \mathbb{N}$ and for every $n\ge K$ and every $1\le i\le N$ there is $z^{n}_{i}\in \mathcal{B}_R(o)$ such that 
    \begin{equation}\label{eq:LastStepVersion}
        X^{(n)}_i(z^n_{i})+2D< X^{(n)}_k(z^n_{i})
    \end{equation}
    for every $k\in \mathbb{N}\setminus \{i\}$.

    We show that for large $n\in \mathbb{N}$, \eqref{eq:LastStepVersion} holds with the choice $z^n_i=z_i$, where $z_i$ is from \eqref{eq:SingleWall}.
    Assume for a contradiction that there are infinitely many $n\ge n_0$ such that $X^{(n)}_{k_n}(z_{i})\le X^{(n)}_i(z_{i})+2D$ for some $k_n\in \mathbb{N}\setminus \{i\}$.
    As $\{X_i\}_{i\in \mathbb{N}}$ is admissible almost surely, it must be the case that there is $k'\in \mathbb{N}$ such that $k_n\le k'$ for every such $n\in \mathbb{N}$.
    Indeed, we have
    $$X^{(n)}_{k_n}(o)\le X^{(n)}_{k_n}(z_{i})+d_X(z_{i},o)\le X^{(n)}_i(z_{i})+2D+R\le M$$
    for some $M>0$ as $X^{(n)}_i(z_{i})\to X_i(z_{i})$ by Theorem~\ref{thm:Coupling}.
    In particular, there are infinitely many $n\in \mathbb{N}$ such that $k_n=k\in \mathbb{N}$, and consequently
    $$X^{(n)}_k(z_{i})\to X_k(z_{i})\le X_i(z_{i})+2D,$$
    which contradicts the definition of $z_{i}\in W^\Upsilon_{4D}(X_i,X_i)$.
    This shows that $K$ is defined a.s. on the event from \eqref{eq:SingleWall}.

    To get the desired contradiction, take $n\in \mathbb{N}$ such that $\mathbb{P}(K\le n)>1-\eps$.
    Indeed, on this event, we have that for every $1\le i\le N$ there is $z^n_i\in \mathcal{B}_R(o)\subseteq \mathcal{B}_{R_n}(o)$ such that $X^{(n)}_i(z^n_{i})+2D< X^{(n)}_k(z^n_{i})$ holds for every $k\in \mathbb{N}\setminus \{i\}$.
    Translating back to $Y^{(\lambda_n)}$, this means that for every $1\le i\le N$ we have that
    $$d_X(z^n_i,Y^{(\lambda_n)}_i)+2D<d_X(z^n_i,Y^{(\lambda_n)}_k)$$
    for every $k\in \mathbb{N}\setminus \{i\}$, which in turn, using the triangle inequality, implies that 
    $$\mathcal{B}_D(z^n_i)\subseteq C^{(\lambda_n)}_i$$
    as desired.
    This finishes the proof.
\end{proof}

We will need the following result, which may be viewed as a converse to Theorem~\ref{thm:StrongerA1}, in Sections~\ref{sec:sparse} and \ref{sec:Lattices}.

\begin{proposition} \label{pr:UpperBoundBall}
    Let $G$ be a connected higher rank semisimple real Lie group and let $(X,d_X)$ be its symmetric space.
    Then for every $R>0$ and $\eps>0$ there are $n\in \mathbb{N}$ and $\lambda_0>0$ such that for every $0<\lambda\le \lambda_0$ we have that
    $$\mathbb{P}^{(\lambda)}(\# \text{ Voronoi cells intersecting } \mathcal{B}_R(o)\ge n )<\eps.$$
\end{proposition}
\begin{proof}
    Suppose for a contradiction that there are $R>0$ and $\eps>0$ that do not satisfy the claim.
    It follows that for each $n\in \mathbb{N}$ we find $\lambda_n>0$ such that $\lambda_n\to 0$ and
    \begin{equation}\label{eq:EventZero}
        \mathbb{P}^{(\lambda_n)}(\# \text{ Voronoi cells intersecting } \mathcal{B}_R(o)\ge n )\ge \eps.
    \end{equation}
    By Theorem~\ref{thm:BasicSymmetricSpace} (5), there are $t_n\to \infty$ such that $\vol(\mathcal{B}_{t_n}(o))=\lambda_n^{-1}$.
    By Proposition~\ref{pr:BasicIdealMeasure}, we may without loss of generality assume that $\mu_{t_n}\overset{v}\to \mu$.
    As this sequence will be fixed from now on, we set $\mu_n=\mu_{t_n}$, $\Upsilon_{n}=\Upsilon_{t_n}$, etc.

    Let $\Upsilon$ and $\Upsilon_n$ for every $n\in \mathbb{N}$ be coupled as in Theorem~\ref{thm:Coupling} and define a random variable
    $$T=\min\{k\in \mathbb{N}: X_k(o)>X_1(o)+4R\}.$$
    As $\Upsilon$ is a.s. admissible by Proposition~\ref{pr:BasicPPP}, we have that $|T|<\infty$ a.s.
    In particular, there is $n_0\in \mathbb{N}$ such that
    \begin{equation}\label{eq:EventOne}
        \mathbb{P}(T\le n_0)>1-\eps/2.
    \end{equation}
    By Theorem~\ref{thm:Coupling}, we have that a.s.
    $$\lim_{n\to\infty} X^{(n)}_i\to X_i$$
    for every $i\in \mathbb{N}$.
    In particular, a.s.
    $$\lim_{n\to\infty} X^{(n)}_i(o)\to X_i(o),$$
    for every $1\le i\le n_0$.
    It follows that there is $n_1\ge n_0$ such that
    \begin{equation}\label{eq:EventTwo}
        \mathbb{P}(\forall 1\le i\le n_0 \  |X^{(n_1)}_i(o)-X_i(o)|<R)>1-\eps/2.
    \end{equation}
    On the intersection of the events from \eqref{eq:EventOne} and \eqref{eq:EventTwo}, we have by triangle inequality that
    $$X^{n_1}_{n_0}(o)-X^{n_1}_1(o)\ge X_{n_0}(o)-X_1(o) -|X^{(n_1)}_{n_0}(o)-X_{n_0}(o)|-|X^{(n_1)}_1(o)-X_1(o)|>2R.$$
    Consequently, by monotonicity of $\{X^{n_1}_i(o)\}_{i}$ we conclude that
    $$\mathbb{P}^{(\lambda_{n_1})}(X^{n_1}_{n_1}(o)-X^{n_1}_1(o)>2R)>1-\eps$$
    which reads as
    \begin{equation}\label{eq:EventThree}
        \mathbb{P}^{(\lambda_{n_1})}\left(d_X(o,Y^{(\lambda_{n_1})}_{n_1})-d_X(o,Y^{(\lambda_{n_1})}_1)>2R\right)>1-\eps
    \end{equation}
    using the notation of $Y^{(\lambda_{n_1})}$ back in $X$.
    Finally, on the event from \eqref{eq:EventThree} given any $m\ge n_1$, we have by the triangle inequality that
    $$d_X(x,Y^{(\lambda_{n_1})}_1)<d_X(x,Y^{(\lambda_{n_1})}_m)$$
    for every $x\in \mathcal{B}_R(o)$.
    This implies that, on the event from \eqref{eq:EventThree}, $C^{(\lambda_{n_1})}_m\cap \mathcal{B}_R(o)=\emptyset$ for every $m\ge n_1$, which contradicts the choice of $n_1$ in \eqref{eq:EventZero}.
\end{proof}

\section{Property (T) and long-range order} \label{sec:threshold}

In this section, we prove Theorem \ref{thm:threshold}, which provides the aforementioned criterion for long-range order of sufficiently well-behaved continuum percolation models when there is a suitable action by a group with property~(T). For this purpose, we define the following technical notion.

\begin{definition}\label{def:normal}
    Let $(M,d)$ be a proper geodesic metric space.
    A random closed subset $\mathcal Z$ of $M$ is called \emph{normal} if 
    \begin{itemize}
        \item[{\rm(i)}] the path-connected component of every $x\in M$ in $\mathcal{Z}$ is a closed set.
        \item[{\rm(ii)}] every bounded subset of $M$ intersects finitely many path-connected components of $\mathcal{Z}$.
    \end{itemize}
    We refer to the path-connected components as \emph{clusters}.
\end{definition}

In the setting of Definition \ref{def:normal}, it  makes sense to treat the clusters of a normal random closed subset $\mathcal Z$ as random variables. In particular, $\{x\overset{\mathcal Z}{\longleftrightarrow}y\}$ will denote the event that there exists a path joining $x\in M$ to $y\in M$ in $\mathcal Z$ or, in other words, the clusters of $x$ and $y$ are equal. Similarly, $\{x\overset{\mathcal Z}{\longleftrightarrow}A\}$ will denote the event that there exists a path joining $x\in M$ to some point $y\in A$ in $\mathcal Z$, where $A\subseteq M$ is measurable.
Note that Poisson-Voronoi percolation on a symmetric spaces defines a normal random closed subset by Lemma \ref{lm:closed}.

We now recall a definition of property~(T) suitable for the purposes of this paper: Let $S$ be a set. We say that a map $k\colon S\times S\to\mathbb R$ is a {\em positive definite kernel} if $\sum_{i,j=1}^n \overline{a_i} a_j k(x_i,x_j) \geq 0$ for every $a_1,\ldots,a_n\in \mathbb C$ and $x_1,\ldots,x_n\in S$. 
It is {\em normalized} if $k(x,x)=1$ for all $x\in S$. Similarly, we say that a map $\varphi:G\to\mathbb R$ on a group $G$ is a {\em positive definite function}, if $k_{\varphi}(g,h):= \varphi(g^{-1}h)$ is a positive definite kernel, and is {\em normalized}, if $k_\varphi$ is normalized.

For the purposes of this paper, the following well-known characterization of property (T) may be taken as the definition \cite[Théorème 11]{HV89}: A lcsc group $G$ has {\em property~(T)} if every sequence of continuous, normalized, positive definite functions on $G$ which converges to $1$ uniformly on compact subsets of $G$, converges to $1$ uniformly on $G$.

We are now in a position to state the main result of this section.

\begin{theorem}[Long-range order threshold] \label{thm:threshold}
    Let $G$ be a lcsc group acting continuously and transitively by isometries on a proper geodesic metric space $(M,d)$ and fix some $o\in M$. Suppose that $G$ has property (T). Then for every $R>0$, there exists $p^*<1$ such that every $G$-invariant normal random closed subset $\mathcal Z$ of $M$ with $\mathbb P(o \in \partial \mathcal Z)=0$ satisfies
    \begin{equation*}
        \inf_{x,y\in \mathcal{B}_R(o)} \mathbb P \big( x\overset{\mathcal Z}{\longleftrightarrow} y \big)>p^* \quad \Rightarrow \quad \inf_{x,y\in M} \mathbb P \big( x\overset{\mathcal Z}{\longleftrightarrow} y \big)>0.
    \end{equation*}
\end{theorem}

\medskip

Let us include two remarks, which provide the necessary context.

\begin{remark}
    In the setting of Theorem \ref{thm:threshold}, if
    \begin{equation*}
        \inf_{x,y\in M} \mathbb P \big( x\overset{\mathcal Z}{\longleftrightarrow} y \big)>0,
    \end{equation*}
    we say that $\mathcal Z$ exhibits {\em long-range order}, in accordance with terminology used in percolation theory~\cite{LS99}. Since the quantity
    \begin{equation*}
    \inf_{x,y\in \mathcal B_R(o)} \mathbb P \big( x\overset{\mathcal Z}{\longleftrightarrow} y \big)
    \end{equation*}
    may be viewed as a measurement of the {\em density} of $\mathcal Z$, Theorem \ref{thm:threshold} asserts that, for a family of continuum percolation models, a density exceeding the {\em threshold} $p^*$ implies long-range order. 
\end{remark}

\begin{remark}
    Theorem \ref{thm:threshold} is a continuum percolation analogue of a result about group-invarint percolation on Cayley graphs first noted in \cite{LS99}, and strengthened to a characterization of property~(T) for finitely generated groups in \cite{MR23}. The proof follows along the same lines modulo some necessary modifications for the continuous setting. In particular, the condition that $\mathbb P(o\in\partial Z)=0$ guarantees that the two-point function, i.e.~$\mathbb P^{\mathcal Z}( x\leftrightarrow y)$, is continuous.
\end{remark}

\begin{proof}[Proof of Theorem \ref{thm:threshold}]  
    Let $\mathcal Z$ be a $G$-invariant normal random closed subset of $M$ such that $\mathbb P(o \in \partial \mathcal Z)=0$ and $\mathbb P(o \in \mathcal Z)>0$. Define $\tau:M\times M\to[0,1]$ and $\varphi:G\to[0,1]$ by
    \begin{equation*}
        \tau(x,y) := \mathbb P \big( x\overset{\mathcal Z}{\longleftrightarrow} y \big) := \mathbb P^{\mathcal Z}(x\leftrightarrow y) \quad \mbox{and} \quad \varphi(g):= \tau(o,go).
        \vspace{1mm}
    \end{equation*}
    
    \begin{claim} \label{claim-PDF}
        The function $\varphi$ is continuous and positive definite. In particular,
        \begin{equation*} 
            \varphi'(g):= \varphi(g)/\varphi(e) = \frac{\varphi(g)}{\mathbb P( o\in\mathcal Z)}
        \end{equation*}
        is a continuous, normalized, positive definite function.
    \end{claim}

    \begin{proof}
        This can be checked as in \cite[Lemma 2.2]{MR23}. We include the details for the convenience of the reader. Let us start with the observation that $\tau$ defines a positive definite kernel, which goes back to \cite{AN84}. Let $x_1,\ldots,x_n\in M$ and $a_1,\ldots,a_n\in\mathbb C$. Let $\mathcal C$ denote the set of clusters of $\mathcal Z$. Then
        \begin{align*}
        \sum_{i,j=1}^n \overline{a_i}a_j\tau(x_i,x_j) &= \sum_{i,j} \overline{a_i}a_j \mathbb E^{\mathcal Z} \big[ \mathbf 1_{\{x_i \leftrightarrow x_j\}}\big] = \mathbb E^{\mathcal Z} \biggl[  \ \sum_{i,j} \overline{a_i}a_j \mathbf 1_{\{x_i \leftrightarrow x_j\}} \biggr] \\
        & =  \mathbb E^{\mathcal Z} \biggl[ \ \sum_{C\in\mathcal C} \, \sum_{\{x_i,x_j\} \subset C} \overline{a_i}a_j \biggr] = \mathbb E^{\mathcal Z} \biggl[ \ \sum_{C\in\mathcal C} \, \Bigl| \sum_{x_i \in C} a_i \Bigr|^2 \biggr] \geq 0,
        \end{align*}
        proving positive definiteness. Now since $\mathcal Z$ is $G$-invariant, we have that
        $$
        k_{\varphi}(g,h)=\varphi(g^{-1}h)=\tau(o,g^{-1}ho)=\tau(go,ho).
        $$
        It follows that $k_\varphi$ is a positive definite kernel, i.e.~$\varphi$ is a positive definite function. Since constant multiples of positive definite functions are clearly positive definite, this implies that $\varphi'$ is a normalized, positive definite function.
        
        It remains to show that $\varphi$ is continuous. To see this, let $g,g_1,g_2,\ldots\in G$ with $g_i\to g\in G$. Note that since $G$ acts continuously, $g_io\to go$ and $\eps_i:=d(go,g_io)\to 0$. We first show that $\varphi(g)\geq \limsup_{i\to\infty} \varphi(g_i)$. Note that for every subsequence of $(g_i)_{i=1}^\infty$, there exists a further subsequence $(g_{i_k})_{k=1}^\infty$ for which $d(go,g_{i_k}o)$ is decreasing and hence 
        \begin{equation*}
            \big\{ o \overset{\mathcal Z}{\longleftrightarrow} \mathcal {\overline B}_{\eps_{i_k}}(go)\big\},
        \end{equation*} 
        where $\mathcal {\overline B}_{\eps_{i_k}}(go)$ denotes the closed ball of radius $\eps_{i_k}$ around $go$, are decreasing events.
        Since $\mathcal Z$ is normal we also have that
    \begin{equation*}
        \big\{ o \overset{\mathcal Z}{\longleftrightarrow} go \big\} = \bigcap_{k=1}^\infty \big\{ o \overset{\mathcal Z}{\longleftrightarrow} \mathcal {\overline B}_{\eps_{i_k}}(go) \big\}.
    \end{equation*} 
    It follows that 
    \begin{equation*}
        \varphi(g) = \lim_{k\to\infty} \mathbb P \big( o \overset{\mathcal Z}{\longleftrightarrow} \mathcal {\overline B}_{\eps_{i_k}}(go) \big) \geq \limsup_{k\to\infty} \, \mathbb P \big( o \overset{\mathcal Z}{\longleftrightarrow} g_{i_k} o \big).
    \end{equation*}
    Since the subsequence was arbitrary, we obtain the claimed inequality.
    
    We now show that $\varphi(g)\leq \liminf_{i\to\infty} \varphi(g_i)$, or, equivalently, $\limsup_{i\to\infty} (\varphi(g)-\varphi(g_i)) \leq 0$:
    \begin{align*}
        \limsup_{i\to\infty} \, (\varphi(g)-\varphi(g_i)) & = \limsup_{i\to\infty} \, \Big( \mathbb P^{\mathcal Z}\big( o \leftrightarrow go, o \not\leftrightarrow g_io \big) - \mathbb P^{\mathcal Z}\big( o \not\leftrightarrow go, o \leftrightarrow g_io \big) \Big) \\
        & \leq \limsup_{i\to\infty} \, \ \mathbb P^{\mathcal Z}\big( o \leftrightarrow go, o \not\leftrightarrow g_io \big) \\
        & \leq \mathbb P^{\mathcal Z}\Big( \limsup_{i\to\infty} \, \big\{ o \leftrightarrow go, o \not\leftrightarrow g_io\big\} \Big) \le \mathbb P^{\mathcal Z}(go \in \partial Z) = 0,
    \end{align*}
    which proves the claimed inequality. The proof of Claim \ref{claim-PDF} is thus complete.      \end{proof}

    We now conclude the proof of the theorem. To reach a contradiction, suppose there exist $R>0$ and $G$-invariant normal random closed subsets $\mathcal Z_n$ of $M$, $n\geq1$, such that $\mathbb P(o \in \partial Z_n)=0$, 
    \begin{equation*}
        \inf_{x,y\in \mathcal B_R(o)} \mathbb P \big( x\overset{\mathcal Z_n}{\longleftrightarrow} y \big)>1-1/n \qquad \text{and} \qquad \inf_{x,y\in M} \mathbb P \big( x\overset{\mathcal Z_n}{\longleftrightarrow} y \big)=0.
    \end{equation*}
    Then 
    $$
    \varphi_n'(g) := \frac{\mathbb P \big( o \overset{\mathcal Z_n}{\longleftrightarrow} go \big)}{\mathbb P\big(o\in \mathcal Z_n\big)}
    $$
    defines a continuous, normalized, positive definite function on $G$ for every $n$ by Claim~\ref{claim-PDF}. Since $G$ acts transitively and $\inf_{x,y\in M} \mathbb P^{\mathcal Z_n}( x \leftrightarrow y \big)=0$,
    $\varphi_n$ does not converge to $1$ uniformly on $G$. To obtain the desired contradiction with property~(T), we now show that $\varphi_n'$ converges to $1$ uniformly on compacts. 
    Since 
    $$
    1-1/n < \inf_{x,y\in \mathcal B_R(o)} \mathbb P \big( x\overset{\mathcal Z_n}{\longleftrightarrow} y \big) \le \mathbb P(o\in\mathcal Z_n) \le 1,
    $$
    it suffices to show that $\varphi_n(g):=\mathbb P \big( o \overset{\mathcal Z_n}{\longleftrightarrow} go \big)$
    converges to $1$ uniformly on compacts. To see this, let $H\subset G$ be compact. Since $g\mapsto d(o,go)$ is continuous, $Ho\subset M$ is contained in a ball $\mathcal B_S(o)$ of large enough radius. Let $g\in H$. Then $d(o,go) < S$ and it follows that we may choose $o=y_1,\ldots,y_N=go$ on the geodesic from $o$ to $go$ such that $d(y_i,y_{i+1})<R$ for all $i=1,\ldots,N-1$ and $N\le S/R$. Since $G$ acts transitively by isometries and $\mathcal Z_n$ is $G$-invariant,
    \begin{equation*}
        \mathbb P\big( y_i \overset{\mathcal Z_n}{\longleftrightarrow} y_{i+1} \big) > 1-1/n
    \end{equation*}
    for every $i=1,\ldots,N-1$. It follows that
    \begin{equation*}
        \varphi(g) \geq \mathbb P\bigg( \bigcap_{i=1}^{N-1} \{y_i \overset{\mathcal Z_n}{\longleftrightarrow} y_{i+1}\} \bigg) \geq 1- (N-1)/n.
    \end{equation*}
    Since $g\in H$ was arbitrary, $\varphi_n \to 1$ as $n\to\infty$ uniformly on $H$. The proof of Theorem \ref{thm:threshold} is thus complete.
\end{proof}

\section{Long-range order and uniqueness} \label{sec:LRO}

In this section, we prove Theorem \ref{thm:LRO} which asserts that the uniqueness phase of Voronoi percolation on a symmetric space is characterized by long-range order.

\begin{theorem}[Long-range order implies uniqueness] \label{thm:LRO}
    Let $G$ be a non-compact connected semisimple real Lie group and $(X,d_X)$ be its symmetric space. Then 
    \begin{equation*}
        p_u(\lambda) = \inf \Big\{ p \, : \, \inf_{x,y\in X} \mathbb P_{p}^{(\lambda)}( x \leftrightarrow y )>0 \Big\}.
    \end{equation*}
\end{theorem}

\medskip

The idea behind the proof is as follows. Consider the {\em Delaunay graph} associated to Poisson-Voronoi percolation, i.e.~the graph defined by declaring the cells to be vertices and by declaring edges between every pair of cells sharing a boundary.
Let us insert an additional point at the origin and declare its cell to be the root. With this definition, we obtain an extremal unimodular random graph. It follows that Bernoulli percolation on this graph has an almost surely constant number of infinite clusters, which is $0$, $1$ or $\infty$. The important observation, that follows from Lemma~\ref{lm:closed} and Lemma~\ref{lm:BasicVoronoi}, is that the number of infinite clusters corresponds to the number of unbounded clusters in Poisson-Voronoi percolation. To prove Theorem \ref{thm:LRO}, it thus remains to show that long-range order for Poisson-Voronoi percolation implies that the number of infinite clusters cannot be $0$ and cannot be $\infty$. The first possibility is easy to rule out. To rule out the second possibility, we use a version of a celebrated method from discrete percolation theory due to Lyons and Schramm \cite{LS99}. This version is described below.

In the discrete setting, the method of \cite{LS99} proceeds by considering  {\em cluster frequencies} which associate to each cluster the asymptotic density of visits by an independent random walk. If the clusters are {\em indistinguishable} in the sense of \cite{LS99}, then the presence of infinitely many infinite clusters implies that each cluster has frequency equal to $0$. This is easily seen to contradict long-range order because the latter implies that the expected frequency of, say, the cluster of the origin is positive. For the purposes of this paper, it would be desirable to use this method for Bernoulli percolation on the Delaunay graph. In fact, the method has been developed for percolation clusters on a unimodular random graph $(G,o)$ in the literature. Here, frequencies are measured according to an auxiliary simple random walk defined {\em conditionally} on $(G,o)$, see~\cite[Section 6]{AL07}. Since Bernoulli percolation clusters are indistinguishable in the appropriate sense \cite[Theorem 6.16]{AL07}, every cluster again has frequency equal to $0$ whenever there are infinitely many infinite clusters, cf.~\cite[Theorem 6.15]{AL07}. It seems intuitively obvious that this property contradicts long-range order. However, in our setting, the following difficulty arises: The random walk on the Delaunay graph and the configuration of Poisson-Voronoi percolation both depend on the underlying Poisson point process, hence are not independent. Therefore the expected frequency of, say, the cluster of the origin can not be computed in a straightforward way, as was the case in the discrete setting. To circumvent this problem, we follow another approach to cluster frequencies, which was developed for the discrete setting in \cite{HJ06}.

We point out that a different version of the Delaunay graph, where edges are declared only between pairs of cells which share a boundary of co-dimension $1$, has been considered by Benjamini, Paquette and Pfeffer~\cite{BPP18}. However, it will be clear from the discussion below that our definition is appropriate for analyzing percolation (see Section \ref{sec:Delaunay} for a comparison).

Let us also report the following description of the basic phase transition in Poisson-Voronoi percolation. Recall the definition of the critical parameters $p_c(\lambda)$ and $p_u(\lambda)$ of Poisson-Voronoi percolation from \eqref{equ-pc} and \eqref{equ-pu}.

\begin{corollary}[Phase transition for Poisson-Voronoi percolation] \label{cor:Phases} Let $G$ be a non-compact connected semisimple real Lie group and $(X,d_X)$ be its symmetric space. Fix $\lambda>0$. Then 
\begin{itemize}
\item[{\rm (1)}] for every $p\in [0,p_c(\lambda))$, $\omega_p^{(\lambda)}$ does not have an unbounded cluster a.s.
\item[{\rm (2)}] for every $p\in (p_c(\lambda),p_u(\lambda))$, $\omega_p^{(\lambda)}$ has infinitely many unbounded clusters a.s.
\item[{\rm (3)}] for every $p\in (p_u(\lambda),1]$, $\omega_p^{(\lambda)}$ has a unique unbounded clusters a.s.
\end{itemize}
\end{corollary}

While this result seems to be known to experts, we could not find a suitable reference and therefore have included the short proof based on the aforementioned link with Bernoulli percolation on the Delaunay graph. See Section \ref{sec:closing} for open questions related to this phase transition.

The rest of this section is devoted to the proof of Theorem~\ref{thm:LRO} and Corollary~\ref{cor:Phases}.

\subsection{The Delaunay graph}\label{sec:Delaunay} In this section, we recall relevant background about unimodular random graphs, define the Delaunay graph and lay out its fundamentals.

\subsubsection{Unimodular random graphs}

Let us start by recalling the definition and basic properties of unimodular random graphs, which were introduced in \cite{BS01}. For more details, see~\cite{AL07}.

A {\em rooted graph} $(G,\rho)$ is a simple undirected countable locally finite graph $G$ with a distinguished vertex $\rho$, which is called the {\em root}. Let $\mathcal G_{\bullet}$ denote the set of connected rooted graphs modulo rooted isomorphisms. To lighten the notation, we use the same notation $(G,\rho)$ for the rooted graph and its equivalence class. We equip $\mathcal G_{\bullet}$ with the {\em local metric} $d_{\rm LOC} \colon \mathcal G_\bullet \times \mathcal G_\bullet \to [0,1]$ defined by $d_{\rm LOC}((G,\rho),(G',\rho')) := 1/(1+r)$, where $r := \sup \{ n\ge 0 : \mathcal B_n(\rho) \cong \mathcal B_n(\rho') \}.$ 
It is well-known that $(\mathcal G_{\bullet}, d_{\rm LOC})$ is a complete separable metric space. We equip it with its Borel $\sigma$-field. A {\em random rooted graph} is a $\mathcal G_{\bullet}$-valued random variable. We denote by $\mathcal P(\mathcal G_{\bullet})$ the space of Borel probability measures on $\mathcal G_\bullet$. Let $\mathcal G_{\bullet \bullet}$ denote the analogue of the space~$\mathcal G_{\bullet}$ with two distinguished roots. The law $\mathcal L$ of a random rooted graph $(G,\rho)$ is {\em unimodular} if 
\begin{equation} \label{equ-MTP}
\mathbb E \bigg[ \sum_{x\in V(G)} f(G,\rho,x) \bigg] = \mathbb E \bigg[ \sum_{x\in V(G)} f(G,x,\rho) \bigg]
\end{equation}
for every measurable function $f\colon \mathcal G_{\bullet \bullet} \to[0,\infty]$. In this case, we also say that $(G,\rho)$ is a {\em unimodular random graph} (URG). Equation \eqref{equ-MTP} is called the {\em Mass Transport Principle} (MTP) because it asserts that the expected mass sent out by the root equals the expected mass received by the root. The class of unimodular probability measures on $\mathcal G_{\bullet}$ is convex. A unimodular probability measure is called {\em extremal}, if it can not be written as a convex combination of other unimodular probability measures. These measures admit the following description in terms of the {\em invariant} $\sigma$-field $\mathcal I$, which is the $\sigma$-field of Borel measurable subsets of $\mathcal G_{\bullet}$ which are invariant under non-rooted isomorphisms.

\begin{theorem}[Extremality] \label{thm:Extremal} Let $\mathcal L\in\mathcal P(\mathcal G_{\bullet})$ be unimodular. Then $\mathcal L$ is extremal if and only if $\mathcal I$ is $\mathcal L$-trivial, i.e.~$\mathcal L(A)\in\{0,1\}$ for all $A\in \mathcal I$.
\end{theorem}

\begin{proof} See \cite[Theorem 4.7]{AL07}. \end{proof}

The definition of unimodular random rooted graphs extends to {\em networks}, i.e.~graphs with additional marks on edges and vertices. More precisely, let $\Xi$ be a complete separable metric space of marks and let $\psi$ be an assignment of marks to edges and vertices. Then a random rooted network $(G,\rho,\psi)$ is {\em unimodular} if 
\begin{equation} \label{equ-MTPnetwork}
\mathbb E \bigg[ \sum_{x\in V(G)} f(G,\rho,x,\psi) \bigg] = \mathbb E \bigg[ \sum_{x\in V(G)} f(G,x,\rho,\psi) \bigg]
\end{equation}
for every measurable, non-negative function $f$ of isomorphism classes of bi-rooted (with ordered roots) networks, see~\cite{AL07} for details.

\medskip

{\noindent\bf Bernoulli percolation.} Let $(G,\rho)$ be a unimodular random rooted graph. A {\em (site) percolation} of $(G,\rho)$ is a unimodular random rooted network with $\{0,1\}$-valued marks on the vertices such that the random rooted graph obtained by forgetting the marks has the same law as $(G,\rho)$. Deleting from this random rooted graph all vertices whose mark is $0$ yields a random subgraph of $G$, which we call the {\em percolation configuration}. Its connected components are called {\em clusters}.

For the purposes of this paper, we focus on {\em $p$-Bernoulli site percolation}, or simply {\em Bernoulli percolation} which is defined by deleting each vertex independently with probability $p\in(0,1)$. The main properties are collected below.

\begin{proposition}[Bernoulli percolation on extremal URGs] \label{prop:Bernoulli} Let $(G,\rho)$ be an extremal unimodular random graph and let $G[p]$ denote the configuration of $p$-Bernoulli percolation on $(G,\rho)$. Then the following hold:
\begin{itemize}
\item[{\rm(1)}] The number of infinite clusters in $G[p]$ is almost surely constant and $0$, $1$ or $\infty$.
\item[{\rm(2)}] There exists a constant $p_c$ such that for any $p>p_c$, $G[p]$ has an infinite cluster almost surely, while for any $p<p_c$, there is almost surely no infinite cluster.
\item[{\rm(3)}] There exists a constant $p_u$ such that for any $p>p_u$, $G[p]$ has a unique infinite cluster almost surely, while for any $p<p_u$, there is almost surely not a unique infinite cluster.
\end{itemize}
\end{proposition}

\begin{proof} See \cite[Corollary 6.9]{AL07}, \cite[Section 6]{AL07} and \cite[Theorem 6.7]{AL07}, where these results are stated for bond percolation; the proofs for site percolation are similar. \end{proof}

\subsubsection{The Delaunay graph} We now define the Delaunay graph, which is the canonical URG associated to Poisson-Voronoi percolation, and discuss two useful refinements. We also compare our definition with another definition from the literature.

\medskip

{\bf\noindent The unrooted Delaunay graph.} Let $X$ be a symmetric space and let $Y$ be a point process with associated Voronoi diagram ${\rm Vor}(Y)$ defined similarly to \eqref{equ-Voronoi}. The {\em Delaunay graph} $\mathcal G(Y)$ associated to the point process $Y$ is the graph with vertex set ${\rm Vor}(Y)$ and edges between every pair of vertices whose corresponding cells have non-empty intersection.

It will be useful to consider the following two refinements of the Delaunay graph: Let  $\mathbf Y$ be the point process $Y$ together with iid ${\rm Unif}[0,1]$-marks. The {\em embedded Delaunay graph} associated to the point process $Y$ is the graph $\mathcal G(Y)$ together with the assignment $\psi(Y)$ of marks, which marks each vertex by the location of its nucleus and each edge by the midpoint of the geodesic between its endpoints. Similarly, the {\em embedded Delaunay graph with labels} associated to the marked point process $\mathbf Y$ is the graph $\mathcal G(Y)$ together with the assignment $\Psi(\mathbf Y)=(\psi(Y),\psi'(\mathbf Y))$ of marks, where $\psi'$ additionally marks each vertex by the mark of its nucleus in $\mathbf Y$.

\medskip

{\bf\noindent The Delaunay graph.} Let $X$ be a symmetric space and let  $\mathbf Y$ be a point process $Y$ together with iid ${\rm Unif}[0,1]$-marks. We may, and will, treat each of the Delaunay graphs associated to~$\mathbf Y$ as a random rooted network by inserting a point at the origin $o$ with an independent ${\rm Unif}[0,1]$-mark, and declaring its cell to be the root.

More precisely, consider the point processes $Y_0:=Y\cup\{o\}$ and $\mathbf Y_0:=\mathbf Y\cup\{\mathbf o\}$, where $\mathbf o=(o,Z)$ denotes the origin equipped with an independent ${\rm Unif}[0,1]$-mark. Let $\rho(Y_0)$ denote the cell of the origin in ${\rm Vor}(Y_0)$. We will subsequently work with the random rooted graph $(\mathcal G(Y_0),\rho(Y_0))$ and the random rooted networks $(\mathcal G(Y_0),\rho(Y_0),\psi(Y_0))$ and $(\mathcal G(Y_0),\rho(Y_0),\Psi(\mathbf Y_0))$.

In Section \ref{sec:MTP} below, we recall the well-known fact that the Delaunay graph obtained in this way is unimodular and extremal, show that the embedded Delaunay graph satisfies a MTP for functions which additionally depend on the marks in a natural way and extend this statement to the embedded Delaunay graph with labels. The first fact suffices to establish the basic phase transition of Poisson-Voronoi percolation via Proposition~\ref{prop:Bernoulli}; the other facts will be needed in the proof of Theorem~\ref{thm:LRO}.

\medskip

{\bf\noindent Comparison with the Delaunay graph defined in \cite{BPP18}.} In two recent papers~\cite{BPP18,P18}, probabilistic properties (amenability, anchored amenability and random walk speed) of Poisson-Voronoi tessellations in symmetric spaces are studied. In these works, the embedded Delaunay graph is defined by declaring an edge only between those pairs of cells which share a boundary of co-dimension $1$. Let us denote this graph by $\mathcal D$. In \cite[Proposition 1.6]{BPP18}, reversibility of the degree-biased law of $\mathcal D$ and the fact that every vertex has finite expected degree are shown. This implies that $\mathcal D$ is unimodular by a classical argument, see e.g.~\cite[Proposition 2.5]{BC12}.

For the purpose of analyzing percolation, it is clear that cells have to be considered adjacent if they have non-trivial intersection. With this definition, there could a priori be more edges adjacent at each vertex than in $\mathcal{D}$. The following lemma shows that this does not cause too many issue.

\begin{lemma}[Finite expected degree] \label{lm:FinExpDegree}
Let $G$ be a non-compact connected semisimple real Lie group and $(X,d_X)$ be its symmetric space.
Then there is $\alpha\ge 1$ such that if $Y$ is a Poisson point process on $X$ with intensity $\lambda \, {\rm vol}$ for some $\lambda>0$, $Y_0=Y \cup \{o\}$ and ${\rm Vor}(Y_0)$ is the associated Voronoi diagram, then $\mathbb E[{\rm deg}_{\mathcal G}(\rho)] = O(\lambda^{-\alpha})$ as $\lambda\to0$,
where $\mathcal G$ denote the Delaunay graph associated to ${\rm Vor}(Y_0)$. In particular, $\mathcal G$ is locally finite a.s.
\end{lemma}

We leave it as an open problem whether $\alpha=1$, see Remark \ref{rm:DegreeAsymp}.

\begin{proof} This follows along similar lines as the proof of \cite[Proposition 1.6]{BPP18}.
    For $y\in Y$, let $\mathcal G^{(y)}$ denote the Delaunay graph of $Y_{o,y}:= Y\cup\{o,y\}$. The Mecke equation \cite[Theorem 4.1]{LP17} applied with
$$
F : \mathbf M(X) \times X \to [0,1] \, , \, F\big(\eta, y\big) := \mathbf 1\{y\in \eta\}\mathbf 1\{o,y \, \, \text{are connected in} \, \, \mathcal G^{(y)}\}
$$
yields 
\begin{align}
\mathbb E[{\rm deg}_{\mathcal G}(\rho)] & = \mathbb E\bigg[ \sum_{y\in Y} \mathbf1\{ o,y \, \, \text{are connected in} \, \, \mathcal G^{(y)} \}\bigg] \nonumber \\ 
& = \lambda \int_X \mathbb E\big[ \mathbf1\{ o,y \, \, \text{are connected in} \, \, \mathcal G^{(y)} \} \big] \, {\rm vol}(dy). \label{equ:ExpectedDegree}
\end{align}
Fix $y\in Y$ and set $r:=d_X(o,y)$. We claim that 
\begin{equation}\label{equ:ProbConnected}
\mathbb P\big( o,y \, \, \text{are connected in} \, \, \mathcal G^{(y)} \big) \le 2 f(r) \exp(-\lambda f(r/4)),
\end{equation}
where $f(t)=\vol(\mathcal{B}_t(o))$.
This may be seen by the same argument as in the proof of Lemma~\ref{lm:BasicVoronoi}. Indeed, let $C_x^{(y)}$ denote the Voronoi cell of $x\in X$ in ${\rm Vor}(Y_{o,y})$, then
 $$
 \mathbb P\big( o,y \, \, \text{connected in} \, \, \mathcal G^{(y)} \big) = \mathbb P ( C_o^{(y)} \cap C_y^{(y)} \neq \emptyset ) \le \mathbb P ( C_o^{(y)} \not\subset \mathcal B_{r/2}(o)) + \mathbb P ( C_y^{(y)} \not\subset \mathcal B_{r/2}(y))
 $$
and $C_o^{(y)}\subseteq C_o$.
As $f$ is continuous increasing by Theorem~\ref{thm:BasicSymmetricSpace}~(5), we get 
$$
\mathbb E[{\rm deg}_{\mathcal G}(\rho)] \le 2\lambda \int_0^\infty  f(r) e^{-\lambda f(r/4)} f'(r) \, d r.
$$
Observe that by Theorem~\ref{thm:BasicSymmetricSpace}~(6), there are $r_0>0$ and $\beta\in(0,1)$ such that $f(r/4)>f(r)^\beta$ for every $r\ge r_0$.
We obtain, with $\gamma:=2/\beta-1>1$, that
\begin{align*}
\int_{r_0}^\infty  f(r) e^{-\lambda f(r/4)} f'(r) \, d r & \le  \int_0^\infty  f(r) e^{-\lambda f(r)^\beta} f'(r) \, d r = \int_0^\infty xe^{-\lambda x^\beta} \, d x \\
& = \beta^{-1} \int_0^\infty x^{2-\beta} e^{-\lambda x^\beta} \beta x^{\beta-1}  \, d x \\
& = \beta^{-1} \int_0^\infty x^\gamma e^{-\lambda x} \, dx = (\beta\lambda)^{-1} \, \mathbb E[ Z^\gamma],
\end{align*}
where $Z\sim{\rm Exp}(\lambda)$.
Since with $k:=\lceil \gamma \rceil$, $\mathbb E[ Z^\gamma] \le 1 + \mathbb E[ Z^k] = 1+k!/\lambda^k$, and $\int_{0}^{r_0}  f(r) e^{-\lambda f(r/4)} f'(r) \, d r\le \vol(\mathcal{B}_{r_0}(o))^2$ for every $\lambda>0$, we obtain that there exists $\alpha\ge 1$ such that
\begin{equation} \label{equ:DegreeAsymp}
\mathbb E[{\rm deg}_{\mathcal G}(\rho)] = O(\lambda^{-\alpha}) \qquad \text{as} \, \, \lambda\to0
\end{equation}
as desired.
\end{proof}

\subsection{Unimodularity and Mass Transport Principles} \label{sec:MTP} We now provide the main properties of the Delaunay and embedded Delaunay graphs which will be based on corresponding properties of Poisson point processes.

\subsubsection{Palm distribution and MTP for the Poisson point process} In this section, we recall two fundamental results about Poisson point processes which fall within the scope of Palm theory. Roughly speaking, a Palm version of a stationary point process is the point process conditioned to have a point at the origin.  For instance, it is well-known that if $Y$ (resp.~$\mathbf Y$) is a Poisson point process on $X$ with intensity $c \, {\rm vol}$ for $c>0$ (resp.~with intensity $c \, {\rm vol}$ and iid ${\rm Unif}[0,1]$-marks), then a Palm version is given by $Y_0=Y\cup\{o\}$ (resp.~$\mathbf Y_0=\mathbf Y\cup\{\mathbf o\}$, where $\mathbf o=(o,Z)$ is the origin equipped with an independent ${\rm Unif}[0,1]$-mark). Since in this paper we will work with Poisson point processes, the formalism of Palm theory is not needed. We thus refrain from developing the formalism here and instead use the terminology {\em Palm version} to refer to the explicit point processes described above. We refer to \cite{DV08,LP17} for background.

We now recall ergodicity of the Palm version, which will be used throughout this section. In the next result, the action of $G$ on marked configurations $\overline \eta \in\mathbf M(X\times[0,1])$ is the action induced by the diagonal action on $X\times[0,1]$, where $G$ acts on the $X$-coordinate as before and leaves the $[0,1]$-coordinate as is.

\begin{lemma}[Ergodicity of the Palm version] \label{lm:Palm} Let $G$ be a non-compact connected semisimple real Lie group and $(X,d_X)$ be its symmetric space. Let $\mathbf Y$ be a Poisson point process on $X$ with intensity $c \, {\rm vol}$ for some $c>0$ and with iid ${\rm Unif}[0,1]$-marks. Let $\mathbf Y_0=\mathbf Y \cup \{\mathbf o\}$, where $\mathbf o$ is equipped with an independent ${\rm Unif}[0,1]$-mark. Then the following hold:
\begin{itemize} 
\item[{\rm(1)}] Let $A$ be a $G$-invariant event and let $A_0$ be the restriction of $A$ to configurations containing the origin. Then $\mathbb P(\mathbf Y_0\in A_0)=\mathbb P(\mathbf Y\in A)\in \{0,1\}$.
\item[{\rm(2)}] Let $B_0$ be an event such that $B_0=B'$, where $B'$ is the restriction of the event $B:=GB_0$ to configurations that contain the origin.
Then $\mathbb P(\mathbf Y_0\in B_0)=\mathbb P(\mathbf Y\in B)$. In particular, $\mathbb P(\mathbf Y_0\in B_0)\in\{0,1\}$.
\end{itemize}
The same conclusions hold when $\mathbf Y$ is replaced with the unmarked Poisson point process $Y$ and $\mathbf Y_0$ is replaced with $Y_0=Y\cup\{o\}$.
\end{lemma}

\begin{proof} This is known to hold in greater generality, see e.g.~\cite[Section 3.4]{AM22}. We include a direct proof based on properties of the Poisson point process for the reader's convenience. We shall only prove the statement for $\mathbf Y$; the proofs in the unmarked case are simpler.

(1).
We express $\mathbb P(\mathbf Y\in A)$ in terms of Palm probabilities using the Mecke equation.
By the Marking Theorem (see \cite[Theorem 5.6]{LP17}), $\mathbf Y$ is a Poisson point process on $X\times[0,1]$ with intensity measure $c\, {\rm vol} \otimes {\rm Unif}[0,1]$. The Mecke equation (see \cite[Theorem 4.1]{LP17}) applied with
$$
F : \mathbf M(X\times[0,1]) \times (X\times[0,1]) \to [0,1] \, , \, F\big(\overline \eta, (y,z)\big) := \mathbf 1_{\{\overline \eta \in A\}}
$$
yields 
\begin{equation} \label{equ:Palm1}
\mathbb E \bigg[ \sum_{(y,z)\in \mathbf Y} F\big(\mathbf Y, (y,z)\big) \bigg] = c \int_{X\times[0,1]} \mathbb E\big[ F\big(\mathbf Y \cup \{(y,z)\},(y,z)\big) \big] \ {\rm vol}\otimes{\rm Unif}[0,1](dy,dz).
\end{equation}
For $y\in X$, let $\mathbf y$ denote the point $y$ equipped with an independent ${\rm Unif}[0,1]$-mark and choose $g\in G$ such that $g^{-1}y=o$. Then 
$$
g^{-1}( \mathbf Y \cup \{\mathbf y\} ) \overset{(d)}{=} \mathbf Y \cup \mathbf o
$$ 
by $G$-invariance of $\mathbf Y$. In particular, $\mathbb P( \mathbf Y \cup \{\mathbf y \} \in A )= \mathbb P(\mathbf Y \cup \{\mathbf o \} \in A_0)$ by $G$-invariance of $A$ and the fact that $\mathbf Y_0$ is supported on configurations containing the origin. By Fubini's theorem (applied twice), we thus obtain 
\begin{align*}
& \int_{X\times[0,1]} \mathbb E\big[ F\big(\mathbf Y \cup \{(y,z)\},(y,z)\big) \big] \ {\rm vol}\otimes{\rm Unif}[0,1](dy,dz) \\
& \qquad \qquad = \int_X \bigg( \int_0^1 \mathbb E\big[ F\big(\mathbf Y \cup \{(y,z)\},(y,z)\big) \big] \, dz \bigg) \, {\rm vol}(dy) \\
& \qquad \qquad = \int_X \mathbb E\big[ F\big(\mathbf Y \cup \{\mathbf y\},\mathbf y\big) \big] \, {\rm vol}(dy) \\
& \qquad \qquad = \int_X \mathbb P\big( \mathbf Y \cup \{\mathbf y\} \in A \big) \, {\rm vol}(dy) \\
& \qquad \qquad = \int_X \mathbb P\big( \mathbf Y \cup \{ \mathbf o\} \in A_0 \big) \, {\rm vol}(dy).
\end{align*}
By ergodicity of $\mathbf Y$, we have $\mathbb P(\mathbf Y\in A)\in\{0,1\}$. If $\mathbb P(\mathbf Y\in A)=0$, then the left-hand side in \eqref{equ:Palm1} is zero, hence $\mathbb P\big( \mathbf Y \cup \{\mathbf o\} \in A_0 \big)$. On the other hand, if $\mathbb P(\mathbf Y\in A)=1$, then the same argument shows that $\mathbb P\big( \mathbf Y \cup \{ \mathbf o\} \in A_0^c \big)=0$, which proves the claim.
 
(2). Note that $B:=GB_0$ is $G$-invariant, hence the same argument as above applies to $B'$, which is equal to $B_0$ by assumption. The additional statement follows from ergodicity of $\mathbf Y$.
\end{proof}

We now formulate the Mass Transport Principle for the Poisson point process in the notation of Section \ref{section PVP}. As above, $\tau\in{\rm Isom(X)}$ acts on marked configurations $\overline\eta \in\mathbf M(X\times[0,1])$ by the action induced by the diagonal action on $X\times[0,1]$ where $\tau$ acts on the $X$-coordinate as before and leaves the $[0,1]$-coordinate as is.

\begin{proposition}[MTP for the marked Poisson point process] \label{prop:MTPPPP}
Let $G$ be a non-compact connected semisimple real Lie group and $(X,d_X)$ be its symmetric space.
Let ${\bf Y}^{(\lambda)}$ be the marked Poisson point process on $X$ with intensity $\lambda \, {\rm vol}$ for some $\lambda>0$.
Let ${\bf Y}^{(\lambda)}_o={\bf Y}^{(\lambda)} \cup \{\mathbf o\}$, where~$\mathbf o$ is equipped with an independent ${\rm Unif}[0,1]$-mark.
Then
\begin{equation*} 
\mathbb E \bigg[ \sum_{i\in \mathbb{N}} F(o,Y^{(\lambda)}_i,{\bf Y}_o^{(\lambda)}) \bigg] = \mathbb E \bigg[ \sum_{i\in \mathbb{N}} F(Y^{(\lambda)}_i,o,{\bf Y}_o^{(\lambda)}) \bigg]
\end{equation*}
for every measurable $F : X\times X\times \mathbf M(X\times[0,1]) \to [0,\infty]$ such that 
$$
F(\tau(x),\tau(y),\tau \circ \overline \eta) = F(x,y, \overline \eta)
$$
for every $\overline\eta \in\mathbf M(X\times[0,1])$ and every $x,y\in X$ and $\tau\in{\rm Isom}(X)$ which interchanges $x$ and $y$.
\end{proposition}

\begin{proof} This will follow from Proposition~\ref{prop:Delaunay}~(3),
which provides the MTP which will be used in the sequel, as the marked graph appearing there fully determines the information of the marked Poisson point process appearing here.
\end{proof}

\subsubsection{Unimodularity of the Delaunay and embedded Delaunay graphs} We now specialize to the setting of the Delaunay graph.

For $\tau\in{\rm Isom}(X)$ and a bi-rooted embedded graph $(H,u,v,\varphi)$, that is a bi-rooted graph $(H,u,v)$ with an assignment $\varphi$ of marks in $X$ to vertices and edges, define $\tau \circ \varphi$ pointwise. Similarly, for a bi-rooted embedded graph with labels $(H,u,v,\Phi)$, that is a bi-rooted graph $(H,u,v)$ with $\Phi=(\varphi,\varphi')$ consisting of an assignment $\varphi$ of marks in $X$ to vertices and edges and an assignment $\varphi'$ of $[0,1]$-marks to vertices, define $\tau\circ \Phi=(\tau \circ \varphi,\varphi')$. Note that if $(\mathcal G(Y),\psi(Y))$ is the embedded Delaunay graph associated to a point process $Y$, then $(\mathcal G(Y),\tau\circ\psi(Y))$ is equal (as a network) to the embedded Delaunay graph associated to $\tau(Y)$.

\begin{proposition}[Unimodularity of the embedded Delaunay graph] \label{prop:Delaunay} Let $G$ be a non-compact connected semisimple real Lie group and $(X,d_X)$ be its symmetric space. Let $\mathbf Y$ be a Poisson point process $Y$ on $X$ with intensity $c \, {\rm vol}$ for some $c>0$ and with iid ${\rm Unif}[0,1]$-marks. Let $Y_0=Y\cup\{o\}$ and $\mathbf Y_0=\mathbf Y \cup \{\mathbf o\}$, where $\mathbf o$ is the origin equipped with an independent ${\rm Unif}[0,1]$-mark. Let $\mathcal G$ denote the Delaunay graph associated to ${\rm Vor}(Y_0)$ and let $\rho$ denote the vertex corresponding to the cell of the origin. Then the following hold:
\begin{itemize}
\item[{\rm(1)}] The Delaunay graph $(\mathcal G,\rho)$ defines a unimodular random graph. 
\item[{\rm(2)}] Let $\psi$ denote the marking of the embedded Delaunay graph. Then $(\mathcal G,\rho,\psi)$ satisfies 
\begin{equation}\label{equ:MTPEmbeddedDelaunay}
\mathbb E \bigg[ \sum_{x\in V(\mathcal G)} f(\mathcal G,\rho,x,\psi) \bigg] = \mathbb E \bigg[ \sum_{x\in V(\mathcal G)} f(\mathcal G,x,\rho,\psi) \bigg]
\end{equation}
for every non-negative measurable function $f$ of bi-rooted embedded graphs $(H,u,v,\varphi)$ with the additional property that
\begin{equation} \label{equ:DiagonalInvariance1}
f(H,u,v,\tau \circ \varphi)= f(H,u,v,\varphi)
\end{equation}
for every isometry $\tau$ of $X$ which interchanges $\varphi(u)$ and $\varphi(v)$. 
\item[{\rm(3)}] Let $\Psi$ denote the marking of the embedded Delaunay graph with labels. Then $(\mathcal G,\rho,\Psi)$ satisfies 
\begin{equation}\label{equ:MTPEmbeddedDelaunayLabels}
\mathbb E \bigg[ \sum_{x\in V(\mathcal G)} f(\mathcal G,\rho,x,\Psi) \bigg] = \mathbb E \bigg[ \sum_{x\in V(\mathcal G)} f(\mathcal G,x,\rho,\Psi) \bigg]
\end{equation}
for every non-negative measurable function $f$ of bi-rooted embedded labeled graphs $(H,u,v,\Phi=(\varphi,\varphi'))$ with the additional property that 
\begin{equation} \label{equ:DiagonalInvariance}
f(H,u,v,\tau \circ \Phi)= f(H,u,v,\Phi)
\end{equation} 
for every isometry $\tau$ of $X$ which interchanges $\varphi(u)$ and $\varphi(v)$. 
\end{itemize}
\end{proposition}

The Mass Transport Principle~\eqref{equ:MTPEmbeddedDelaunay} may be interpreted as unimodularity of the embedded random graph. It is essentially due to \cite{BPP18}, who proved it for the slightly different Delaunay graph described in Section \ref{sec:Delaunay}. Similarly to~\cite{BPP18}, the key property which yields \eqref{equ:MTPEmbeddedDelaunay}, and \eqref{equ:MTPEmbeddedDelaunayLabels} as well, is the existence of an {\em involutive isometry} at every point $x\in X$, i.e.~the existence of an isometry which fixes $x$ and reverses all geodesics through $x$.  We note that among Riemannian manifolds, this property is particular to symmetric spaces.

We remark that even though the isometry $\tau$ in Proposition~\ref{prop:Delaunay} interchanges $\varphi(u)$ and $\varphi(v)$, the vertices are not swapped in \eqref{equ:DiagonalInvariance1} and \eqref{equ:DiagonalInvariance} as $(H,u,v)$, the isomorphism type of the bi-rooted Delauney graph, does not change when applying $\tau$, c.f. Proposition~\ref{prop:MTPPPP}.

\begin{proof}[Proof of Proposition \ref{prop:Delaunay}] First note that $\mathcal G$ is locally finite almost surely by Lemma \ref{lm:FinExpDegree}. 

(1).~The proof of unimodularity of the Delaunay graph is standard, see e.g.~\cite[Lemma 7.12]{CL16}. Below, we adapt this proof to show \eqref{equ:MTPEmbeddedDelaunay}. 

(2).~Let $f$ be a non-negative measurable function of bi-rooted embedded graphs satisfying~\eqref{equ:DiagonalInvariance1}. For $\eta\in \mathbf M(X)$, denote by $\mathcal G(\eta)$ the Delaunay graph defined using $\eta$, let $\psi(\eta)$ denote the marking for the embedded Delaunay graph and let $v_y$ denote the vertex corresponding to the Voronoi cell of $y\in \eta$. Define the measurable function 
$$
F\colon \mathbf M(X)\times X\to[0,\infty) \ , \ F(\eta,y) := \mathbf 1_{y\in \eta} \, f\big(\mathcal G(\eta \cup \{o\}),v_o,v_y,\psi(\eta \cup \{o\})\big).
$$
By the Mecke equation (see \cite[Theorem 4.1]{LP17})
\begin{equation*}
\mathbb E \bigg[ \sum_{y\in Y} F(Y,y) \bigg] = c \int_X \mathbb E [ F(Y \cup \{y\},y) ] \, {\rm vol}(dy),
\end{equation*}
which may be rewritten as
\begin{equation*}
\mathbb E \bigg[ \sum_{x\in V(\mathcal G) \setminus\{\rho\}} f(\mathcal G,\rho,x,\psi) \bigg] = c \int_X \mathbb E \big[ f\big(\mathcal G(Y \cup \{o,y\}),v_o,v_y,\psi(Y \cup \{o,y\}) \big) \big] \, {\rm vol}(dy).
\end{equation*}
By similar reasoning, we have that
\begin{equation*}
\mathbb E \bigg[ \sum_{x\in V(\mathcal G) \setminus\{\rho\}} f(\mathcal G,x,\rho,\psi) \bigg] = c \int_X \mathbb E \big[ f\big(\mathcal G(Y \cup \{o,y\}),v_y,v_o,\psi(Y \cup \{o,y\}) \big) \big] \, {\rm vol}(dy).
\end{equation*}
To prove \eqref{equ:MTPEmbeddedDelaunay}, we now show that the integrands in the previous two displays are the same: Fix $y\in X$. Let $\tau$ denote an involutive isometry at the midpoint $m$ of the geodesic between $o$ and~$y$; note that $\tau$ interchanges $o$ and $y$. Let $Y'=\tau(Y)$. Distributional invariance of the Poisson point process w.r.t.~isometries implies 
$$
\mathbb E \big[ f\big(\mathcal G(Y \cup \{o,y\}),v_o,v_y,\psi(Y \cup \{o,y\}) \big) \big] = \mathbb E \big[ f\big(\mathcal G(Y' \cup \{o,y\}),v_o,v_y,\psi(Y' \cup \{o,y\}) \big) \big].
$$
The crucial observation is that
$$
(\mathcal G(Y' \cup \{o,y\}),v_o,v_y,\tau \circ \psi(Y' \cup \{o,y\})) = ( \mathcal G(Y\cup\{0,y\}),v_y,v_o, \psi(Y\cup\{0,y\})).
$$
By assumption \eqref{equ:DiagonalInvariance1}, we thus obtain that 
$$
\mathbb E [ f(\mathcal G(Y \cup \{o,y\}),v_o,v_y,\psi(Y \cup \{o,y\})) ] = \mathbb E [ f(\mathcal G(Y\cup\{0,y\}),v_y,v_o, \psi(Y\cup\{0,y\})) ],
$$
and \eqref{equ:MTPEmbeddedDelaunay} follows (up to adding $\mathbb E[f(\mathcal G,\rho,\rho,\psi)]$ to both terms). We now further adapt this argument to show \eqref{equ:MTPEmbeddedDelaunayLabels}.

(3).~By the Marking Theorem (see \cite[Theorem 5.6]{LP17}), $\mathbf Y$ is a Poisson point process on $X\times[0,1]$ with intensity measure $c{\rm vol} \otimes {\rm Unif}[0,1]$. Let $f$ be a non-negative measurable function of bi-rooted embedded labeled graphs satisfying \eqref{equ:DiagonalInvariance}. For $\overline\eta\in \mathbf M(X\times[0,1])$, denote by $\mathcal G(\eta)$ the Delaunay graph defined using $\eta$, let $\Psi(\overline\eta)=(\psi(\eta),\psi'(\overline\eta))$ denote the marking of the embedded Delaunay graph and let $v_y$ denote the vertex corresponding to the cell of $y\in \eta$. For $z_0\in[0,1]$, define the measurable function 
\begin{align*}
& F_{z_0} \colon \mathbf M(X\times[0,1])\times (X\times[0,1])\to[0,\infty) \\
& \qquad \qquad \ F_{z_0} \big(\overline\eta,(y,z)) := \mathbf 1_{(y,z)\in \overline\eta} \ f\big(\mathcal G(\eta \cup \{o\}),v_o,v_y,\Psi(\overline\eta \cup\{o,z_0\})\big).
\end{align*}
 By the Mecke equation (see \cite[Theorem 4.1]{LP17})
\begin{equation*}
\mathbb E \bigg[ \sum_{(y,z) \in \mathbf{Y}} F_{z_0}(\mathbf Y,(y,z)) \bigg] = c \int_{X\times[0,1]} \mathbb E [ F_{z_0}(\mathbf Y \cup \{(y,z)\},(y,z)) ] \ {\rm vol}\otimes{\rm Unif}[0,1](dy,dz).
\end{equation*}
By Fubini's theorem,
\begin{align*}
& \mathbb E \bigg[ \sum_{x\in V(\mathcal G)\setminus\{\rho\}} f(\mathcal G,\rho,x,\Psi) \bigg] = \int_0^1 \mathbb E \bigg[ \sum_{(y,z) \in \mathbf{Y}} F_{z_0}(\mathbf Y,(y,z)) \bigg] dz_0 \\
& \ \ = c \int_{X\times[0,1]\times[0,1]} \mathbb E \big[ f\big(\mathcal G(Y \cup \{o,y\}),v_o,v_y,\Psi(\mathbf Y \cup \{(o,z_0),(y,z)\})) \big] \ {\rm vol}(dy) \, dz \, dz_0.
\end{align*}
By similar reasoning 
\begin{align*}
& \mathbb E \bigg[ \sum_{x\in V(\mathcal G)\setminus\{\rho\}} f(\mathcal G,x,\rho,\Psi) \bigg] \\
& \ \ = c \int_{X\times[0,1]\times[0,1]} \mathbb E \big[ f\big(\mathcal G(Y \cup \{o,y\}),v_y,v_o,\Psi(\mathbf Y \cup \{(o,z_0),(y,z)\})) \big] \ {\rm vol}(dy) \, dz \, dz_0 
\end{align*}
We again show that the integrands in the above two displays are the same: Fix $y\in X$. Let $\tau$ denote an involutive isometry at the midpoint $m$ of the geodesic between $o$ and~$y$ and let $\mathbf Y'=\tau(\mathbf Y)$. By distributional invariance of $\mathbf Y$,
\begin{align*}
& \mathbb E \big[ f\big(\mathcal G(Y \cup \{o,y\}),v_o,v_y,\Psi(\mathbf Y \cup \{(o,z_0),(y,z)\}) \big) \big] \\
& \qquad \qquad = \mathbb E \big[ f\big(\mathcal G(Y' \cup \{o,y\}),v_o,v_y,\Psi(\mathbf Y' \cup \{(o,z_0),(y,z)\}) \big) \big].
\end{align*}
Since again
$$
(\mathcal G(Y' \cup \{o,y\}),v_o,v_y,\tau \circ \Psi(\mathbf Y' \cup \{(o,z_0),(y,z)\})) = ( \mathcal G(Y\cup\{0,y\}),v_y,v_o, \Psi(\mathbf Y \cup \{(o,z_0),(y,z)\})),
$$
assumption \eqref{equ:DiagonalInvariance} implies the claim. The proof of Proposition \ref{prop:Delaunay} is thus complete.
\end{proof}

We point out that Proposition~\ref{prop:Delaunay}~(1) holds more generally for (locally finite) factor graphs of Palm versions of stationary point processes~\cite{AL07,AM22}. Proposition~\ref{prop:Delaunay}~(2) has been shown more generally for Delaunay  -- in the sense of \cite{BPP18} -- graphs of Palm versions of general stationary point processes on symmetric spaces, see \cite[Theorem 1.4]{P18}.

\begin{lemma}[Extremality of the Delaunay graph] \label{lm:DelaunayExt} Let $G$ be a non-compact connected semisimple real Lie group and $(X,d_X)$ be its symmetric space. Let $Y$ be a Poisson point process on $X$ with intensity $c \, {\rm vol}_X$ for some $c>0$. Let $Y_0=Y \cup \{o\}$ and let ${\rm Vor}(Y_0)$ be the associated Voronoi diagram. Let $\mathcal G$ denote the Delaunay graph associated to ${\rm Vor}(Y_0)$, and let $\rho$ be the vertex of $\mathcal G$ corresponding to the cell of the origin. Then the law of $(\mathcal G,\rho)$ is an extremal unimodular probability measure.
\end{lemma}

\begin{proof} This follows from Proposition \ref{prop:Delaunay}~(1), Theorem \ref{thm:Extremal} and Lemma \ref{lm:Palm}~(2). \end{proof}

We may also point out that Proposition~\ref{prop:Delaunay}~(1) together with Lemma~\ref{lm:FinExpDegree} imply reversibility of the degree-biased version of the Delaunay graph.

\subsection{Cluster frequencies} \label{sec:Frequency}

In this section, we define a notion of frequency of a cluster suitable for the purposes of this paper.  

Let $\nu_x$ denote the normalized restriction of ${\rm vol}$ to $\mathcal{B}_1(x)$, the open unit ball around $x$, for every $x\in X$.
As every isometry of $(X,d_X)$ preserves $\vol$ by Theorem~\ref{thm:BasicSymmetricSpace}~(1), we have that $\nu_x$ is the push-forward of $\nu_o$ for every isometry $\tau$ that satisfies $\tau(o)=x$.
Let $(R_k)_{k=1}^\infty$ denote \emph{random walk on~$X$} with transition probabilities $\nu_x$ started in $o\in X$, i.e.~the Markov chain on $X$ with $R_0=o$ and 
\begin{equation}
\mathbb P[ R_{k+1}\in \cdot \mid R_1,\ldots,R_k] = \mathbb P[ R_{k+1}\in \cdot \mid R_k] = \nu_{R_k}.
\end{equation}
The random walk started in $x\in X$ will be denoted by $(R^x_k)_k$.

\begin{remark} In this section, the {\em random walk} always refers to the particular random walk defined above and we will only use this random walk to measure cluster frequencies. It seems possible to work with other spatially homogeneous Markov processes.
\end{remark}

Following an approach, which in the discrete setting is due to H\"{a}ggstr\"{o}m and Jonasson \cite[Section 8]{HJ06}, we now show that random walk on $X$ allows us to define an isometry invariant cluster frequency.
Let us emphasize that invariance under all isometries, as opposed to $G$-invariance, will be important in our subsequent applications of the MTP \eqref{equ:MTPEmbeddedDelaunayLabels}.

To be more precise, given $Z\subseteq X$, we will be interested in the asymptotic density of the sequence $({\bf 1}_Z(R_k))_{k=1}^\infty$, where ${\bf 1}_Z$ is the indicator of $Z$.
We will establish the first results regarding this quantity for random closed sets under the weaker assumption of $G$-invariance, as opposed to isometry invariance. We start with a preliminary observation.

\begin{proposition}\label{pr:StationaryWalks}
    Let $\mathcal{Z}$ be a $G$-invariant random closed subset of $X$ and $(R_k)_{k=1}^\infty$ be an independent random walk on $X$.
    Then $({\bf 1}_\mathcal{Z}(R_k))_{k=1}^\infty$ is a stationary process.
\end{proposition}
\begin{proof}
    We need to show that 
    \begin{equation*}
    ({\bf 1}_\mathcal{Z}(R_k))_k \overset{(d)}{=} ({\bf 1}_\mathcal{Z}(R_{k+m}))_k
    \end{equation*}
    for every $m\in\mathbb N$. To see this, it suffices to show that for every $(\sigma_0,\ldots,\sigma_\ell)\in\{0,1\}^\ell$
    $$\mathbb{P}({\bf 1}_\mathcal{Z}(R_0)=\sigma_0,\dots,{\bf 1}_\mathcal{Z}(R_\ell)=\sigma_\ell)=\mathbb{P}({\bf 1}_\mathcal{Z}(R_m)=\sigma_0,\dots,{\bf 1}_\mathcal{Z}(R_{m+\ell})=\sigma_\ell).$$
    Let $x\in X$ and choose $g\in G$ with $g\cdot o=x$. Then 
    $$
    (R_0^x,\ldots,R_\ell^x) \overset{(d)}{=}(go,gR_1,\ldots,gR_{\ell}),
    $$
    where $(R^x_k)_k$ denotes the random walk starting at $x$. Hence
    \begin{align*}
    \mathbb P ({\bf 1}_\mathcal{Z}(R_m)=\sigma_0,\dots,{\bf 1}_\mathcal{Z}(R_{m+\ell})=\sigma_\ell \mid R_m = x) & = \mathbb P ({\bf 1}_\mathcal{Z}(R_0^x)=\sigma_0,\dots,{\bf 1}_\mathcal{Z}(R^x_\ell)=\sigma_\ell) \\
    & = \mathbb P ({\bf 1}_\mathcal{Z}(go)=\sigma_0,\dots,{\bf 1}_\mathcal{Z}(gR_{\ell})=\sigma_\ell) \\
    & = \mathbb P ({\bf 1}_{g^{-1}\mathcal{Z}}(o)=\sigma_0,\dots,{\bf 1}_{g^{-1}\mathcal{Z}}(R_{\ell})=\sigma_\ell) \\
    & = \mathbb P ({\bf 1}_{\mathcal{Z}}(o)=\sigma_0,\dots,{\bf 1}_{\mathcal{Z}}(R_{\ell})=\sigma_\ell)
    \end{align*}
   where we have used $G$-invariance of $\mathcal Z$ and the fact that $\mathcal Z$ and $(R_k)$ are independent in the last step. Since $x\in X$ was arbitrary, the claim follows. 
\end{proof}

\begin{corollary} \label{cor:Birkhoff}
    Let $\mathcal{Z}$ be a $G$-invariant random closed subset of $X$ and $(R_k)_{k=1}^\infty$ be an independent random walk on $X$.
    Then
    $$\beta(\mathcal{Z},(R_k)_k):=\lim_{n\to\infty}\frac{1}{n}\sum_{i=0}^{n-1} {\bf 1}_\mathcal{Z}(R_i)$$
    exists almost surely and in $L^1$.
\end{corollary}

\begin{proof} This is a consequence of Proposition~\ref{pr:StationaryWalks} and the ergodic theorem, see for instance \cite[Theorem 6.2.1]{D19}.
\end{proof}

In fact, the following analogue of \cite[Lemma 8.2]{HJ06} shows that the limit in Corollary \ref{cor:Birkhoff} does not depend on the random walk path.

\begin{proposition}\label{pr:DefinitionOfFrequency}
     Let $\mathcal{Z}$ be a $G$-invariant random closed subset of $X$ and $(R_k)_{k=1}^\infty$ be an independent random walk on $X$.
    Then almost surely the value $\beta(\mathcal{Z},(R_k)_k)$ depends only on $\mathcal Z$, i.e.~does not depend on $(R_k)_k$.
\end{proposition}

In the setting of Proposition \ref{pr:DefinitionOfFrequency}, we may thus define the {\em frequency} of $\mathcal{Z}$ with respect to the random walk $(R_k)_k$ to be the almost sure limit 
\begin{equation} \label{def:frequency}
\beta(\mathcal{Z}):=\beta(\mathcal{Z},(R_k)_k) = \lim_{n\to\infty}\frac{1}{n}\sum_{i=0}^{n-1} {\bf 1}_\mathcal{Z}(R_i).
\end{equation}

\begin{proof}[Proof of Proposition \ref{pr:DefinitionOfFrequency}]
    The argument is identical with the proof of \cite[Lemma~8.2]{HJ06}.
    Namely, by Corollary~\ref{cor:Birkhoff} and L\'{e}vy's $0$-$1$-law we have for every $c\in [0,1]$ that
    $$\lim_{m\to\infty} \mathbb{P}(\beta(\mathcal{Z},(R_k)_k)\le c \ | \ R_0,\dots,R_m,\mathcal{Z})={\bf 1}_{\beta(\mathcal{Z},(R_k)_k)\le c}$$
    almost surely. This implies that for every $\eps>0$ there is $M\in \mathbb{N}$ such that
    $$
    \mathbb{P}(\beta(\mathcal{Z},(R_k)_k)\le c \ | \ R_0,\dots,R_M,\mathcal{Z}) \in[0,\eps]\cup [1-\eps,1]
    $$
    with probability at least $1-\eps$. Since 
    $$
    \lim_{n\to\infty}\frac{1}{n}\sum_{i=0}^{n-1} {\bf 1}_\mathcal{Z}(R_i) = \lim_{n\to\infty}\frac{1}{n}\sum_{i=M}^{M+n-1} {\bf 1}_\mathcal{Z}(R_i)
    $$ 
   where the right-hand side depends on $R_0,\ldots,R_M$ only through $R_M$, we obtain that
    $$
    \mathbb{P}\bigg( \lim_{n\to\infty}\frac{1}{n}\sum_{i=M}^{M+n-1} {\bf 1}_\mathcal{Z}(R_i) \le c \ \bigg| \ R_M,\mathcal{Z}\bigg)\in[0,\eps]\cup [1-\eps,1
    ]$$
    with probability at least $1-\eps$. By $G$-invariance of $\mathcal{Z}$ and independence of $(R_k)_k$ and $\mathcal Z$, we conclude that
    $$\mathbb{P}(\beta(\mathcal{Z},(R_k)_k)\le c \ | \ \mathcal{Z})\in[0,\eps]\cup [1-\eps,1]$$
    with probability at least $1-\eps$.
    As $c$ and $\eps$ were arbitrary, the proof is finished.
\end{proof}

Next, we prove an auxiliary claim about the distribution of $R_k$.

\begin{lemma}\label{lm:DistributionRW}
    Let $(R_k)_{k=1}^\infty$ be random walk on $X$ and let $k\in \mathbb{N}$.
    Then the distribution of~$R_k$ is supported on $\mathcal{B}_{k}(o)$ with Radon--Nikodym derivative 
    $$
    F_k := \frac{ d \mathbb P(R_k\in \cdot)}{d {\rm vol}_{|\mathcal B_k(o)}} > 0.
    $$ 
\end{lemma}
\begin{proof}
    The fact that the distribution of $R_k$ is concentrated on $\mathcal{B}_{k}(o)$ follows directly from the definition. Regarding the additional part, note that the Radon--Nikodym derivative of $\nu_o=\mathbb P^{R_1}$ with respect to ${\rm vol}$ is given by a scalar multiple of ${\bf 1}_{\mathcal{B}_1(o)}$. Hence the claim holds for $k=1$.
    Suppose it holds for $k\ge 1$. Then, for every measurable $A\subset X$, we have that
    \begin{equation} \label{equ:Fk}
    \mathbb P ( R_{k+1} \in A) = \int_X \mathbb P ( R_{k+1} \in A \mid R_k = x) \mathbb P^{R_k}(d x) = \int_{\mathcal B_k(o)} \nu_x(A) \, F_k \, {\rm vol} (dx).
    \end{equation}
    
    To show that $\mathbb P^{R_{k+1}}$ is absolutely continuous with respect to ${\rm vol}_{|\mathcal B_{k+1}(o)}$, let $A\subset \mathcal B_{k+1}(o)$ measurable with ${\rm vol}(A)=0$. Similarly to $\nu_0$, $\nu_x$ is absolutely continuous with respect to ${\rm vol}$ for every $x\in X$, hence $\nu_x(A)=0$. Hence \eqref{equ:Fk} implies that $\mathbb P(R_{k+1}\in A)=0$, which proves the claimed absolute continuity. 
    
    In particular, the Radon-Nikodym derivative $F_{k+1} := d \mathbb P^{R_{k+1}}/d {\rm vol}_{|\mathcal B_{k+1}(o)}$ exists. To show that $F_{k+1}>0$, it suffices to prove mutual absolute continuity. To that end, let $A\subset \mathcal B_{k+1}(o)$ measurable with $\mathbb P(R_{k+1}\in A)=0$. Since $F_k>0$ by the induction hypothesis, \eqref{equ:Fk} implies that $\nu_x(A)=0$ for ${\rm vol}_{|\mathcal B_{k}(o)}$-almost every $x$. Since ${\rm vol}_{|\mathcal B_1(x)}$ is absolutely continuous with respect to $\nu_x$, we also have that ${\rm vol}_{|\mathcal B_1(x)}(A)=0$ for ${\rm vol}_{|\mathcal B_{k}(o)}$-almost every $x$. We claim that     \begin{equation} \label{equ:volB1x}
     {\rm vol}_{|\mathcal B_1(x)}(A)=0 \quad \text{for every} \, \, x \in \mathcal B_k(o).
     \end{equation}
     Suppose there is $x\in \mathcal B_k(o)$ such that ${\rm vol}_{|\mathcal B_1(x)}(A)>0$. By monotonicity 
     $$
     {\rm vol}_{|\mathcal B_1(x)}(A)= {\rm vol}(A \cap \mathcal B_1(x)) = \lim_{\eps\to0} \, {\rm vol} (A \cap \mathcal B_{1-\eps}(x)),
     $$ 
     hence ${\rm vol} (A \cap \mathcal B_{1-\eps_0}(x))>0$ for some $\eps_0>0$. If $y\in X$ with $d_X(x,y)<\eps_0$, then $\mathcal B_{1-\eps_0}(x)\subset \mathcal B_1(y)$ and thus ${\rm vol}_{|\mathcal B_1(y)}(A) \ge {\rm vol} (A \cap \mathcal B_{1-\eps_0}(x))>0$. Hence ${\rm vol}_{|\mathcal B_1(y)}(A)>0$ for all $y\in \mathcal B_{\eps_0}(x)$, a set of positive volume. This proves \eqref{equ:volB1x}.
     
     Now note that ${\rm vol}(A) := \lim_{\eps\to0} \,{\rm vol}(A \cap \mathcal B_{k+1-\eps})$ by monotonicity. To prove ${\rm vol}(A)=0$, it thus suffices to show that ${\rm vol}(A \cap \mathcal B_{k+1-\eps})=0$ for every $\eps>0$. For each $\eps>0$, 
     $$
     \mathcal {\overline B}_{k+1-\eps}(o) \subset \bigcup_{x\in B_{k}(o)} \mathcal B_1(x)
     $$
     is an open cover of the compact set $\mathcal {\overline B}_{k+1-\eps}$, hence $\mathcal B_{k+1-\eps}(o)\subset \bigcup_{i=1}^n \mathcal B_1(x_i)$ for some $n\ge1$ and $x_1,\ldots,x_n \in \mathcal B_{k(o)}$. By \eqref{equ:volB1x}, ${\rm vol}(A \cap \mathcal B_{k+1-\eps}) \le \sum_{i=1}^n {\rm vol}_{|\mathcal B_1(x_i)}(A) = 0$ as claimed.
\end{proof}

We now show that $\beta(\mathcal{Z})$ defined in \eqref{def:frequency} above is isometry invariant. In fact, we shall prove a stronger statement. So far we have associated frequencies to $G$-invariant random closed subsets of $X$. The next step is to associate frequencies to other subsets $Z$ of $X$, and in particular to clusters of $\mathcal Z$. We will show in Proposition \ref{pr:Invariance} below that these frequencies are also isometry invariant.

 Let $Z\subset X$ closed and let $(R_k)_{k=1}^\infty$ be random walk on $X$. The {\em frequency} of $Z$ is defined as
\begin{equation} \label{def:FrequencyZ}
\beta(Z) := \beta(Z,(R_k)_k) :=  \lim_{n\to\infty}\frac{1}{n}\sum_{i=0}^{n-1} {\bf 1}_{Z}(R_i)
\end{equation}
provided that the limit exists almost surely and does not depend on the random walk.
Similarly, if $\mathcal Z$ is a random closed subset of $X$ and, for $y\in X$, $C_y(\mathcal Z)$ denotes the cluster of $\mathcal Z$ containing~$y$, the {\em cluster frequency} of $C_y$ is defined as
\begin{equation} \label{def:FrequencyCluster}
\beta(C_y) := \lim_{n\to\infty}\frac{1}{n}\sum_{i=0}^{n-1} {\bf 1}_{C_y}(R_i)
\end{equation}
provided that the limit exists almost surely and does not depend on the random walk.

\begin{proposition}[Invariance of frequencies] \label{pr:Invariance}
    Let $(R_k)_{k=1}^\infty$ be random walk on $X$. Let $Z$ be a closed subset of $X$ such that $\beta(Z)$ exists a.s.~and does not depend on the random walk.
    Then, for every isometry $\tau$ of $(X,d_X)$,
    $$\beta(\tau(Z)) = \beta(Z).$$
    In particular, if $\mathcal{Z}$ is a $G$-invariant enough random closed subset of $X$ which is independent of $(R_k)_k$, then $\beta(\mathcal{Z})$ is an isometry invariant measurable function.
\end{proposition}
\begin{proof}
    Let $\tau$ be an isometry.
    Showing that $\beta(Z)=\beta(\tau(Z)):=\beta(\tau(Z),(R_k)_k)$ is the same as showing that
    $$\beta(Z)=\beta(Z,(R^{x}_k)_k)$$
    for almost every $(R^x_k)_k$, where $x=\tau^{-1}(o)$.
    Indeed, as every isometry preserves $\vol$ by Theorem~\ref{thm:BasicSymmetricSpace}~(1), we have that
    $$(\tau^{-1}(R_k))_{k}\overset{(d)}{=}(R^x_k)_k$$
    by the definition of the random walk.
    
    Choose $k\in \mathbb{N}$ such that $\mathcal{B}_1(x)\subseteq \mathcal{B}_{k}(o)$. By the assumption, for $\mathbb P^{R_k}$-almost every $y\in X$, we have that
    $$\beta(Z)=\beta(Z,(R^y_k)_k)$$
    for almost every realization of $(R^y_k)_k$.
    By Lemma~\ref{lm:DistributionRW}, the Radon--Nikodym derivative $F_{k}$ of $\mathbb P^{R_k}$ with respect to ${\rm vol}$ is strictly positive on $\mathcal{B}_{k}(o)$, and in particular on $\mathcal{B}_1(x)$.
    It follows that for $\nu_x$-almost every $y\in \mathcal{B}_1(x)$, we have that
    $$\beta(Z)=\beta(Z,(R^y_k)_k)$$
    for almost every realization of $(R^y_k)_k$.
    We infer that
    $$\beta(Z)=\beta(Z,(R^x_k)_k)$$
    for almost every realization of $(R^x_k)_k$, as desired.
    The additional part of the statement is a direct consequence of the main part together with Proposition~\ref{pr:DefinitionOfFrequency}.
\end{proof}

The next theorem asserts that almost surely all Poisson-Voronoi percolation clusters have a cluster frequency. To state the result, recall that we denote by $\Omega$ the space of configurations of $\omega_p^{(\lambda)}$ and write $C_y(\omega)$ for the cluster of $\omega\in\Omega$ containing $y\in X$.

\begin{theorem}[Cluster frequencies]\label{thm:ClusterFrequency}
    Let $\lambda>0$ and $p\in(0,1]$. 
    There is a measurable map
    $$\tilde{\beta}:X\times \Omega\to [0,1],$$
    which is diagonally invariant under isometries,
   and an isometry invariant event $\mathcal A\subset \Omega$ with $\mathbb{P}^{(\lambda)}_p(\mathcal A)=1$ such that    
   \begin{equation}\label{eq:Satisfy}
       \tilde{\beta}(y,\omega)=\beta(C_y(\omega))
   \end{equation}
   for every $(y,\omega)\in X\times\mathcal A$.
\end{theorem}
\begin{proof}
    Define $\tilde{\beta}(y,\omega^{(\lambda)}_p):= \beta(C_y)=\beta(C_y,(R_k)_k)$ provided that the right-hand side exists almost surely and does not depend on the random walk, otherwise set $\tilde{\beta}(y,\omega^{(\lambda)}_p):=0$.
    The map $\tilde{\beta}$ is isometry invariant by Proposition~\ref{pr:Invariance}, and it is easy to check that it is also measurable.

    Set $\mathcal{A}_o$ for the event that of $\tilde{\beta}(o,\omega^{(\lambda)}_p)= \beta(C_o)=\beta(C_o,(R_k)_k)$ exists almost surely and does not depend on the random walk.
    
    \begin{claim}
    We have that $\mathbb{P}^{(\lambda)}_p(\mathcal A_o)=1$.
    \end{claim}
    \begin{proof}
        The argument is identical with the proof of \cite[Theorem~8.4]{HJ06}.
        Consider the random closed subset $\Gamma\subset X$ defined by erasing each cluster of $\omega_p^{(\lambda)}$ independently with probability $1/2$. It is easy to see that $\Gamma$ is $G$-invariant. Let again $C_o$ denote the cluster of the origin in $\omega_p^{(\lambda)}$ and let $C^\Gamma_o$ denote the cluster of the origin in $\Gamma$.
        Note that conditioned on $C_o\not=\emptyset$ and $C^\Gamma_o\not=\emptyset$, $C_o$ and $C^\Gamma_o$ have the same distribution.
        On the same probability space, define $\Gamma'$ as follows.
        If $C_o=\emptyset$, set $\Gamma=\Gamma'$.
        Otherwise, let $\Gamma'$ be such that $C_o=\Gamma\triangle \Gamma'$, that is, $\Gamma'$ takes the opposite of the outcome of the coin flip at $C_o$.
        It follows that $\Gamma'$ and $\Gamma$ have the same distribution. By Proposition~\ref{pr:DefinitionOfFrequency}, we have that almost surely
        \begin{equation}\label{eq:cluster}
            \lim_{n\to\infty} \frac{1}{n} \sum_{\ell=0}^n \left({\bf 1}_{\Gamma}(R_k)-{\bf 1}_{\Gamma'}(R_k)\right)
        \end{equation}
        takes the same value for almost every realization of $(R_k)_{k}$.
        Note that the value in \eqref{eq:cluster} is, up to a sign, equal to $\beta(C_o,(R_k)_k)$.
        In particular, $\tilde\beta(C_o,\omega^{(\lambda)}_p)$ is well-defined almost surely as a function of $\Gamma$.
        This finishes the proof by the above observation that, conditioned on $C_o\not=\emptyset$ and $C^\Gamma_o\not=\emptyset$, $C_o$ and $C^\Gamma_o$ have the same distribution.
    \end{proof}

    Now, the proof is finished as follows.
    Let $\{g_n\}_{n=1}^\infty$ be a countable dense subset of $G$ and define $\mathcal{A}=\bigcap_{n\in \mathbb{N}} g_n\cdot \mathcal{A}_o$, where $\omega^{(\lambda)}_p\in g\cdot \mathcal{A}_0$ if and only if $g^{-1}\cdot \omega^{(\lambda)}_p\in \mathcal{A}_0$ for every $g\in G$.
    We claim that $\mathcal{A}$ works as required.

    Clearly $\mathcal{A}$ is a $\mathbb{P}^{(\lambda)}_p$-conull event as $g_n\cdot \mathcal{A}_0$ is $\mathbb{P}^{(\lambda)}_p$-conull for every $n\in \mathbb{N}$.
    By the definition and Proposition~\ref{pr:Invariance}, we have that if $\omega^{(\lambda)}_p\in \mathcal{A}$, then for every $n\in \mathbb{N}$
    $$\tilde{\beta}(g_n\cdot o,\omega^{(\lambda)}_p)=\beta(C_{g_n\cdot o},(R_k)_k)$$
    for almost every realization of $(R_k)_k$.
    Given an arbitrary $y\in X$, there are two cases.
    Either $y\not\in \omega^{(\lambda)}_p$, which clearly implies that 
    $$\tilde{\beta}(y,\omega^{(\lambda)}_p)=\beta(C_y,(R_k)_k)=0$$
    for every realization of $(R_k)_k$, or there is $n\in \mathbb{N}$ such that $g_n\cdot o\in C_y$ because every cluster has non-empty interior.
    In that case we have
    $$\beta(C_{y},(R_k)_k)=\beta(C_{g_n\cdot o},(R_k)_k)$$
    for every realization of $(R_k)_k$.
    This shows that $X\times \mathcal{A}$ satisfies \eqref{eq:Satisfy}.
    
    Finally, the proof is finished by noting that $\mathcal{A}$ is invariant under all isometries of $(X,d_X)$.
    Indeed, let $\omega^{(\lambda)}_p\in \mathcal{A}$.
    By a combination of Proposition~\ref{pr:Invariance} combined with the fact that $X\times \mathcal{A}$ satisfies \eqref{eq:Satisfy}, we have that for every isometry $\tau$
    $$g_n^{-1}\cdot \tau(\omega^{(\lambda)}_p)\in \mathcal{A}_o$$
    for every $n\in \mathbb{N}$.
    In particular,
    $$\tau(\omega^{(\lambda)}_p)\in \bigcap_{n\in \mathbb{N}} g_n\cdot \mathcal{A}_o$$
    and the proof is finished.
\end{proof}

Let us conclude by recording the following well-known consequence of long-range order for later use.

\begin{lemma}[Long-range order implies positive frequency]\label{lm:LROimpliesPositiveFreq}
    Let $\omega_p^{(\lambda)}$ be Poisson-Voronoi percolation with parameters $\lambda>0$ and $p\in(0,1]$ such that
    $$
    \inf_{x,y\in X} \mathbb P_{p}^{(\lambda)}\big( x \leftrightarrow y \big) \geq \delta
    $$
    for some $\delta>0$.
    Let $\tilde{\beta}$ be the map from Theorem~\ref{thm:ClusterFrequency}. Then
    $$\mathbb E \big[ \tilde{\beta}(o,\omega^{(\lambda)}_p) \big]  \ge \delta.$$    
    In particular, $\tilde{\beta}(o,\omega^{(\lambda)}_p)>0$ with positive probability.   
\end{lemma}
\begin{proof}
Let $\mathbb P^{(\lambda)}_p$ denote the distribution of $\omega_p^{(\lambda)}$ and let $\mathbf P$ denote the distribution of  random walk $(R_k)$ on $X$. By Theorem~\ref{thm:ClusterFrequency}, the Fubini-Tonelli theorem and the dominated convergence theorem we have that 
\begin{align*}
\mathbb E \big[ \tilde{\beta}(o,\omega^{(\lambda)}_p) \big] & = \int \left(\int \lim_{n\to\infty} \frac{1}{n}\sum_{\ell=0}^n{\bf 1}_{C_o}(R_\ell) \ d\mathbf P\right) \ d\mathbb{P}^{(\lambda)}_p \\
& = \int \left(\int \lim_{n\to\infty} \frac{1}{n}\sum_{\ell=0}^n{\bf 1}_{C_o}(R_\ell) \ d\mathbb{P}^{(\lambda)}_p\right) \ d \mathbf P \\
& = \lim_{n\to\infty} \frac{1}{n}\sum_{\ell=0}^n \int \left(\int {\bf 1}_{C_o}(R_\ell) \ d\mathbb{P}^{(\lambda)}_p\right) \ d \mathbf P \\
& \ge \lim_{n\to\infty} \frac{1}{n}\sum_{\ell=0}^n \int \inf_{x,y\in X} \mathbb P_{p}^{(\lambda)}\big( x \leftrightarrow y \big) \ d \mathbf P \ge \delta.
\end{align*}
The proof of Lemma \ref{lm:LROimpliesPositiveFreq} is thus complete.
\end{proof}

\begin{remark}[Clusters of maximal frequency vs.~uniqueness] In the situation of Lemma~\ref{lm:LROimpliesPositiveFreq}, it is possible to prove that almost surely there exists a unique cluster $C_\infty$ with $\beta(C_\infty)>0$. Moreover, $C_\infty$ is unbounded and its cluster frequency is a positive constant. Upon first glance, this seems to prove Theorem \ref{thm:LRO}.
However, we have not yet ruled out the possibility that there are unbounded clusters with frequency $0$. In classical percolation theory, this possibility does not occur by the {\em Indistinguishability Theorem}~\cite{LS99}. While a version of this result exists for percolation on unimodular random graphs \cite[Theorem 6.15]{AL07}, that notion of indistinguishability does not directly apply to the frequency defined above; more precisely, it does not state that Poisson-Voronoi percolation clusters are indistinguishable with respect to $G$-invariant events~$A\subset\Omega$. This would be the ''natural'' notion of indistinguishability. Since it is possible to prove Theorem~\ref{thm:LRO} by more elementary arguments, see~Section \ref{sec:ProofLRO}, we do not pursue the question about ''natural'' indistinguishability here. Let us, however, mention that in \cite{T21} ''natural'' indistinguishability is proved for a wide, but different, class of continuum percolations and that the method may be relevant for carrying out the alternative approach in our setting.
\end{remark}

\subsection{Proof of Theorem \ref{thm:LRO}} \label{sec:ProofLRO}

Define 
\begin{equation}
    p_{\rm LRO}(\lambda) := \inf \Big\{ p \, : \, \inf_{x,y\in X} \mathbb P_{p}^{(\lambda)}\big( x \leftrightarrow y \big)>0 \Big\}.
\end{equation}
As pointed out before, proving $p_{\rm LRO}(\lambda) \le p_u(\lambda)$ is standard. More precisely, suppose that $\omega_p^{(\lambda)}$ has a unique unbounded cluster~$C_{p,\infty}^{(\lambda)}$. By $G$-invariance and the fact that $G$ acts transitively on $X$, we have that
$$
p_\infty(\lambda) := \mathbb P \big( o \in C_{p,\infty}^{(\lambda)} \big) = \mathbb P \big( x \in C_{p,\infty}^{(\lambda)} \big) > 0
$$
for all $x\in X$. The Harris-FKG-Inequality, see Lemma \ref{lm:FKG}, then implies 
$$
\inf_{x,y\in X} \mathbb P_{p}^{(\lambda)}\big( x \leftrightarrow y \big) \ge \inf_{x,y\in X} \mathbb P\big( x,y\in C_{p,\infty}^{(\lambda)}\big) \geq p_\infty(\lambda)^2>0.
$$
This proves that $p_{\rm LRO}(\lambda) \le p_u(\lambda)$. We now prove the converse, which is an immediate consequence of the following claim.

\begin{claim}\label{cl:unique}
    Let $\lambda>0$ and $p\in(0,1]$ be such that 
    \begin{equation} \label{equ:LRO1}
        \inf_{x,y\in X} \mathbb P_{p}^{(\lambda)}\big( x \leftrightarrow y \big)>0.
    \end{equation}
    Then, for every $p'>p$, $\omega_{p'}^{(\lambda)}$ has a unique unbounded cluster almost surely.
\end{claim}

\begin{proof}[Proof of Claim \ref{cl:unique}] Since the event that there exists a unique unbounded cluster is $G$-invariant, Lemma  \ref{lm:Palm}~(1) shows that it suffices to prove the same conclusion with the underlying Poisson point process ${\bf Y}^{(\lambda)}$ replaced with its Palm version ${\bf Y}^{(\lambda)}\cup\{{\bf o}\}$. Let $\mathcal G$ denote the associated Delaunay graph and let $\rho$ denote the vertex corresponding to the cell of the origin. Here, we have dropped the dependence on $\lambda$ to lighten the notation -- this parameter will be fixed throughout the proof of Claim \ref{cl:unique}. Let $\mathcal G[q]$ denote Bernoulli percolation with parameter $q\in(0,1)$ on $\mathcal G$. Note that since cells are bounded by Lemma~\ref{lm:BasicVoronoi} and every bounded set is covered by finitely many cells a.s.~by Lemma~\ref{lm:closed}, there is a unique unbounded cluster in Poisson-Voronoi percolation on $Y^{(\lambda)}\cup\{o\}$ with parameter $q$ if and only if there is a unique infinite cluster in~$\mathcal G[q]$. Thus, it suffices to show that $\mathcal G[p']$ has a unique infinite cluster almost surely for every $p'>p$. This is the statement we prove below.

Let $p'>p$ be fixed. By Lemma \ref{lm:DelaunayExt}, $(\mathcal G,\rho)$ is an extremal unimodular random graph. In particular, Proposition \ref{prop:Bernoulli} implies that the number of infinite clusters in $\mathcal G[p']$ is almost surely constant and equal to $0$, $1$ or $\infty$. To show that it is equal to $1$, we now rule out the other two options.

\medskip

{\em\noindent Ruling out $0$.} Clearly, \eqref{equ:LRO1} implies that $\omega_p^{(\lambda)}$ has an unbounded cluster with positive probability and, by Lemma \ref{lm:Ergodicity}, almost surely. By Lemma \ref{lm:Palm}~(1), the same is true for Poisson-Voronoi percolation on the Palm version. But, as explained above, this is equivalent with the number of infinite clusters in $\mathcal G[p]$ being non-zero almost surely. By monotonicity, also  $\omega_{p'}^{(\lambda)}$ has unbounded clusters almost surely.

\medskip

{\em\noindent Ruling out $\infty$.} 
Recall the definition of the cluster frequency $\beta$ from \eqref{def:FrequencyCluster}. By Lemma \ref{lm:LROimpliesPositiveFreq}, \eqref{equ:LRO1} implies that there exists a cluster $C$ of $\omega_p^{(\lambda)}$ with $\beta(C)>0$. Since $\sum_{C'} \beta(C')\le 1$, where the sum rangers over all clusters $C'$, there can be at most finitely many clusters of maximal frequency, i.e.~maximizing~$\beta$. Note that every such cluster must then be unbounded. 
By Proposition~\ref{pr:Invariance}, the event that there exist finitely many clusters with maximal frequency is $G$-invariant.
By Lemma~\ref{lm:Palm}~(1), Poisson-Voronoi percolation on the Palm version also has finitely many unbounded clusters maximizing $\beta$ almost surely. Hence there are finitely many infinite clusters in $\mathcal G[p]$ which are {\em special}, meaning their embedding into $X$ maximizes $\beta$.

Consider the canonical monotone coupling between Bernoulli percolations $\mathcal G[p]$ and $\mathcal G[p']$ on the embedded Delaunay graph with labels $(\mathcal G,\rho, \Psi)$, cf.~Section \ref{sec:Delaunay}.
In the rest of the proof we work with the embedded Delaunay graph with labels but suppress the dependence on $\Psi$ to lighten the notation.
We will now use the Mass Transport Principle \eqref{equ:MTPEmbeddedDelaunayLabels}, proved in Proposition~\ref{prop:Delaunay}, to prove that every infinite cluster of $\mathcal G[p']$ contains a special infinite cluster of $\mathcal G[p]$. This proof is a straightforward adaptation to the setting of unimodular random graphs of the proof of uniqueness monotonicity \cite{HP99} given in \cite[Theorem 5.4]{HJ06}. For completeness, we include the details below.

By monotonicity of the coupling, it suffices to show that every infinite $\mathcal G[p']$-cluster intersects a special infinite $\mathcal G[p]$-cluster almost surely. For $u\in V(\mathcal G)$, define
$$
D(u):= \min \big\{ {\rm dist}_{\mathcal G}(u,v) : v \, \, \text{is in some special infinite cluster of} \, \, \mathcal G[p] \big\}.
$$
Moreover, let $C_{p'}(u)$, resp.~$C_{p}(u)$, denote the cluster of $u$ in $\mathcal G[p']$, resp.~$\mathcal G[p]$. Note that on the event that some infinite $\mathcal G[p']$-cluster does not intersect any special infinite $\mathcal G[p]$-cluster, there exists a vertex $u\in V(\mathcal G)$ such that $|C_{p'}(u)|=\infty$ and
$$
D(u) = \min \big\{ D(v) : v\in C_{p'}(u) \big\} >0.
$$
Therefore it suffices to show that no such vertex exists almost surely in ${\bf Y}^{\lambda}\cup \{{\bf o}\}$.

By Lemma~\ref{lm:Palm}~(1), it is enough to show that a.s. no such vertex exists in $\mathbf{Y}^{(\lambda)}$.
For that it suffices to show, by using the same argument as in the proof of Lemma~\ref{lm:Palm}~(1), that $\mathbb P(B)=0$, where $B$ is the event that the origin is such a vertex.
To see that $\mathbb P(B)=0$, define, for $u\in V(\mathcal G)$, $M'(u)$ to be the set of vertices $v\in C_{p'}(u)$ which minimize the distance to special infinite $\mathcal G[p]$-clusters and let $M(u):=|M'(u)|$. Partition
$$
B = B_\infty \cup B_{{\rm fin}}, \qquad \text{where} \ \ B_\infty:=B \cap \{M(\rho)=\infty\} \ \ \text{and} \ \ B_{\rm fin}:=B \cap \{M(\rho)<\infty\}.
$$

We first show that $\mathbb P(B_\infty)=0$, which can be proved without \eqref{equ:MTPEmbeddedDelaunayLabels}. Namely, further partition $B_\infty=\bigcup_{k=1}^\infty B_{k,\infty}$, where 
$$
B_{k,\infty} := B_\infty \cap \{ D(\rho) = k \}.
$$
To show that $\mathbb P(B_{k,\infty})=0$, consider the following: condition on $\mathcal G[p]$ and then condition on the configuration of $\mathcal G[p']$ of all vertices which are not within distance $k$ of infinite special $\mathcal G[p]$-clusters. The conditional distribution of the $\mathcal G[p']$-configuration of the remaining vertices is then iid with probability $(p'-p)/(1-p)$ to be present in the configuration. Now note that on the event $B_{k,\infty}$, there are infinitely many disjoint paths of length $k$ of such vertices which tie $C_{p'}(\rho)$ to a special infinite $\mathcal G[p]$-cluster. Hence, on $B_{k,\infty}$, $C_{p'}(\rho)$ intersects a special infinite $\mathcal G[p]$-cluster almost surely. This proves $\mathbb P(B_{k,\infty})=0$ and thus $\mathbb P(B_\infty)=0$.

We now show that $\mathbb P(B_{\rm fin})=0$ by an application of the Mass Transport Principle~\eqref{equ:MTPEmbeddedDelaunayLabels}. Define a non-negative function $f$ of bi-rooted embedded labeled graphs by 
$$
f(H,u,v,\Phi)=1/M_H(u)
$$
whenever the embeddings of all $H[p]$-clusters have a well-defined frequency $\beta$, $u$ is in an infinite $H[p']$-cluster which does not intersect any special $H[p]$-cluster, the set $M'_H(u)$ of vertices in $u$'s infinite $H[p']$-cluster which minimize the distance from special infinite $H[p]$-clusters satisfies $M_H(u):=|M_H(u)|<\infty$, and $v\in M'_H(u)$ is such a minimizer. Otherwise, set $f(H,u,v,\Phi)=0$. By Theorem \ref{thm:ClusterFrequency}, $f$ is measurable and the condition \eqref{equ:DiagonalInvariance} holds. By~the Mass Transport Principle \eqref{equ:MTPEmbeddedDelaunayLabels},
$$
\mathbb E \bigg[ \sum_{x\in V(\mathcal G)} f(\mathcal G,x,\rho,\Psi) \bigg] = \mathbb E \bigg[ \sum_{x\in V(\mathcal G)} f(\mathcal G,\rho,x,\Psi) \bigg] \le 1.
$$
Since on $B_{\rm fin}$ we have that
$$
\sum_{x\in V(\mathcal G)} f(\mathcal G,x,\rho,\Psi) = \infty,
$$
it follows that $\mathbb P(B_{\rm fin})=0$.

We have thus shown that every infinite cluster of $\mathcal G[p']$ contains a special infinite cluster of $\mathcal G[p]$. Hence $\mathcal G[p']$ does not have infinitely many infinite clusters.

\medskip

This finishes the proof of Claim \ref{cl:unique}.
\end{proof}

\noindent The proof of Theorem \ref{thm:LRO} is thus complete.

\begin{proof}[Proof of Corollary \ref{cor:Phases}] This follows from Lemma \ref{lm:DelaunayExt} and Proposition \ref{prop:Bernoulli}. 
\end{proof}

\section{Vanishing uniqueness thresholds} \label{sec:vanishing}

In this section, we state and prove our main result.

\begin{theorem}[Vanishing uniqueness thresholds]\label{thm:vanishing}
    Let $G$ be a connected higher rank semisimple real Lie group with property (T) and let $(X,d_X)$ be its symmetric space. Then 
    \begin{equation*}
        \lim_{\lambda\to0} p_u(\lambda)=0.
    \end{equation*}
\end{theorem}

The proof of this result is done roughly as follows.
Given $\eps>0$, we aim to show that for every small enough $\lambda>0$ we have that 
\begin{equation*}
    \inf_{x,y\in X} \mathbb P_\eps^{(\lambda)} \big(x \leftrightarrow y) >0,
\end{equation*}
which is enough by Theorem~\ref{thm:LRO}.
To do this, we would like to employ Theorem~\ref{thm:threshold}, however, as $\lambda>0$ is small, we cannot guarantee that
\begin{equation*}
    \inf_{x,y\in \mathcal{B}_1(o)} \mathbb P_\eps^{(\lambda)} \big(x \leftrightarrow y) > p^*
\end{equation*}
for the threshold $0<p^*<1$ from Theorem~\ref{thm:threshold} with $R=1$.
In fact, even more severely, the density tends to $0$ as $\eps\to 0$. 
The idea is to define an auxiliary normal random closed set $\mathcal{Z}^{(\lambda)}_\eps$ that is a $G$-invariant thickening of the random closed set $\omega_\eps^{(\lambda)}$ so that 
    \begin{equation*}
        \inf_{x,y\in \mathcal{B}_1(o)} \mathbb P \big( x\overset{\mathcal{Z}^{(\lambda)}_\eps}{\longleftrightarrow} y \big)>p^*,
    \end{equation*}
and to then transfer the lower bound on the two-point function of $\mathcal{Z}^{(\lambda)}_\eps$ to a lower bound for~$\omega^{(\lambda)}_\eps$, by proving that it is at least a constant multiple of the two-point function of $\mathcal{Z}^{(\lambda)}_\eps$. We now provide the details.

\subsection{Construction of the auxiliary percolations} Throughout this section, consider the setting of Theorem \ref{thm:vanishing}. The goal is to prove the following result, which constructs the auxiliary normal random closed sets which will be used in the proof of Theorem \ref{thm:vanishing}.

\begin{theorem}\label{thm:AuxiliaryPercolation}
    Let $\lambda,R>0$ and $p\in(0,1)$.
    Then there is a random closed subset $\mathcal{Z}=\mathcal{Z}^{(\lambda)}_{p,R}$~of~$X$ defined as a $G$-equivariant measurable function of $\omega^{(\lambda)}_p$ such that the following hold a.s.
    \begin{enumerate}
        \item[{\rm(1)}] $\mathcal{Z}$ is normal,
        \item[{\rm(2)}] ${\rm vol}(\partial(\mathcal{Z}))=0$,
        \item[{\rm(3)}] for every cluster $\omega\subseteq \omega^{(\lambda)}_p$ there is a cluster $Z\subseteq \mathcal{Z}$ such that $        \omega^{(\lambda)}_p\cap Z=\omega$, in particular, $Z\cap\omega'=\emptyset$ for every cluster $\omega'\subset\omega_p^{(\lambda)}$ with $\omega'\ne\omega$,
        \item[{\rm(4)}] $\mathcal{B}_{R}(x)\subseteq \mathcal{Z}$ for every $x\in \omega^{(\lambda)}_p$ such that $\mathcal{B}_{6R}(x)$ only intersects one cluster of $\omega^{(\lambda)}_p$.
    \end{enumerate}
\end{theorem}

We first need some preparations. Let $C=C^{(\lambda)}_i$, $i\in \mathbb{N}$, be such that $C\subseteq \omega^{(\lambda)}_p$ and write $A_C$ for the union of cells $C^{(\lambda)}_j$ such that the distance of $C^{(\lambda)}_j$ to $C$ is at most $10R$, $C^{(\lambda)}_j\subseteq \omega^{(\lambda)}_p$, and $C^{(\lambda)}_j$ and $C$ lie in different clusters of $\omega^{(\lambda)}_p$. Note that $A_C$ is closed and $A_C\cap C=\emptyset$ a.s.

For $0\le \alpha\le 1$, define
\begin{equation}\label{eq:ThickeningCell}
    C_\alpha=\Big\{x\in X:\exists z\in C \text{ such that } \ d_X(x,z)\le \min\{3(1-\alpha)R,1/2(1-\alpha)d_X(z,A_C)\}\Big\}.
\end{equation}

The next result describes the basic properties of this thickening procedure.

\begin{proposition}\label{pr:PropertiesAuxiliaryPerc}
    Let $C$ and $A_C$ be as above.
    For every $0\le \alpha\le 2/3$, the following hold a.s.
    \begin{enumerate}
        \item[{\rm(1)}] $C_\alpha$ is compact and path-connected,
        \item[{\rm(2)}] $C\subseteq C_\alpha\subseteq C_\beta$ for every $0\le \beta<\alpha$,
        \item[{\rm(3)}] $\mathcal{B}_R(z)\subseteq C_\alpha$ for every $z\in C$ such that $\mathcal{B}_{6R}(z)\cap A_C=\emptyset$,
        \item[{\rm(4)}] $\partial(C_\alpha)\cap \partial(C_\beta)=\emptyset$ whenever $\alpha\not=\beta$,
        \item[{\rm(5)}] if $D\subseteq \omega^{(\lambda)}_p$ is a cell that belongs to a different cluster than $C$, then $C_\alpha\cap D_\beta=\emptyset$ for every $0<\beta\le 2/3$.
    \end{enumerate}
\end{proposition}
\begin{proof}
    (1). By Lemma~\ref{lm:BasicVoronoi}, we have that a.s.~$C$ and $A_C$ are compact.
    Let $(x_k)_k$ be a sequence in $C_\alpha$ such that $x_k\to x\in X$.
    By \eqref{eq:ThickeningCell}, there is a sequence $(z_k)_k\subseteq C$ such that
    $$d_X(x_k,z_k)\le \min\{3(1-\alpha)R,1/2(1-\alpha)d_X(z_k,A_C)\}$$
    for every $k\in \mathbb{N}$.
    Without loss of generality we may assume that $z_k\to z\in C$ as $C$ is compact.
    It follows that
    $$3(1-\alpha)R\ge d_X(x_k,z_k)\to d_X(x,z)$$
    as $k\to\infty$.
    This shows that $C_\alpha$ is compact in case that $A_C=\emptyset$. If $A_C\not=\emptyset$, then, as $A_C$ is compact, there is a sequence $(y_k)_k\subseteq A_C$ such that $d_X(z_k,A_C)=d_X(z_k,y_k)$ and $y_k\to y\in A_C$.
    Observe that
    $$d_X(z,A_C)= d_X(z,y)=\lim_{k\to\infty}d_X(z_k,y_k) \text{ and } d_X(z,x)=\lim_{k\to\infty} d_X(z_k,x_k).$$
    Consequently,
    $$d_X(x,z)\le 1/2(1-\alpha)d_X(z,A_C)$$
    by \eqref{eq:ThickeningCell}, which shows that $C_\alpha$ is compact.

    In order to see that $C_\alpha$ is path-connected, recall first that $C$ is path-connected by Lemma~\ref{lm:closed}.
    Given $x\in C_\alpha$ and $z\in C$ as in \eqref{eq:ThickeningCell}, we see that $x\in \overline{\mathcal{B}}_S(z)\subseteq C_\alpha$, where $\overline{\mathcal{B}}_S(z)$ is the closed ball of radius $S=\min\{3(1-\alpha)R,1/2(1-\alpha)d_X(z,A_C)\}$ around $z$.
    The claim then follows from the fact that closed balls in $X$ are path connected as $X$ is a geodesic space by Theorem~\ref{thm:BasicSymmetricSpace}~(2).
    
    \medskip

    (2).
    Follows directly from \eqref{eq:ThickeningCell}.

    \medskip

    (3).
    Let $z\in C$ be such that $\mathcal{B}_{6R}(z)\cap A_C=\emptyset$ and $x\in \mathcal{B}_R(z)$.
    Then we have that
    $$d_X(x,z)\le R\le \min\{R,1/6d_X(z,A_C)\}\le \min\{3(1-\alpha)R,1/2(1-\alpha)d_X(z,A_C)\}$$    
    for every $0\le \alpha\le 2/3$.

    \medskip

    (4).
    Suppose that $0\le \beta<\alpha$.
    Let $x\in C_\alpha$ and $z\in C$ be as in \eqref{eq:ThickeningCell} for $\alpha$.
    Define
    $$S_\alpha=(1-\alpha)\min\{3R,1/2d_X(z,A_C)\} \text{ and } S_\beta=(1-\beta)\min\{3R,1/2d_X(z,A_C)\}.$$
    Set $\eps=S_\beta-S_\alpha$ and note that $\eps>0$.
    We claim that $\mathcal{B}_\eps(x)\subseteq C_\beta$.
    Indeed, if $y\in \mathcal{B}_\eps(x)$, then
    $$d_X(y,z)\le d_X(x,z)+d_X(x,y)\le S_\alpha+\eps=S_\beta=(1-\beta)\min\{3R,1/2d_X(z,A_C)\}$$
    as needed.
    It follows that $\partial(C_\alpha)\subseteq C_\alpha\subseteq C_\beta\setminus \partial(C_\beta)$, which gives the claim.

    (5).
    Suppose that there is $x\in C_\alpha\cap D_\beta$.
    It follows from \eqref{eq:ThickeningCell} that $D\subseteq A_C$ and $C\subseteq A_D$.
    Consequently, again by \eqref{eq:ThickeningCell}, there are $z\in C$ and $y\in D$ such that
    $$d_X(x,z)\le 1/2(1-\alpha) d_X(z,y) \text{ and } d_X(x,y)\le 1/2(1-\beta) d_X(z,y).$$
    By the triangle inequality, we have that 
    $$d_X(z,y)\le d_X(x,z)+d_X(x,y)\le d_X(z,y)(1-1/2\alpha-1/2\beta),$$
    which is a contradiction.
\end{proof}

In order to apply the above thickening procedure, it will be important to choose $\alpha$ suitably for each cell. The justification is provided by the following result.

\begin{proposition}\label{pr:ComputableNumber}
    There is $0< \alpha(C)\le 2/3$, which can be computed in an isometry invariant measurable way from $\omega^{(\lambda)}_p$ and the cell $C\subseteq \omega^{(\lambda)}_p$, such that ${\rm vol}(\partial(C_{\alpha(C)}))=0$ a.s.
\end{proposition}
\begin{proof}
    By Proposition~\ref{pr:PropertiesAuxiliaryPerc}~(1) and (2), we have that $C_\alpha\subseteq C_0$ for every $0\le \alpha\le 2/3$ and $C_0$ is compact  a.s.
    In particular, ${\rm vol}(C_0)<\infty$, which implies by Proposition~\ref{pr:PropertiesAuxiliaryPerc} (4), that for every $\delta>0$, the set
    \begin{equation}\label{eq:ExceptionalPoints}
        P_\delta(C)=\{\alpha\in [0,2/3]: {\rm vol}(\partial(C_{\alpha}))>\delta\}
    \end{equation}
    is finite.
    The set $P_\delta(C)$ is moreover isometry invariant as every isometry preserves $\vol$ by Theorem~\ref{thm:BasicSymmetricSpace}~(1).

    \newcommand{\diam}{\operatorname{diam}}

    Let $\delta_n=2^{-n}$.
    The desired number $\alpha(C)$ is expressed as $\{\alpha(C)\}=\bigcap_{n\in \mathbb{N}} I_{C,n}$, where for every $n\in \mathbb{N}$ we have that $I_{C,n}$ is a closed non-trivial interval such that $I_{C,n+1}\subseteq I_{C,n}$ and $\diam(I_{C,n})\to 0$.
    Set $I_{C,0}=[0,2/3]$ and suppose that $I_{C,n}$ has been defined with the additional property that $I_{C,n}\cap P_{\delta_{n+1}}=\emptyset$.
    Let $[a,b]=I_{C,n}$ and define
    $$a'=\max\{\alpha\in I_{C,n}:\alpha\in P_{\delta_{n+2}}\cup\{a\}, \ \alpha<b\}.$$
    Set
    $$I_{C,n+1}=\left[\frac{3a'+1b}{4},\frac{1a'+3b}{4}\right].$$
    It can be easily checked that $\diam(I_{C,n+1})\le 1/2\diam(I_{C,n})$ and $I_{C_{n+1}}\subseteq I_{C,n}\setminus P_{\delta_{n+2}}$.
    It follows from the construction that $\alpha(C)$ is defined in an isometry invariant way, and that ${\rm vol}(\partial(C_{\alpha(C)}))=0$ as
    $$\alpha(C)\not\in \{0\}\cup \bigcup_{n\in \mathbb{N}} P_{\delta_n}=\{\alpha\in [0,2/3]: {\rm vol}(\partial(C_{\alpha}))>0\}.$$
    It remains to argue that the assignment $(C,\omega^{(\lambda)}_p)\mapsto \alpha(C)$ is measurable.
    To that end observe that the maps $(C,\omega^{(\lambda)}_p)\mapsto A_C$ and $(C,A_C,\alpha)\mapsto {\rm vol}(C_\alpha$) are measurable.
    The rest follows from the fact that the sets from \eqref{eq:ExceptionalPoints} are finite for every $\delta>0$.   
\end{proof}

With these preparations, we are now in a position to construct the auxiliary percolations.

\begin{proof}[Proof of Theorem~\ref{thm:AuxiliaryPercolation}]
    For each $C=C^{(\lambda)}_i\subset\omega_p^{(\lambda)}$, define $Z_C=C_{\alpha(C)}$, where $\alpha(C)$ is from Proposition~\ref{pr:ComputableNumber}.
    Set
    $$\mathcal{Z}=\bigcup_{C^{(\lambda)}_i\subseteq \omega_p^{(\lambda)}} Z_{C^{(\lambda)}_i}.$$
    It follows from the construction and Proposition~\ref{pr:ComputableNumber} that $\mathcal{Z}$ is defined in a $G$-equivariant and measurable way from $\omega^{(\lambda)}_p$.
    We now show that it satisfies (1)-(4).

    \medskip

    (1).
    Let $\omega$ be the cluster of $C$ in $\omega^{(\lambda)}_p$.
    By Lemma~\ref{lm:closed}, $\omega$ is closed and path-connected.
    By Proposition~\ref{pr:PropertiesAuxiliaryPerc}~(1), every $C_{\alpha(C)}$ is compact and path-connected.
    Recall that every bounded set of $X$ is covered by finitely many Voronoi cells a.s.~by Lemma~\ref{lm:closed}.
    As $C_\alpha\subseteq \bigcup_{x\in C} \overline{\mathcal{B}}_{3R}(x)$ by \eqref{eq:ThickeningCell}, it follows that a.s.
    $$Z_\omega=\bigcup_{C^{(\lambda)}_i\subseteq \omega} Z_{C^{(\lambda)}_i}$$
    is closed and path connected and every bounded set is intersected by at most finitely many sets of the form $Z_\omega$.
    In particular, $\mathcal{Z}$ is a normal random closed subset.

    (3).
    Observe that by the definition, we have that $\omega\subseteq Z_\omega$.
    Suppose that $\omega,\omega'\subseteq \omega^{(\lambda)}_p$ are different clusters.
    Then $Z_\omega\cap Z_{\omega'}=\emptyset$.
    Indeed, otherwise, by the definition, there are Voronoi cells $C\subseteq \omega$ and $D\subseteq \omega'$ such that $C_{\alpha(C)}\cap D_{\alpha(D)}\not=\emptyset$.
    This contradicts Proposition~\ref{pr:PropertiesAuxiliaryPerc}~(6) as $\alpha(C),\alpha(D)>0$ by Proposition~\ref{pr:ComputableNumber}.

    (2).
    By Lemma~\ref{lm:closed}, every bounded set of $X$ is covered by finitely many Voronoi cells a.s.
    Consequently, we have that a.s.
    $$\partial(\mathcal{Z})\subseteq \bigcup_{C\subseteq \omega^{(\lambda)}_p} \partial(C_{\alpha(C)}),$$
    which implies that ${\rm vol}(\partial(\mathcal{Z}))=0$ a.s.

    (4).
    The fact that $\mathcal{B}_R(x)\subseteq \mathcal{Z}$ whenever $x\in \omega^{(\lambda)}_p$ is such that $\mathcal{B}_{6R}(x)$ intersect a single cluster of $\omega^{(\lambda)}_p$ follows directly from Proposition~\ref{pr:PropertiesAuxiliaryPerc} (3).
\end{proof}

\subsection{Proof of the main result.}
We are ready to finish the proof of Theorem~\ref{thm:vanishing}.

\begin{proof}[Proof of Theorem~\ref{thm:vanishing}]
Let $\eps>0$. We show that 
\begin{equation*}
    \limsup_{\lambda\to0}p_u(\lambda)\leq \eps,
\end{equation*}
which implies the conclusion of the theorem as $\eps>0$ was arbitrary. By Theorem \ref{thm:LRO}, it suffices to show that there exists $\lambda_0>0$ such that for all $\lambda<\lambda_0$
\begin{equation} \label{equ:vanishing1}
    \inf_{x,y\in X} \mathbb P_\eps^{(\lambda)} \big(x \leftrightarrow y) >0.
\end{equation}
We prove \eqref{equ:vanishing1} in two steps.

\medskip

{\noindent\em Step 1} (Long-range order for auxiliary percolations).
By Theorem \ref{thm:threshold}, there exists $p^*<1$ such that every $G$-invariant normal random closed subset $\mathcal Z$ of $X$ with $\mathbb P(o \in \partial \mathcal Z)=0$, or equivalently ${\rm vol}(\partial \mathcal Z) = 0$ a.s. by Lemma~\ref{lm:boundary}, satisfies
    \begin{equation*}
        \inf_{x,y\in \mathcal{B}_1(o)} \mathbb P \big( x\overset{\mathcal Z}{\longleftrightarrow} y \big)>p^* \quad \Rightarrow \quad \inf_{x,y\in X} \mathbb P \big( x\overset{\mathcal Z}{\longleftrightarrow} y \big)>0.
    \end{equation*}
Choose $N\geq1$ such that
\begin{equation} \label{equ:vanishing2}
    \big(1-(1-\eps)^N\big) \big((1+p^*)/2\big) > p^*.
\end{equation}
By Theorem \ref{thm:A1}, we may then choose $S>0$ and $\lambda_1>0$ such that 
\begin{equation} \label{equ:vanishing3}
    \inf_{\lambda<\lambda_1} \mathbb P \left(\forall 1\le i\le N \ \mathcal{B}_{S}(o)\cap C^{(\lambda)}_i\not=\emptyset\right) \ge p^* + (1-p^*)3/4.
\end{equation}
Assume, without loss of generality, that $S>2$.
For $\lambda,R>0$, where $R=2S$, let $\mathcal{Z}^{(\lambda)}:=\mathcal{Z}^{(\lambda)}_{\eps,R}$ be the normal random closed set that satisfies the conclusion of Theorem~\ref{thm:AuxiliaryPercolation}.

\medskip

\begin{claim} \label{cl:LROAuxiliary}
    There exists $\lambda_0>0$ such that 
    \begin{equation} \label{equ:LROxi}
        \inf_{x,y\in X} \mathbb P \big( x\overset{\mathcal{Z}^{(\lambda)}}{\longleftrightarrow} y \big)>0
    \end{equation}
    for every $\lambda<\lambda_0$.
\end{claim}

\begin{proof}[Proof of Claim \ref{cl:LROAuxiliary}]
    By Theorem \ref{thm:A2}, we have that
        $$
        \liminf_{\lambda\to0} \, \, \mathbb{P} \bigg(\forall i,j \in \mathbb{N} \  C^{(\lambda)}_i\cap \mathcal{B}_{7R}(o)\not=\emptyset\not= C^{(\lambda)}_j\cap \mathcal{B}_{7R}(o) \  \Rightarrow \  C^{(\lambda)}_i\cap C^{(\lambda)}_j\not=\emptyset \bigg) = 1.
        $$
    Combined with \eqref{equ:vanishing3}, it follows that 
        \begin{align} \label{equ:vanishing4}
        \mathbb{P} \bigg(\forall i,j \in \mathbb{N} & \  C^{(\lambda)}_i\cap \mathcal{B}_{7R}(o)\not=\emptyset\not= C^{(\lambda)}_j\cap \mathcal{B}_{7R}(o) \  \Rightarrow \  C^{(\lambda)}_i\cap C^{(\lambda)}_j\not=\emptyset \ \mbox{and} \\
        & \forall 1\le i\le N \ \mathcal{B}_{S}(o)\cap C^{(\lambda)}_i\not=\emptyset \bigg) > (1+p^*)/2 \nonumber
        \end{align}
        for all sufficiently small $\lambda>0$, say $\lambda\leq\lambda_0$. But conditional on the event in \eqref{equ:vanishing4}, independently coloring each cell black with probability $\eps$ yields a unique black cluster $\omega$ in $\mathcal{B}_{7R}(o)$ that has non-empty intersection with $\mathcal{B}_S(o)$ with probability at least $1-(1-\eps)^N$.
        In particular, conditioned on this event, we have that $\mathcal{B}_1(o)\subseteq \mathcal{Z}^{(\lambda)}$.
        This is because there is $x\in \mathcal{B}_S(o)\cap \omega$ such that $\mathcal{B}_{6R}(x)\subseteq \mathcal{B}_{7R}(o)$ intersects a single cluster of $\omega^{(\lambda)}_\eps$, which implies by Theorem~\ref{thm:AuxiliaryPercolation}~(4) that $\mathcal{B}_1(o)\subseteq \mathcal{B}_{R}(x)\subseteq \mathcal{Z}$ as $2R=S>2$.
        Recalling the choice of $N$ from \eqref{equ:vanishing1}, we obtain that 
        $$
        \inf_{x,y\in \mathcal{B}_1(o)} \mathbb P \big( x\overset{\mathcal{Z}^{(\lambda)}}{\longleftrightarrow} y \big) > p^*
        $$
        for all $\lambda\le \lambda_0$.
        Combined with Theorem~\ref{thm:AuxiliaryPercolation} and the definition of $p^*$, Claim \ref{cl:LROAuxiliary} follows from Theorem~\ref{thm:threshold}.
\end{proof}

{\noindent\em Step 2} (Comparison). We now show that long-range order for $\mathcal{Z}^{(\lambda)}$ implies the same for $\omega_\eps^{(\lambda)}$. Since the former was shown in \eqref{equ:LROxi} for $\lambda\le\lambda_0$, this will prove \eqref{equ:vanishing1}.

More precisely, we show that for every $\lambda$, 
the two-point function of $\omega_\eps^{(\lambda)}$ is bounded from below by the two-point function of  $\mathcal{Z}^{(\lambda)}$ times a multiplicative factor which depends only on $\lambda$ and $R$. Set
    \begin{equation*}
        c:= c(\lambda,R):= \mathbb P \big( \mathcal B_{3R}(o) \subset \omega_\eps^{(\lambda)}\big) > 0.
    \end{equation*}
    The Harris-FKG-inequality, see Lemma \ref{lm:FKG}, implies that
    \begin{equation*}
        \mathbb P \Big( \mathcal B_{3R}(x) \overset{\omega_\eps^{(\lambda)}}{\longleftrightarrow} \mathcal B_{3R}(y), \mathcal B_{3R}(x) \subset \omega_\eps^{(\lambda)}, \mathcal B_{3R}(y) \subset \omega_\eps^{(\lambda)} \Big) \geq c^2 \mathbb P \Big[ \mathcal B_{3R}(x) \overset{\omega_\eps^{(\lambda)}}{\longleftrightarrow} \mathcal B_{3R}(y) \Big],
    \end{equation*}
    where $\mathcal B_{3R}(x) \overset{\omega_\eps^{(\lambda)}}{\longleftrightarrow} \mathcal B_{3R}(y)$ is the event that there is a path in $\omega_\eps^{(\lambda)}$ that connects a point from $\mathcal B_{3R}(x)$ to a point from $\mathcal B_{3R}(y)$. 
    Hence
    \begin{align*}
    \mathbb P \Big( x \overset{\mathcal{Z}^{(\lambda)}}{\longleftrightarrow} y\Big) & \leq \mathbb P \Big( \mathcal B_{3R}(x) \overset{\omega_\eps^{(\lambda)}}{\longleftrightarrow} \mathcal B_{3R}(y) \Big) \\
    & \leq \, c^{-2} \, \mathbb P \Big( \mathcal B_{3R}(x) \overset{\omega_\eps^{(\lambda)}}{\longleftrightarrow} B_{3R}(y), \mathcal B_{3R}(x) \subset \omega_\eps^{(\lambda)}, \mathcal B_{3R}(y) \subset \omega_\eps^{(\lambda)} \Big) \\
    & \leq \, c^{-2} \, \mathbb P \Big( x \overset{\omega_\eps^{(\lambda)}}{\longleftrightarrow} y\Big),
    \end{align*}
    where the first inequality follows from the construction of $\mathcal{Z}^{(\lambda)}$.
    Indeed, by \eqref{eq:ThickeningCell} there are $x',y'\in \omega^{(\lambda)}_\eps$ that are in the same cluster of $\mathcal{Z}^{(\lambda)}$ as $x,y$ (hence $x'$ and $y'$ are in the same cluster of $\omega^{(\lambda)}_\eps$ by Theorem~\ref{thm:AuxiliaryPercolation}~(3)) and that satisfy $d_X(x,x'),d_X(y,y')<3R$.
    The proof of Theorem~\ref{thm:vanishing} is thus complete.
\end{proof}

\section{Sparse factor graphs with unique infinite cluster} \label{sec:sparse}

In this section, we apply Theorem \ref{thm:vanishing} to prove the existence of sparse factor graphs with a unique infinite cluster for the Poisson point process on a connected higher rank semisimple real Lie group $G$ with property (T) (see Corollary \ref{cor:FIIDsparseUSectionEight}). This fact is a consequence of the following lemma, which provides a similar statement for the Poisson point process on the symmetric space $X$ of $G$.

To shorten notation, we say that a Poisson point process $Y$ on $X$, resp.~$\Pi$ on $G$, has {\em intensity} equal to a constant $c>0$, if its intensity measure is $c{\rm vol}$, resp.~$c m_G$ with $m_G$ denoting left-invariant Haar measure on $G$.

\begin{lemma}[FIID sparse unique infinite clusters on the symmetric space] \label{lm:sparseX}
    Let $G$ be a connected higher rank semisimple real Lie group with property (T) and let $(X,d_X)$ be its symmetric space. Let~$\mathbf Y$ be a Poisson point process $Y$ on $X$ of intensity $1$ together with iid ${\rm Unif}[0,1]$ marks. Then, for every $\eps>0$, there exists $\lambda>0$ such that if $\mathbf Y^{(\lambda)}$ is an independent Poisson point process on $X$ of intensity $\lambda$ equipped with iid ${\rm Unif}[0,1]$ marks,  there is a graph $\mathcal H$ on $Y$ defined as an isometry-equivariant factor of $(\mathbf Y,\mathbf Y^{(\lambda)})$ with a unique infinite cluster and $\mathbb E\big[ {\rm deg}_{\mathcal H(Y_0)}(o) \big]\le\eps$, where $Y_0:=Y\cup\{o\}$.
\end{lemma}
\begin{proof}
    Let $\eps>0$ be fixed. To define the factor graph $\mathcal H$, we will use Poisson-Voronoi percolation with suitable parameters $\lambda$ and $p$, which are obtained as follows. By Lemma \ref{lm:FinExpDegree}, $\mathbb E [{\rm deg}_{\mathcal G(Y_0)}(\rho)]<\infty,$ 
where $\mathcal G(Y_0)$ denotes the Delaunay graph of $Y_0$ and $\rho=C_o$ denotes the cell of the origin. In particular, by Lemma~\ref{lm:BasicVoronoi}, there exists $R>0$ such that 
\begin{equation} \label{assumption:R}
\mathbb E \big[{\rm deg}_{\mathcal G(Y_0)}(\rho) \, \mathbf 1\{ C_o\not\subset \mathcal B_R(o) \} \, \big] \le \eps/3
\end{equation}
By Proposition \ref{pr:UpperBoundBall}, there exist $n_0\in \mathbb{N}$ and $\lambda_0>0$ such that 
\begin{equation} \label{assumption:lambda}
\sup_{0<\lambda<\lambda_0} \mathbb{P}^{(\lambda)}(\# \text{ Voronoi cells intersecting } \mathcal{B}_R(o)\ge n_0 ) < \frac{\eps}{3 \, \mathbb E [{\rm deg}_{\mathcal G(Y_0)}(\rho)]}.
\end{equation}
Let $n_0\in\mathbb N$ and $\lambda_0>0$ be as above. Choose $p\in(0,1)$ such that
\begin{equation} \label{assumption:p}
\sum_{i=1}^{n_0-1} 1-(1-p)^i \le \frac{\eps}{3 \, \mathbb E [{\rm deg}_{\mathcal G(Y_0)}(\rho)]}.
\end{equation}
Finally, using Theorem \ref{thm:vanishing}, choose $\lambda<\lambda_0$ such that $p_u(\lambda)<p$. Note that with these choices of parameters, $(\lambda,p)$-Poisson-Voronoi percolation has a unique unbounded cluster almost surely.

Fix $\lambda$ and $p$ be as above. Let $\mathbf Y^{(\lambda)}$ be an independent Poisson point process $Y^{(\lambda)}$ on $X$ of intensity $\lambda$ together with iid ${\rm Unif}[0,1]$ marks and let $\omega_p^{(\lambda)}$ denote the configuration of Poisson-Voronoi percolation on~$Y^{(\lambda)}$ obtained as previously using the marks. Define the graph $\mathcal H$ on $Y$ as follows. For $x,y\in Y$, let $[x,y]$ be an edge in $\mathcal H$ if and only if both of the following conditions are satisfied:
\begin{itemize}
\item[(i)] $C_x$ and $C_y$ are neighbors in $\mathcal G(Y)$,
\item[(ii)] $C_x$ and $C_y$ intersect the unique unbounded cluster $C_{p,\infty}^{(\lambda)}$ of $\omega_p^{(\lambda)}$.
\end{itemize}
Note that $\mathcal H$ is locally finite because ${\rm deg}_{\mathcal H}(y)\le {\rm deg}_{\mathcal G(Y)}(y)<\infty$ for every $y\in Y$. Moreover, it follows from the definition that $\mathcal H$ is an isometry-equivariant factor of $(\mathbf Y,\mathbf Y^{(\lambda)})$.

\begin{claim} The graph $\mathcal H$ has a unique infinite cluster.
\end{claim}
\begin{proof} Every $y\in Y$ such that $C_y$ does not intersect $C_{p,\infty}^{(\lambda)}$ is an isolated vertex of $\mathcal H$. On the other hand, note that the set of vertices $y \in Y$ such that $C_y$ intersects $C_{p,\infty}^{(\lambda)}$ is infinite because $C_{p,\infty}^{(\lambda)}$ is unbounded and every Voronoi cell is bounded. We claim that it is also connected. Indeed, let $x,y\in Y$ such that  $C_x$ and $C_y$ intersect $C_{p,\infty}^{(\lambda)}$. Consider any path joining $x$ to $y$ in $C_{p,\infty}^{(\lambda)} \cup C_x\cup C_y$.
Following the nuclei in $Y$ corresponding to the cells traversed when going along this path yields a path connecting $x$ to $y$ in $\mathcal H$.
\end{proof}

It remains to show that $\mathbb E\big[ {\rm deg}_{\mathcal H(Y_0)}(o) \big]\le\eps$. To see this, first observe that it follows from the definition of $\mathcal H$, \eqref{assumption:R} and independence of $Y_0$ and $\omega_p^{(\lambda)}$ that
\begin{align*}
\mathbb E\big[ {\rm deg}_{\mathcal H(Y_0)}(o) \big] & = \mathbb E\big[ {\rm deg}_{\mathcal H(Y_0)}(o)  \, \mathbf 1\{ C_o\not\subset \mathcal B_R(o) \} \, \big] + \mathbb E\big[ {\rm deg}_{\mathcal H(Y_0)}(o)  \, \mathbf 1\{ C_o\subset \mathcal B_R(o) \} \, \big] \\
& \le \eps/3 + \mathbb E \big[{\rm deg}_{\mathcal G(Y_0)}(\rho) \, \mathbf 1\{ C_o\subset \mathcal B_R(o) \} \, \mathbf 1\{\mathcal B_R(o)\cap \omega_p^{(\lambda)}\ne \emptyset \} \,\big] \\
& = \eps/3 + \mathbb E \big[{\rm deg}_{\mathcal G(Y_0)}(\rho) \, \mathbf 1\{ C_o\subset \mathcal B_R(o) \} \big] \mathbb P\big( \mathcal B_R(o)\cap \omega_p^{(\lambda)}\ne \emptyset\big) \\
& \le \eps/3 + \mathbb E \big[{\rm deg}_{\mathcal G(Y_0)}(\rho) \big] \mathbb P\big( \mathcal B_R(o)\cap \omega_p^{(\lambda)}\ne \emptyset\big).
\end{align*}

Let $N_R^{(\lambda)}(o)$ denote the number of Voronoi cells of $Y^{(\lambda)}$ intersecting $\mathcal B_R(o)$. By \eqref{assumption:lambda} and \eqref{assumption:p} 
\begin{align*}
\mathbb P\big( \mathcal B_R(o)\cap \omega_p^{(\lambda)}\ne \emptyset \big) & = \mathbb P\big( \mathcal B_R(o)\cap \omega_p^{(\lambda)}\ne \emptyset, N_R^{(\lambda)}(o) \ge n_0 \big) + \mathbb P\big( \mathcal B_R(o)\cap \omega_p^{(\lambda)}\ne \emptyset, N_R^{(\lambda)}(o) < n_0 \big) \\
& \le  \frac{\eps}{3 \, \mathbb E [{\rm deg}_{\mathcal G(Y_0)}(\rho)]} + \sum_{i=1}^{n_0-1} \mathbb P\big( N_R^{(\lambda)}(o) = i \big)\big(1-(1-p)^i) \\
& \le \frac{2\eps}{3 \, \mathbb E [{\rm deg}_{\mathcal G(Y_0)}(\rho)]}.
\end{align*}
Combining the above two inequalities proves that $\mathbb E\big[ {\rm deg}_{\mathcal H(Y_0)}(o) \big] \le \eps$, which finishes the proof of the lemma.
\end{proof}

\begin{remark}\label{rm:DegreeAsymp}
    In the setting of Lemma \ref{lm:sparseX}, an arguably even more intuitive approach would be to first connect all points of $Y$ falling into the same black Voronoi cell of $Y^{(\lambda)}$ by a minimal path and then connect neighboring black Voronoi cells by an edge with endpoints chosen uniformly among its $Y$-points. For this strategy to work, we would need that the expected number of neighbors of each cell is comparable to the expected volume. This leads to the following question which we leave open: Does \eqref{equ:DegreeAsymp} hold with $\alpha=1$?
\end{remark}

We are now in a position to prove our main application of Theorem \ref{thm:vanishing} to the study of factor graphs of the Poisson point process on $G$. Recall that $1_G$ denotes the identity element on $G$.

\begin{corollary}[FIID sparse unique infinite clusters] \label{cor:FIIDsparseUSectionEight} Let $G$ be a connected higher rank semisimple real Lie group with property (T). Let~$\Pi$ be the Poisson point process on $G$ of intensity $1$ equipped with iid ${\rm Unif}[0,1]$ marks. Then, for every $\eps>0$, there is a $G$-equivariant factor graph $\mathcal H$ of $\Pi$ with a unique infinite cluster and $\mathbb E\big[ {\rm deg}_{\mathcal H(\Pi_0)}(1_G) \big]\le\eps$ for $\Pi_0=\Pi\cup\{1_G\}$.
\end{corollary}
\begin{proof} Fix $\eps>0$. Choose $\lambda>0$ as in Lemma \ref{lm:sparseX}. As observed for instance in the proof of \cite[Theorem 2.25]{FMW23}, the Poisson point process $\Pi$ admits $(\Pi,\Pi^{(\lambda)})$ as a $G$-equivariant factor, where $\Pi^{(\lambda)}$ is an independent Poisson point process on $G$ of intensity $\lambda$ equipped with iid ${\rm Unif}[0,1]$-marks (this follows from the construction in \cite[Proposition 5.1]{AM22}). Let $(\mathbf Y,\mathbf Y^{(\lambda)})$ denote the image of $(\Pi,\Pi^{(\lambda)})$ under the canonical projection $p:  G\to X$. Then $\mathbf Y$ is a Poisson point process $Y$ on $X$ of intensity $1$ together with iid ${\rm Unif}[0,1]$ marks and $\mathbf Y^{(\lambda)}$ is an independent Poisson point process on $X$ of intensity $\lambda$ equipped with iid ${\rm Unif}[0,1]$ marks. By choice of $\lambda$, there is a graph $\mathcal H'$ on $Y$ defined as an isometry-equivariant factor of $(\mathbf Y,\mathbf Y^{(\lambda)})$ with a unique infinite cluster and $\mathbb E\big[ {\rm deg}_{\mathcal H'(Y_0)}(o) \big]\le\eps$, where $Y_0:=Y\cup\{o\}$. We now define a graph $\mathcal H$ on $\Pi$ as follows. For $g,h\in \Pi$, let $[g,h]$ be an edge in $\mathcal H$ if and only if $[p(g),p(h)]$ is an edge in $\mathcal H'$. Then $\mathcal H$ defines a $G$-equivariant factor graph of $\Pi$. Since $m_G(K)=0$, there are almost surely no multiple points $g\in \Pi$ which project to the same point in $\mathbf Y$. Note also that $p(1_G)=o\in X$. It follows that $\mathcal H$ has a unique infinite cluster and $\mathbb E\big[{\rm deg}_{\mathcal H(\Pi_0)}(1_G) \big]\le\eps$. 
 \end{proof}

\section{FIID sparse unique infinite clusters} \label{sec:Lattices}

We show that any Cayley graph of a co-compact lattice $\Gamma$ in a connected higher rank semisimple real Lie group $G$ with property (T) gives a positive answer to Question \ref{q:FIIDSparseUnique}.
Recall that a \emph{co-compact lattice} is a discrete subgroup with the property that there is a compact set $B\subseteq G$ such that $G=\bigcup_{\gamma\in \Gamma} B\gamma$ (or equivalently $G=\bigcup_{\gamma\in \Gamma} \gamma B$).
The fact that co-compact lattices indeed exist, for example, in $\group$ for $n\ge 3$, follows from the classical result of Borel and Harish-Chandra \cite{BHC62}. By the \v{S}varc--Milnor lemma every co-compact lattice $\Gamma$ is finitely generated.
We recall that if $S$ is a finite symmetric generating set of $\Gamma$, then the \emph{(right) Cayley graph} ${\rm Cay}(\Gamma,S)$ is a graph on $\Gamma$, where $\gamma_0,\gamma_1\in \Gamma$ form an edge if there is $s\in S$ such that $\gamma_0s=\gamma_1$.

\begin{theorem}[Cayley graphs with the FIID sparse unique infinite cluster property]  \label{thm:FIIDSUICP}
    Let $\Gamma\subset G$ be a co-compact lattice in a connected higher rank semisimple real Lie group $G$ with property~(T) and let ${\rm Cay}(\Gamma,S)$ be the Cayley graph of $\Gamma$ with respect to a finite symmetric generating set $S$.
    Then, for every $\eps>0$, there is a $\Gamma$-equivariant FIID bond percolation $\omega$ on ${\rm Cay}(\Gamma,S)$ with a unique infinite cluster and $\mathbb E\big[ {\rm deg}_{\omega}(1_G) \big]\le\eps$.
\end{theorem}

We remark that this result and all the results in this section hold under the weaker assumption that the connected higher rank semisimple real Lie group $G$ satisfies $p_u(\lambda)\to 0$ as $\lambda\to 0$.

Theorem~\ref{thm:FIIDSUICP} will be a consequence of the following lemma about factor graphs on the lattice.

\begin{lemma}\label{lm:sparseLattice} Let $\Gamma\subset G$ be a co-compact lattice in a connected higher rank semisimple real Lie group $G$ with property (T). 
Then, for every $\eps>0$, there is a $\Gamma$-equivariant factor of iid bounded degree graph $\mathcal H$ on $\Gamma$ with a unique infinite cluster and $\mathbb E\big[ {\rm deg}_{\mathcal H}(1_G) \big]\le\eps$.
\end{lemma}

Before proving Lemma \ref{lm:sparseLattice}, we need the following preparation.

\begin{lemma}\label{lm:CocompactLattice}
    Let $G$ be a connected semisimple real Lie group, $X$ be its symmetric space  and $\Gamma\subseteq G$ be a discrete subgroup.
    Then the following holds.
    \begin{enumerate}
        \item[{\rm(1)}] There is $k\in \mathbb{N}$ such that the restriction $p_\Gamma$ of the projection $p:G\to X$ to $\Gamma$ is $k$-to-$1$.
        Moreover, if $\Gamma$ is torsion-free, then $p_\Gamma$ is injective.
        \item[{\rm(2)}] The set $p_\Gamma(\Gamma)\subset X$ induces a $\Gamma$-equivariant Voronoi diagram $\{V_\gamma\}_{\gamma\in \Gamma}$ as a multi-set.
        Moreover, if $\Gamma$ is co-compact, then there is $R>0$ and $d\in \mathbb{N}$ with the property that $V_\gamma\subseteq \mathcal{B}_R(\gamma o)$ for every $\gamma\in \Gamma$ and every cell intersect $k(d+1)-1$ many other cells (where $k\in \mathbb{N}$ is from {\rm(1)}).
        \item[{\rm(3)}] Let $\mathbf Z=(Z(\gamma))_{\gamma\in\Gamma}$ be a collection of iid ${\rm Unif}[0,1]$ random variables and $\lambda>0$ and $p\in(0,1]$.
        Then $\omega_p^{(\lambda)}$ may be realized as a $\Gamma$-equivariant factor of $\mathbf Z$.        
    \end{enumerate}
\end{lemma}
\begin{proof}
    (1). As $X=G/K$, where $K$ is a compact subgroup and $\Gamma$ is discrete, we see that the group $H=K\cap \Gamma$ must be finite.
    In particular, if $\Gamma$ is torsion free, we have that $H=\{1_G\}$.
    The rest follows from the fact that $p^{-1}_{\Gamma}(\gamma K)=\{\gamma h:h\in H\}$ for every $\gamma\in \Gamma$.

    (2). Let $\gamma_0,\theta\in \Gamma$ and $x\in X$.
    Then we have that
    $$d_X(x,p_\Gamma(\gamma_0))=d_X(\theta \cdot x,\theta\cdot p_\Gamma(\gamma_0))=d_X(\theta \cdot x,p_\Gamma(\theta\gamma_0)).$$
    It follows that $\theta \cdot V_{\gamma_0}= V_{\theta\gamma_0}$.
    In particular, if $p_\Gamma(\gamma)=p_\Gamma(\gamma')$ for some $\gamma,\gamma'\in \Gamma$, then $V_{\theta\gamma}=\theta\cdot V_{_\gamma}=\theta\cdot V_{\gamma'}$ for every $\theta\in \Gamma$.

    Suppose that $\Gamma$ is co-compact.
    We show that $V_o:=V_{1_G}$ is compact, which implies the existence of the desired $R>0$ by the fact that $\{V_\gamma\}_{\gamma\in \Gamma}$ is $\Gamma$-equivariant.
    Assume for a contradiction that $V_o$ is not compact.
    Then there is a sequence $\{x_n\}_{n\in \mathbb{N}}\subseteq V_o$ such that $d_X(o,x_n)\to \infty$.
    Fix a sequence $\{g_n\}_{n\in \mathbb{N}}\subseteq G$ such that $p(g_n)=g_n o=x_n$ for every $n\in \mathbb{N}$.
    As $\Gamma$ is co-compact, there is a compact set $B\subseteq G$ such that $\bigcup \gamma\cdot B=G$.
    Consequently, after passing to a subsequence if necessary, there are $g\in B$ and $\{\gamma_n\}_{n\in \mathbb{N}}\subseteq \Gamma$ such that $\gamma_ng_n\to g$ in $G$.

    Let $x=go$.
    Then $d_X(o,x)\in [0,\infty)$ and $\gamma_n\cdot x_n\to x$ in $X$ by continuity of the projection.
    By left-invariance of the metric, we have that
    $$d_X(\gamma_n^{-1}o,x_n)=d_X(o,\gamma_n\cdot x_n) =d_X(o,\gamma_ng_no) \to d_X(o,x).$$
    In particular, there is $S>0$ such that $d_X(\gamma_n^{-1}o,x_n)<S$ for every $n\in \mathbb{N}$.
    Since $d_X(o,x_n)\to \infty$, this contradicts $\{x_n\}_{n\in \mathbb{N}}\subseteq V_o$. We obtain that $V_0$ is compact. 

    Finally, we claim that $V_o$ intersects finitely many cells $V_\gamma$, where $\gamma\in \Gamma$.
    Indeed, as $G\curvearrowright X$ is proper by Theorem~\ref{thm:BasicSymmetricSpace}~(1), we have that $\{g\in G:g\cdot V_o\cap V_o\not=\emptyset\}$ is compact.
    In particular, the set $\{\gamma\in \Gamma: \gamma\cdot V_o\cap V_o\not=\emptyset\}$ is finite as $\Gamma$ is discrete.
    The claim now follows from the fact that $\gamma\cdot V_o=V_{\gamma}$ for every $\gamma\in \Gamma$.
    The existence of $d\in \mathbb{N}$ then follows from the $\Gamma$-equivariance of $\{V_\gamma\}_{\gamma\in \Gamma}$.

    (3).
    This can be done as in \cite[Proposition 5.1]{AM22}. We provide the details for completeness.
    Namely, let $\{W_\gamma\}_{\gamma\in \Gamma}$ be the Voronoi diagram of $\Gamma$ in $G$ (defined with respect to the canonical left-invariant metric on $G$).
    It can be verified, as in the previous paragraph, that $\{W_\gamma\}_{\gamma\in \Gamma}$ is $\Gamma$-equivariant and consists of compact cells.
    Let $f\colon[0,1]\to \mathbf M(X)$ be a measurable map such that if $Z\sim{\rm Unif}[0,1|$, then $f(Z)$ is a Poisson point process with intensity $\lambda m_G$ on $G$ with independent ${\rm Unif}[0,1]$ marks.
    Recall that the diagonal action of $G$ on marked configurations shifts the points together with their iid marks. 
    Define a point process
    $$\Theta(\mathbf{Z})=\bigcup_{\gamma\in \Gamma} \gamma\cdot \left(f(Z_{\gamma})\cap W_{1_\Gamma}\right)=\bigcup_{\gamma\in \Gamma} (\gamma\cdot f(Z_{\gamma}))\cap W_{\gamma}.$$
    It can be easily verified, noting that $m_G(\partial(W_\gamma))=0$ for every $\gamma\in \Gamma$ (see e.g.~\cite[Section 3.3]{AM22}), 
    that $\Theta(\mathbf{Z})$ is a Poisson point process with intensity $\lambda m_G$ and independent ${\rm Unif}[0,1]$ marks. Moreover, it is a $\Gamma$-equivariant factor of $\mathbf Z$.

    Finally, we apply the projection $p:G\to X$ to $\Theta(\mathbf{Z})$.
    To be more concrete, we apply $p$ to the first coordinate of the marked points in $\Theta(\mathbf{Z})$ and leave the marks intact.
    Then $\mathbf{Y}^{(\lambda)}$ and $p(\Theta(\mathbf{Z}))$ have the same distribution.
    The rest follows from the fact that the projection $p$, and hence $\omega^{(\lambda)}_p$ as a factor of $p(\Theta(\mathbf{Z}))$, is $\Gamma$-equivariant.
\end{proof}

We now provide the proof of Lemma \ref{lm:sparseLattice}. The proof is similar to that of Lemma \ref{lm:sparseX} given above.

\begin{proof}[Proof of Lemma \ref{lm:sparseLattice}]
By Lemma \ref{lm:CocompactLattice}~(1), there is $k\in \mathbb{N}$ such that $p_\Gamma:\Gamma\to X$ is $k$-to-$1$. 
Fix $\eps>0$. By Lemma \ref{lm:CocompactLattice}~(2), there exist $R>0$ and $d\in \mathbb{N}$ such that all cells of the Voronoi diagram $\{V_\gamma\}_{\gamma\in \Gamma}$ induced by $p_\Gamma(\Gamma)$ have $k(d+1)-1$ many neighbors and are contained in the ball of radius $R>0$ around their nucleus.
Set $c=k(d+1)-1$.
By Proposition~\ref{pr:UpperBoundBall}, there exist $n_0\in\mathbb N$ and $\lambda_0>0$ such that 
$$
\sup_{0<\lambda<\lambda_0} \mathbb{P}^{(\lambda)}(\# \text{ cells intersecting } \mathcal{B}_R(o)\ge n_0 ) < \eps/(2c).
$$
Choose $p\in(0,1)$ such that
$$
\sum_{i=1}^{n_0-1} 1-(1-p)^i \le \eps/(2c).
$$
Finally, using Theorem \ref{thm:vanishing}, choose $\lambda<\lambda_0$ such that $p_u(\lambda)<p$. For the rest of this proof, fix these parameters $\lambda$ and $p$.
By Lemma~\ref{lm:CocompactLattice}~(3), $\omega_p^{(\lambda)}$ may be realized as a $\Gamma$-equivariant factor of~$\mathbf Z$.
We now define the graph $\mathcal H$ on $\Gamma$ as follows. For $g,h\in\Gamma$, let $[g,h]$ be an edge in $\mathcal H$ if and only if both of the following conditions are satisfied: 
\begin{itemize}
\item[(i)] $V_{g}$ and $V_{h}$ are neighbors in $\{V_\gamma\}_{\gamma\in \Gamma}$,
\item[(ii)] $V_{g}$ and $V_{h}$ intersect the unique unbounded cluster $C_{p,\infty}^{(\lambda)}$ of $\omega_p^{(\lambda)}$.
\end{itemize}
Then $\mathcal H$ has bounded degrees because ${\rm deg}_{\mathcal H}(g)\le c$ for every $g\in \Gamma$. Moreover, $\mathcal H$ is a $\Gamma$-equivariant factor of $\mathbf Z$. The fact that $\mathcal H$ has a unique infinite cluster follows as in Claim \ref{cl:unique}. Finally, letting again $N_R^{(\lambda)}(o)$ denote the number of Voronoi cells of $Y^{(\lambda)}$ intersecting $\mathcal B_R(o)$, we have that 
$$
\mathbb E\big[ {\rm deg}_\mathcal{H}(1_G) \big]  \le \mathbb E\big[ c \, \mathbf 1\{N_R^{(\lambda)}(o)\ge n_0\} \big] + \mathbb E\big[ c \, \mathbf 1\{N_R^{(\lambda)}(o)< n_0\} \mathbf 1\{\mathcal B_R(o)\cap \omega_p^{(\lambda)}\} \ne \emptyset \big] \le \eps
$$
by choice of $\lambda$ and $p$.
\end{proof}

\begin{proof}[Proof of Theorem \ref{thm:FIIDSUICP}]

    By Lemma \ref{lm:CocompactLattice}~(2), there exists $d\in\mathbb{N}$ such that all cells in the Voronoi diagram $\{V_\gamma\}_{\gamma\in \Gamma}$ have $d$ neighbors.
    For every $g\in\Gamma$ such that $V_g \cap V_{1_G} \ne \emptyset$, choose a shortest path $\mathcal P(g)$ connecting $1_G$ to $g$ in ${\rm Cay}(\Gamma,S)$ in such a way that $\mathcal P(g^{-1})$ is the inverse path of $\mathcal P(g)$ in the following sense: if $s_1,s_2,\ldots,s_n$ denote the generators associated to the edges traversed in $\mathcal P(g)$, then $s_1^{-1},s_2^{-1},\ldots,s_n^{-1}$ are the generators associated to the edges traversed in $\mathcal P(g^{-1})$. Let $M>0$ be an upper bound on the length of $\mathcal{P}(g)$ over all $g\in\Gamma$ with $V_g \cap V_{1_G} \ne \emptyset$. 

    Let $\epsilon>0$. Let $c:= |\mathcal B_{M}(1_G)|$ be the size of the ball of radius $M$ in ${\rm Cay(\Gamma,S)}$. By Lemma~\ref{lm:sparseLattice}, there exists a $\Gamma$-equivariant factor of iid graph $\mathcal H$ on $\Gamma$ with a unique infinite cluster and $\mathbb E\big[ {\rm deg}_{\mathcal H}(1_G) \big]\le\eps/(2c)$. Define a bond percolation $\omega$ on ${\rm Cay}(\Gamma,S)$ as follows. For every edge $[g,h]$ in $\mathcal{H}$, include the edges of the path $\mathcal{P}(g^{-1}h)$ starting at $g$ (that is, the path $g\mathcal{P}(g^{-1}h$) in the configuration $\omega$. Then $\omega$ has a unique infinite cluster and, as a $\Gamma$-equivariant factor of~$\mathcal H$, may be realized as an FIID process. Finally, note that if $[1_G,g]$ is an edge in $\omega$, then there exists an edge $[h_1,h_2]$ in $\mathcal H$ such that $[1_G,g]\in h_1 \mathcal P(h_1^{-1}h_2)$. Clearly, this implies $h_1,h_2\in\mathcal B_M(1_G)$. Note also that since $h_1 \mathcal P(h_1^{-1}h_2)$ is a shortest path in ${\rm Cay}(\Gamma,S)$, it can add at most $2$ edges to $\omega$ which are incident to $1_G$. Combining these observations, we obtain that
    $$
    \mathbb E \big[{\rm deg}_\omega(1_G)\big] \le 2 \, \mathbb E \bigg[ \sum_{g \in \mathcal B_M(1_G)} {\rm deg}_{\mathcal H}(g) \bigg] = 2\, \big|\mathcal B_M(1_G)\big|\, \mathbb E\big[ {\rm deg}_{\mathcal H}(1_G) \big] \le \eps.
    $$
    The proof of Theorem \ref{thm:FIIDSUICP} is thus complete.
\end{proof}

Let us conclude this section with the following corollary, which shows that the FIID sparse unique infinite cluster property is a group property in the sense that it does not depend on the choice of Cayley graph. It may be proved using a modification of the proof of Theorem \ref{thm:FIIDSUICP} given above.

\begin{corollary}
    Let $\Gamma$ be a finitely generated group. Let ${\rm Cay}(\Gamma,S)$ and ${\rm Cay}(\Gamma,T)$ be Cayley graphs of $\Gamma$ with respect to finite symmetric generating sets $S$ and $T$. If, for every $\eps>0$, there is a $\Gamma$-equivariant FIID bond percolation $\omega_S$ on ${\rm Cay}(\Gamma,S)$ with a unique infinite cluster and $\mathbb E\big[ {\rm deg}_{\omega_S}(1_G) \big]\le\eps$, then also the following holds: for every $\eps>0$ there is a $\Gamma$-equivariant FIID bond percolation $\omega_T$ on ${\rm Cay}(\Gamma,T)$ with a unique infinite cluster and $\mathbb E\big[ {\rm deg}_{\omega_T}(1_G) \big]\le\eps$.
\end{corollary}

\section{Closing remarks} \label{sec:closing}

Theorem \ref{maintheorem} establishes the first examples of symmetric spaces with vanishing uniqueness thresholds. Let us thus conclude by highlighting open questions inspired by this new phenomenon. First, it would be very interesting to know whether it appears in the rank one setting.

\begin{question} Consider real hyperbolic space $\mathbb H^d$, $d\ge3$, equipped with its volume measure. Is it true that $p_u(\lambda)\to0$ as $\lambda\to0$?
\end{question}

In a related direction, convergence of Poisson-Voronoi tessellations to a unique limiting tessellation whose cells pairwise share an unbounded border has been recently established for $\mathcal X:=(\mathbb H^2\times\mathbb H^2, L^1)$, i.e.~the product of two hyperbolic planes equipped with the $L^1$-metric, by D'Achille \cite{D24}.
As noted in \cite{D24}, the space $(\mathbb H^2\times\mathbb H^2, L^2)$ endowed with its natural Riemannian structure is a higher rank symmetric space whose isometry group coincides with the isometry group of $\mathcal{X}$.
Thus, this group does not have property (T).

\begin{question}
    Does $p_u(\lambda)\to0$ as $\lambda\to0$ hold for $\mathcal X$?
\end{question}

The following extension was pointed out to us by Matteo D'Achille (private communication): Does the vanishing uniqueness threshold phenomenon hold for $\mathcal X_p:=(\mathbb H^2\times\mathbb H^2, L^p)$, $p\ge 1$? Since these spaces are quasi-isometric, this question fits into the broader goal to understand the interplay between quasi-isometric metric spaces, their isometry groups and their IPVTs. Due to the fact that \cite{D24} treats $p=1$ in detail, we have singled out this important test case.

Also a related question is whether Theorem \ref{maintheorem} extends to products of rank one symmetric spaces, or even to non-amenable products of non-compact lcsc groups. The relevance of this question stems from the fact that an affirmative answer implies fixed price $1$ by the approach in the present paper. For more information as well as recent progress, we refer to \cite{AM22,M23} and the references therein.

Another direction pertains to finitely generated groups, where we are interested in either Poisson-Voronoi or Bernoulli-Voronoi percolation on the associated Cayley graph \cite{B19,DCELU23}.

\begin{question} \label{q:discrete} Are there non-amenable Cayley graphs such that $p_u(\lambda)\to0$ as $\lambda\to0$? 
\end{question}

An affirmative answer would be especially interesting because it implies an affirmative answer to Question \ref{q:FIIDSparseUnique} {\em directly using} Poisson/Bernoulli-Voronoi percolation. We believe this to be a promising strategy towards finding examples of non-amenable Cayley graphs with the FIID sparse unique infinite cluster property beyond the ones obtained in Theorem \ref{thm:FIIDSUICP}.

Let us also single out the following concrete version of our previous question because of the important connection to the fixed price problem (Question \ref{q:FixedPrice}).

\begin{question} Does $p_u(\lambda)\to0$ as $\lambda\to0$ hold for {\em all} Cayley graphs of 
\begin{itemize}
\item[{\rm1.}] groups with property (T)?
\item[{\rm2.}] non-amenable products?
\end{itemize}
\end{question}

More generally, it remains open whether these classes of groups satisfy the sparse FIID unique infinite cluster property, i.e.~whether Question \ref{q:FIIDSparseUnique} has an affirmative answer for groups in one of these special classes. In this vein, let us mention a consequence of \cite[Theorem 1.2]{MR23}.

\begin{proposition} For every Cayley graph of a countable group with property (T), we have that $\sup_{\lambda} p_u(\lambda) < 1.$
\end{proposition}

This result provides a first hint at interesting behavior of Poisson/Bernoulli-Voronoi percolation on Cayley graphs of groups with property (T). In light of this fact, the following could be a starting point.

\begin{question}
Are there non-amenable Cayley graphs with $0<\inf_\lambda p_u(\lambda)\leq\sup_\lambda p_u(\lambda)<1$?
\end{question}

Finally, several natural questions about Poisson-Voronoi percolation both in the setting explored in this paper as well as in the setting of Cayley graphs remain open. In particular, describing the number of unbounded clusters at $p_u$ would be very interesting.

\begin{question}
    Let $G$ be a connected higher rank semisimple real Lie group and $(X,d_X)$ its symmetric space. Let $\lambda>0$. Is it true that $\omega_{p_u(\lambda)}^{(\lambda)}$ does not have a unique unbounded cluster?
\end{question}

Let us also ask about an analogue of a famous conjecture about Bernoulli percolation due to Benjamini and Schramm \cite{BS96} for Poisson-Voronoi percolation.

\begin{question}
    For every non-amenable Cayley graph, resp.~every symmetric space of a non-compact connected semisimple real Lie group, and $\lambda>0$, do we have that $p_c(\lambda)<p_u(\lambda)$?
\end{question}

Finally, we point out that in the opposite regime ''$\lambda\to\infty$'', which has been considered in the very recent preprint \cite{BDRS25}, Euclidean behavior arises in contrast to the results in this paper.

\end{document}